\documentclass[11pt,a4paper]{amsart}

\usepackage{a4wide}
\usepackage{amsfonts,amsthm,amsmath,amssymb}
\usepackage{amscd,color}
\usepackage{graphicx,subfigure}
\usepackage[shortlabels]{enumitem}

\theoremstyle{plain}
\newtheorem{lem}{Lemma}[section]
\newtheorem{prop}[lem]{Proposition}
\newtheorem{thm}[lem]{Theorem}

\newtheorem*{ThmA}{Theorem A}
\newtheorem*{CorA'}{Corollary A'}
\newtheorem*{ThmB}{Theorem B}
\newtheorem*{ThmC}{Theorem C}

\newtheorem*{Koebe-dist-sph}{Spherical Koebe's Distortion Theorem}
\newtheorem*{Koebe-1/4}{Koebe's One-quarter Theorem}

\theoremstyle{definition}
\newtheorem{defn}[lem]{Definition}
\newtheorem*{defn*}{Definition}
\newtheorem{ex}[lem]{Example}

\newtheorem*{ex*}{Example}

\newtheorem{rem}[lem]{Remark}

\newtheorem*{rem*}{Remark}

\theoremstyle{remark}

\DeclareMathOperator{\diam}{diam}
\DeclareMathOperator{\dist}{dist}
\DeclareMathOperator{\Arg}{Arg}
\DeclareMathOperator{\Sing}{Sing}
\DeclareMathOperator{\Crit}{Crit}
\DeclareMathOperator{\Acc}{Acc}

\DeclareMathOperator{\area}{area}

\DeclareMathOperator{\inter}{int}

\DeclareMathOperator{\conv}{conv}
\DeclareMathOperator{\length}{length}
\DeclareMathOperator{\Id}{Id}
\DeclareMathOperator{\Orb}{Orb}

\newcommand{\C}{\mathbb C}
\newcommand{\D}{\mathbb D}
\newcommand{\clC}{\widehat{\C}}
\newcommand{\chat}{\widehat{\C}}

\newcommand{\R}{\mathbb R}
\newcommand{\Z}{\mathbb Z}
\newcommand{\N}{\mathbb N}

\newcommand{\DD}{\mathcal D}

\newcommand{\FF}{\mathcal F}
\newcommand{\GG}{\mathcal G}
\newcommand{\bd}{\partial}
\renewcommand{\Re}{\textup{Re}}
\renewcommand{\Im}{\textup{Im}}

\newcommand{\calf}{\mathcal{F}}

\newcommand{\PP}{\mathcal{P}}

\newcommand{\rp}{F} 
\newcommand{\ap}{G} 
\newcommand{\f}{f} 

\newcommand{\II}{\mathcal I}
\newcommand{\JJ}{\mathcal J}

\newcommand{\beq}{\begin{equation}}
\newcommand{\eeq}{\end{equation}}

\renewcommand{\r}{\color{red}}

\newcommand{\g}{\color{green}}

\begin{document}

\title[Local connectivity of boundaries of tame Fatou components]{Local connectivity of boundaries of tame Fatou components of meromorphic functions}
\date{\today}

\author{Krzysztof Bara\'nski}
\address{Institute of Mathematics, University of Warsaw,
ul.~Banacha~2, 02-097 Warszawa, Poland}
\email{baranski@mimuw.edu.pl}

\author{N\'uria Fagella}
\address{Departament de Matem\`atiques i Inform\`atica, Institut de Matem\`atiques de la 
Universitat de Barcelona (IMUB) and Barcelona Graduate School of Mathematics (BGSMath), Gran Via 585, 08007 Barcelona, Catalonia, Spain}
\email{nfagella@ub.edu}

\author{Xavier Jarque}
\address{Departament de Matem\`atiques i Inform\`atica, Institut de Matem\`atiques de la 
Universitat de Barcelona (IMUB), and Barcelona Graduate School of Mathematics (BGSMath), Gran Via 585, 08007 Barcelona, Catalonia, Spain}
\email{xavier.jarque@ub.edu}

\author{Bogus{\l}awa Karpi\'nska}
\address{Faculty of Mathematics and Information Science, Warsaw
University of Technology, ul.~Ko\-szy\-ko\-wa~75, 00-662 Warszawa, Poland}
\email{boguslawa.karpinska@pw.edu.pl}

\thanks{The first and fourth authors are supported by the National Science Centre, Poland, grant no 2018/31/B/ST1/02495. 
The second and third authors are partially supported by  grants PID2020-118281GB-C32 and CEX2020-001084-M (Maria de Maeztu Excellence program) from the Spanish state research agency. The third author is additionally supported by  ICREA Acad\`emia 2020 from the Catalan government. }

\bibliographystyle{amsalpha}

\begin{abstract} We prove local connectivity of the boundaries of invariant simply connected attracting basins for a class of transcendental meromorphic maps. The maps within this class need not be geometrically finite or in class $\mathcal B$, and the boundaries of the basins (possibly unbounded) are allowed to contain an infinite number of post-singular values, as well as the essential singularity at infinity. A basic assumption is that the unbounded parts of the basins are contained in regions which we call `repelling petals at infinity', where the map exhibits a kind of `parabolic' behaviour. In particular, our results apply to a wide class of Newton's methods for transcendental entire maps. As an application, we prove the local connectivity of the Julia set of Newton's method for $\sin z$, providing the first non-trivial example of a locally connected Julia set of a transcendental map outside class $\mathcal B$, with an infinite number of unbounded Fatou components.
\end{abstract}

\maketitle

\section{Introduction}\label{sec:intro}

We consider dynamical systems generated by iterates of transcendental meromorphic maps $f\colon \C\to\chat$, where $\chat=\C\cup \{\infty\}$ is the Riemann sphere and $\infty$ is an essential singularity (we also refer to rational maps $f$ defined on $\clC$). An object of particular interest is the {\em Julia set} of $f$, denoted $J(f)$, which is the smallest closed totally invariant subset of $\clC$ with at least three points, in many cases a fractal set with no interior, where the dynamics of $f$ is chaotic  The topology of $J(f)$ can be complicated, which greatly influences the global dynamics of the map. Note that in this paper we assume $\infty \in J(f)$.

The complement of the Julia set, the {\em Fatou set} $\calf(f)=\clC\setminus J(f)$, is the maximal set of normality (equicontinuity) of the iterates of $f$. Its connected components, called {\em Fatou components}, were already  classified by Fatou \cite{fatou} in the 1920's. If an invariant Fatou component $U$ is simply connected, then a Riemann map conjugates $f|_U$ to a holomorphic self-map of $\D$, and this conjugacy extends to the boundary of $U$ if and only if it is locally connected as a subset of $\clC$. Thus, the local connectivity of the boundary of $U$ (in $\clC$) is a key property that allows to understand the dynamics of $f$ on the closure of its Fatou components. Note that in this paper we restrict the analysis to simply connected Fatou components.

By Torhorst's Theorem \cite[Theorem~2.2]{whyburn}, Fatou components of a map with locally connected Julia set have locally connected boundaries. On the other hand, if $U$ is a completely invariant Fatou component, then the local connectivity of its boundary is equivalent to the local connectivity of the whole Julia set. In particular, such a situation occurs when $U$ is a simply connected basin of attraction to infinity for a polynomial. In this intensively studied area, especially for quadratic polynomials, the classical works on the local connectivity, started in the 1980's by Thurston, Sullivan, Douady, Branner and Hubbard, were continued, among others, by Shishikura, Yoccoz, McMullen, Kahn, Lyubich, Levin, Kozlovski, van Strien, Petersen and Zakeri, see e.g.~\cite{sullivan83, orsay1, orsay2, branner-hubbard,  hubbard3T, mcmullen-book, LvS, yoccoz, milnor-loc_con, petersen-zakeri, milnor, KL, KvS} and references therein. Concerning other polynomial Fatou components, all of them bounded, the results of Roesch and Yin \cite{ry1,ry2} show that their boundaries are always Jordan curves (and hence, are locally connected), provided the component is not eventually mapped to a Siegel disc. 

The question of the local connectivity of the boundaries of (simply connected) Fatou components (or of the whole Julia set) is closely related to the location and dynamical behaviour of the {\em singular values} of $f$ (singularities of the inverse map). In the rational case, the singular values coincide with the critical ones, while in the transcendental case, they consist of the critical and asymptotic values and their accumulation points.
Denote the set of singular values of $f$ by $\Sing(f)$ and define the \emph{post-singular set} of $f$ by
\[
\PP(f) = \bigcup_{n=0}^\infty f^n(\Sing(f)).
\]
The simplest class of functions with locally connected Julia sets consists of polynomial or rational maps with connected Julia sets, which are  \emph{hyperbolic}, i.e.~the closure of $\PP(f)$ (in $\clC$) is disjoint from the Julia set. More generally, Tan and Yin \cite{tanlei-loc_con} proved the local connectivity of connected Julia sets of {\em geometrically finite} rational maps, i.e.~maps $f$ for which every critical point in $J(f)$ has a finite orbit (equivalently, every critical point of $f$ is either eventually periodic or attracted to an attracting or parabolic periodic orbit). In general, the problem of local connectivity of polynomial or rational Julia sets is still open and remains a subject of current research. In particular,  local connectivity was showed for a wide class of {\em Newton maps}, i.e.~Newton's root-finding methods for complex polynomials, see e.g.~\cite{Pascale-Newton, DS-Newton, WYZ-Newton}.

The proofs of the local connectivity of the boundaries of simply connected Fatou components, or of the whole Julia sets, are usually technically complicated. Some of them rely on the use of {\em Yoccoz puzzles} built out of external rays and equipotentials, the tools that are quite specific for polynomials and some rational maps. In other cases, the proofs of the local connectivity of the boundary of a Fatou component $U$ consist in building a conformal Riemannian metric near the boundary of $U$, which is expanding, at least along long blocks of trajectories under $f$. In the hyperbolic case, such metric exists on a neighbourhood of $J(f)$ and can be defined as the \emph{hyperbolic metric} on a suitable domain in $\clC \setminus \PP(f)$. The difficulties caused by the presence of preperiodic critical points in the Julia set (\emph{subhyperbolic case}) can be overcome by the use of an \emph{orbifold metric} with a discrete set of singularities (see e.g.~\cite{carlesongamelin,milnor}, cf.~Subsection~\ref{subsec:orbifold}). For geometrically finite maps, a suitable metric is constructed by patching up the hyperbolic or orbifold metric with another suitable metric near parabolic periodic points (see \cite{orsay2, carlesongamelin}). Once such an expanding metric is defined, one can show that a sequence of Jordan curves approximating the boundary of $U$, constructed by taking successive preimages under inverse branches of $f$ near this boundary, converges uniformly, which proves the local connectivity.

The scenario is completely different when we consider transcendental entire maps $f$, which have an essential singularity at $\infty$. In this case the Julia set $J(f)$ cannot be locally connected if $f$ has an unbounded Fatou component, see \cite{BDconn,osborne}. Moreover, if an unbounded invariant Fatou component $U$ of $f$ has locally connected boundary, then $U$ is a Baker domain and $f$ is univalent on $U$, see \cite{bw,kisaka-loc_con}. Roughly speaking, this is due to the fact that $\infty$ is not only an essential singularity but also an omitted value. Hence, preimages of unbounded paths in $U$ are unbounded, and in many cases induce a `comb-like' structure, which prevents the boundary of $U$ and $J(f)$ from being locally connected. A well-know example is given by the exponential map $z\mapsto \lambda e^z$ for $0<\lambda<1/e$, which has an invariant attracting basin of infinite degree with a dense set of accesses to infinity (see e.g.~\cite{devaney-goldberg}); a finite degree example is provided by the map $z\mapsto z e^{z+1}$, which has an unbounded superattracting basin of degree $2$ with the same property, and which conjecturally contains indecomposable continua of escaping points as parts of its boundary. Compare \cite{AnnaJ} for a related example. 

We extend the definition of hyperbolicity to the case of transcendental maps, assuming $\infty \notin \Sing(f)$ and (in the meromorphic case) neglecting the  undefined terms in the definition of the post-singular $\PP(f)$. Note that in particular, the set $\PP(f)$ for hyperbolic transcendental maps is always bounded. We consider also a transcendental analogue of the class of geometrically finite maps, consisting of maps $f$, for which $\Sing(f) \cap \FF(f)$ is compact, while $\overline{\PP(f)} \cap J(f)$ is finite (with closure taken in $\C$). Note that the Fatou set of a transcendental entire geometrically finite map consists of a finite number of basins of attracting or parabolic periodic orbits. This was showed in \cite[Proposition~2.6]{helena-geom_finite} for entire maps, and the proof extends easily to the meromorphic case by \cite{bergweiler,bak02}. Following \cite{ARS}, we say that a geometrically finite transcendental map is \emph{strongly geometrically finite}, if $J(f)$ does not contain asymptotic values of $f$ and the local degree of $f$ at all points of $J(f)$ is uniformly bounded. 

It is known that for hyperbolic transcendental entire maps, the boundary of a bounded Fatou component is a Jordan curve and even a quasicircle (see \cite{morosawa-loc_con, bfr}). Note that for entire maps, locally connected boundaries of bounded simply connected Fatou components are always Jordan curves by the maximum principle. In \cite{ARS}, it was proved that if a transcendental entire maps $f$ is strongly geometrically finite, then all bounded periodic Fatou components are Jordan curves. The same holds for all Fatou components of $f$ if, additionally, every Fatou component contains at most finitely many critical points. 

The local connectivity of the whole Julia set was proved for hyperbolic and, more generally, strongly geometrically finite transcendental entire maps $f$ with only bounded Fatou components and no asymptotic values, satisfying a uniform bound on the number of critical points (with multiplicities) contained in each of the Fatou components of $f$, see~\cite{morosawa-loc_con,bm,bfr,ARS}. Note that all the results mentioned above consider only transcendental entire maps from the Eremenko--Lyubich class $\mathcal B$, where
\[
\mathcal B= \{f: \Sing(f) \text{ is bounded}\}.
\]

Paradoxically, the situation changes when we take into account a priori more complicated class of transformations, given by transcendental meromorphic maps, for which the essential singularity at infinity is not an omitted value. A trivial example occurs within the tangent family $z\mapsto \lambda \tan z$ for $\lambda >1$, where the Fatou set consists of two completely invariant attracting basins (the upper and lower half plane), and hence the Julia set is equal to the real line together with the point at infinity (i.e.~a circle in~$\clC$). Numerical simulations suggest that also for many meromorphic maps with more complicated dynamics, their unbounded Fatou components may have locally connected boundaries, despite containing the essential singularity at infinity. In some examples, it looks plausible that the whole Julia set has the same property (see the picture on the left side of Figure~\ref{fig:yes_and_no}). Nevertheless, computer pictures indicate that non-locally connected boundaries do exist also for meromorphic maps, similarly to the entire case (see the picture on the right side of Figure~\ref{fig:yes_and_no}). 

\begin{figure}[ht!]
\centerline{
\includegraphics[width=0.45\textwidth]{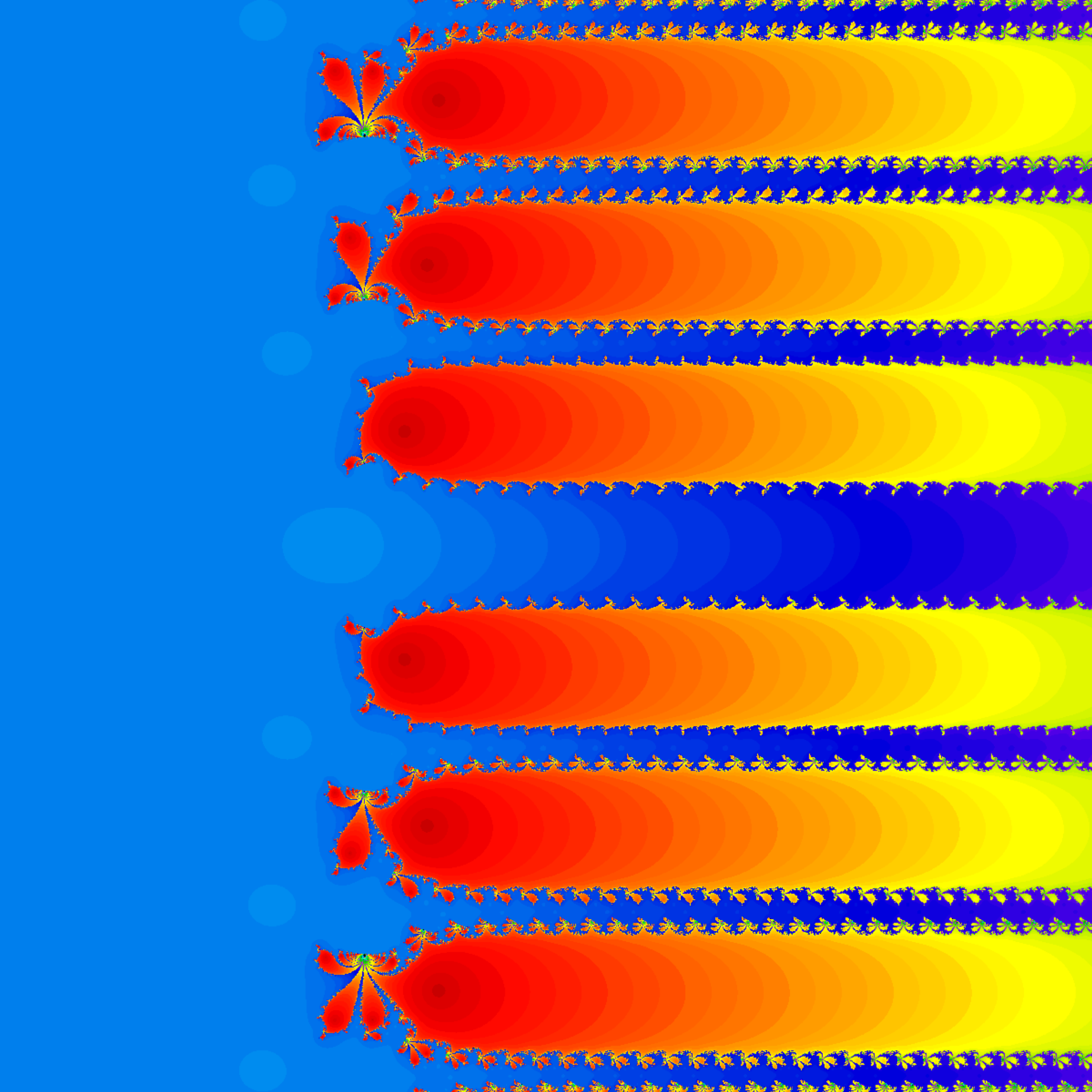} \hfil
\includegraphics[width=0.45\textwidth]{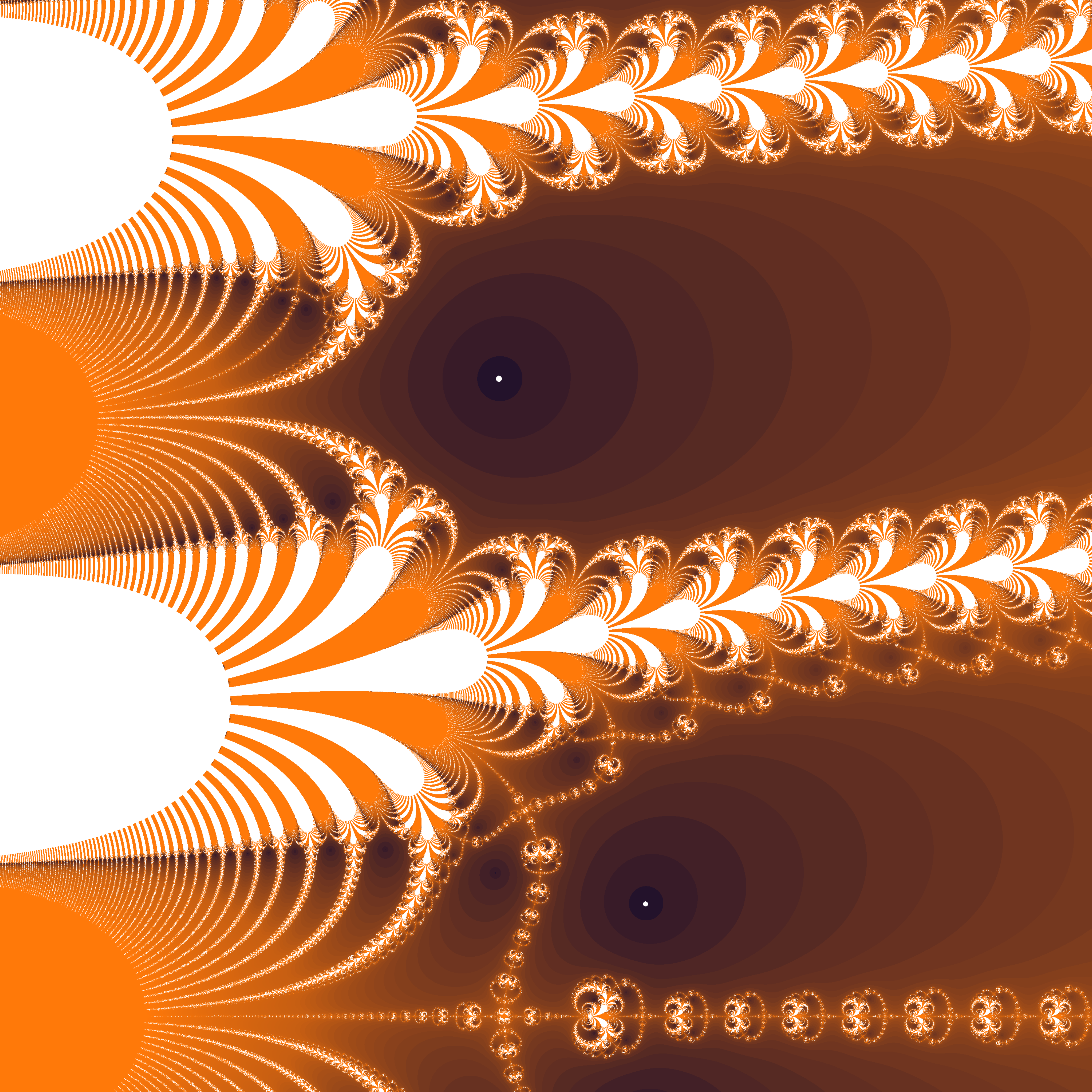}
}
\caption{\small Left: The dynamical plane of $f(z)=\frac{z-1}{1 + e^{-z}}$, Newton's method for $g(z)=z+e^z$. Right: The dynamical plane of the map $f(z) = \frac{z^2-e^{-z}}{1+z}$, Newton's method for $g(z) = 1 + z e^z$.}\label{fig:yes_and_no}
\end{figure}

In this paper we show that for many transcendental meromorphic maps $f$ outside class $\mathcal B$, even with asymptotic values, the presence of (possibly infinitely many) post-singular orbits and the essential singularity in the boundary of a simply connected invariant Fatou component $U$, pose no unsolvable obstacle for local connectivity, as long as $f$ acts `geometrically finitely' on a compact part of the closure $\overline{U}$ in $\C$, and the (possible) unbounded parts of $\overline{U}$ are contained in a finite number of regions, where $f$ is univalent and exhibits a `tame' dynamical behaviour, similar to the one within a repelling petal at a parabolic fixed point. We call these regions {\em repelling petals at infinity}, and their formal definition is given in Section~\ref{sec:petals_at_inf} (see Definition~\ref{defn:petal_at_inf}). Although the definition allows for a quite general behaviour (e.g.~spiralling petals), its simple model is given by Newton' method for the map $z \mapsto z+e^z$, which behaves like the translation $z\mapsto z-1$ for $\Re(z)\to+\infty$ in any sector symmetric with respect to $\R^+$ of angle less than $\pi$ (see the picture on the left side of Figure~\ref{fig:yes_and_no}). Note that unlike petals at parabolic fixed points, repelling petals at infinity may contain points from both the Fatou and Julia set.

In this work we assume that $U$ is an invariant attracting basin, leaving other cases for a forthcoming paper. Note that for transcendental meromorphic maps, periodic components of period larger than $1$ require a separate treatment, since considering an iterate of the map takes us beyond the meromorphic class. 

\begin{defn}\label{defn:tame}
An invariant attracting basin $U$ of a transcendental meromorphic map $f$ is \emph{tame at infinity}, if there exists a disc $D \subset \C$ such that $\overline{U} \setminus D$ is contained in the union of a finite number of repelling petals $P_i$ at infinity for $f$, such that $U \cap f(P_i) \cap P_{i'} = \emptyset$ for $i \ne i'$.
\end{defn}

To formulate our result, define the {\em post-critical} and {\em post-asymptotic set} as 
\begin{align*}
\PP_{crit}(f) &= \{f^n(v) : v \text{ is a critical value of } f, \, n \ge 0\},\\
\PP_{asym}(f) &= \{f^n(v) : v \text{ is an asymptotic value of } f, \, n \ge 0\}
\end{align*}
and write $\overline A$, $\bd A$ for the closure and boundary in $\C$ of a set $A \subset \C$, and $\Acc (A)$ for the set of its accumulation points. We also denote the local degree of a map $f$ at a point $w$ by $\deg_w f$. Our main result is the following.

\begin{ThmA} 
Let $U$ be a simply connected invariant attracting basin of a meromorphic map $f\colon \C \to \clC$. Assume that the following conditions are satisfied.
\begin{enumerate}[$($a$)$]

\item The set $\big(\overline{\PP_{asym}(f)} \cup \Acc(\PP_{crit}(f))\big) \cap \overline{U}$ is contained in the union of a compact subset of $U$ and a finite set of parabolic periodic points of $f$ in $\bd U$.

\item There exists a compact set $L \subset U$ such that for every $z \in \PP_{crit}(\f) \cap \overline{U} \setminus L$ we have $\sup\{\deg_w \f^n: w \in \f^{-n}(z), \, n > 0\} < \infty$.

\item $U$ is tame at infinity.
\end{enumerate}
Then the boundary of $U$ in $\clC$ is locally connected.
\end{ThmA}

Theorem~A implies the following corollary.

\begin{CorA'}
Let $U$ be a simply connected invariant attracting basin of a strongly geometrically finite meromorphic map, such that $U$ is tame at infinity. Then the boundary of $U$ in $\clC$ is locally connected.
\end{CorA'}

Indeed, if $f$ is strongly geometrically finite, then $\overline{\PP(f)} \cap \bd U$ consist of at most finitely many points, all of them from $\PP_{crit}(f)$. Moreover, $\Sing(f)$ intersects only a finite number of Fatou components, which implies that $\overline{\PP(f)} \cap U$ is compact, as a finite union of holomorphic images of compact sets. These facts together with the definition of strongly geometrically finite maps immediately imply the conditions (a)--(b) of Theorem~A. 

\begin{rem}
Note that in Definition~\ref{defn:tame} we do not assume that the set of repelling petals of $f$ at infinity intersecting $U$ is non-empty. Therefore, the result holds also for all bounded simply connected invariant attracting basins of transcendental entire or meromorphic maps satisfying the conditions (a)--(b). In particular, for bounded periodic attracting basins of transcendental entire maps we obtain a generalization of the mentioned above result from \cite{ARS}. However, the case of unbounded basins is our primary area of interest.
\end{rem}

We remark that an attracting basin $U$ satisfying the hypotheses of Theorem~A necessarily fulfils some properties, as described in the following proposition.

\begin{prop} \label{prop:extrabonus}
Under the hypotheses of Theorem~A, the following hold.
\begin{enumerate}[$($a$)$] 
\item The degree of $f$ on $U$ is finite.
\item $\overline{U}$ contains only a finite number of critical points of $f$.
\item Every post-critical point of $f$ in $\partial U$ has a finite orbit.
\item Every asymptotic curve of an asymptotic value $v \in \overline{U}$ is eventually contained in $\C \setminus \overline{U}$.
\item $f$ maps $\overline{U}$ onto the closure of $U$ in $\clC$. In particular, $f$ has a pole in $\partial U$.
\end{enumerate}
\end{prop}

Following the classical ideas explained above, which were used for proving local connectivity for several classes of rational and entire maps (see \cite{orsay2, carlesongamelin, tanlei-loc_con,ARS}), the proof of Theorem~A is based on a construction of a conformal metric $d\varsigma$ with suitable expanding properties on a part of $U$ close to $\bd U$. This is described in the following theorem.

\begin{ThmB}\label{thm:metric_adapted}
Let $U$ be an invariant simply connected attracting basin of a meromorphic map $f\colon \C \to \clC$ satisfying the assumptions of Theorem~A. Then there exists a simply connected domain $A\subset U$ with $\overline{A} \subset U$, such that for every compact set $K \subset U$ one can find a conformal metric $d\varsigma = \varsigma|dz|$ on $U\setminus \overline{A}$ and numbers $b_n \in (0,1)$, $n \in \N$, such that $\sum_{n=1}^\infty b_n < \infty$ and 
\[
|(f^{n})'(z)|_\varsigma > \frac{1}{b_n}
\]
for every $z \in U\setminus \overline{A}$ and $n \in \N$ with $f(z), \ldots, f^{n-1}(z)\in U\setminus \overline{A}$, $f^n(z) \in K \setminus \overline{A}$.
\end{ThmB}

The domain $A$ in Theorem~B is defined as a sufficiently large absorbing domain in $U$, such that $\overline{f(A)} \subset A$. Like in the references mentioned above, the metric $d\varsigma$ is constructed by `patching up' the orbifold metric on a suitable part of $U$ with a `parabolic' metric $\frac{|dz|}{|z-p|^{\alpha_{par}}}$ for some $\alpha_{par} \in [0,1)$ on repelling petals of parabolic points $p \in \bd U$. In our setting, however, we need to add a new element to this puzzle, namely a {\em petal metric} which we use in the unbounded parts of $U$. To this aim, we  prove that the map on a repelling petal at infinity has suitable expanding properties with respect to a metric $\frac{|dz|}{|z|^{\alpha_\infty}}$ for some $\alpha_\infty > 1$. This result (Theorem~\ref{thm:der-inf}) can be of independent interest and we hope it may be used in a much wider setting. The metric $d\varsigma$ on the unbounded parts of $U$ is defined as a suitable modification of the metric $\frac{|dz|}{|z|^{\alpha_\infty}}$. 

Theorem~B follows from a much more general result, Theorem~\ref{thm:metric}, formulated in a more abstract setting. The reason for this generality, which undoubtedly increases the technical difficulties of the proof, is the goal of further applications to other types of Fatou components (parabolic basins and Baker domains). 

Theorem~A has already a number of applications, mostly among transcendental Newton maps (for which all the Fatou components are simply connected, as proved in \cite{bfjk}). For example, most of Newton's methods for trigonometric polynomials studied in \cite{bfjkaccesses} have infinitely many unbounded attracting basins satisfying the hypotheses of Theorem~A.

We believe that in the setting of meromorphic maps, where infinity is no longer an omitted value, the local connectivity of Julia sets is a much more common phenomenon, even in the presence of unbounded Fatou components, as long as they satisfy the hypotheses of Theorem~A. To show some evidence for this statement, we present an example of a transcendental meromorphic map with infinitely many unbounded basins of attraction, whose Julia set is locally connected (see Figure~\ref{dynplanesine}). 
\begin{figure}[ht!]
\centerline{
\includegraphics[width=0.45\textwidth]{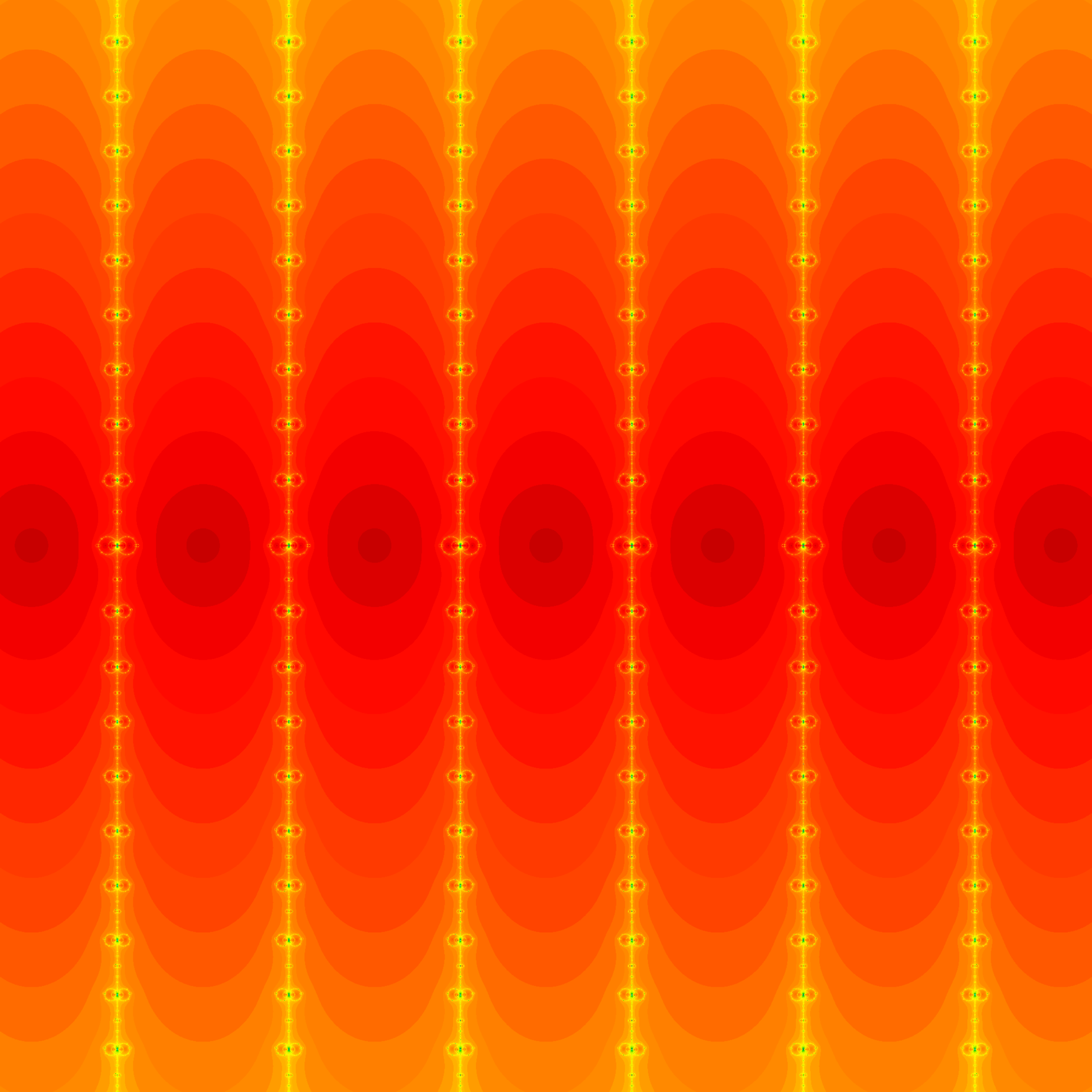}
\hfil
\includegraphics[width=0.45\textwidth]{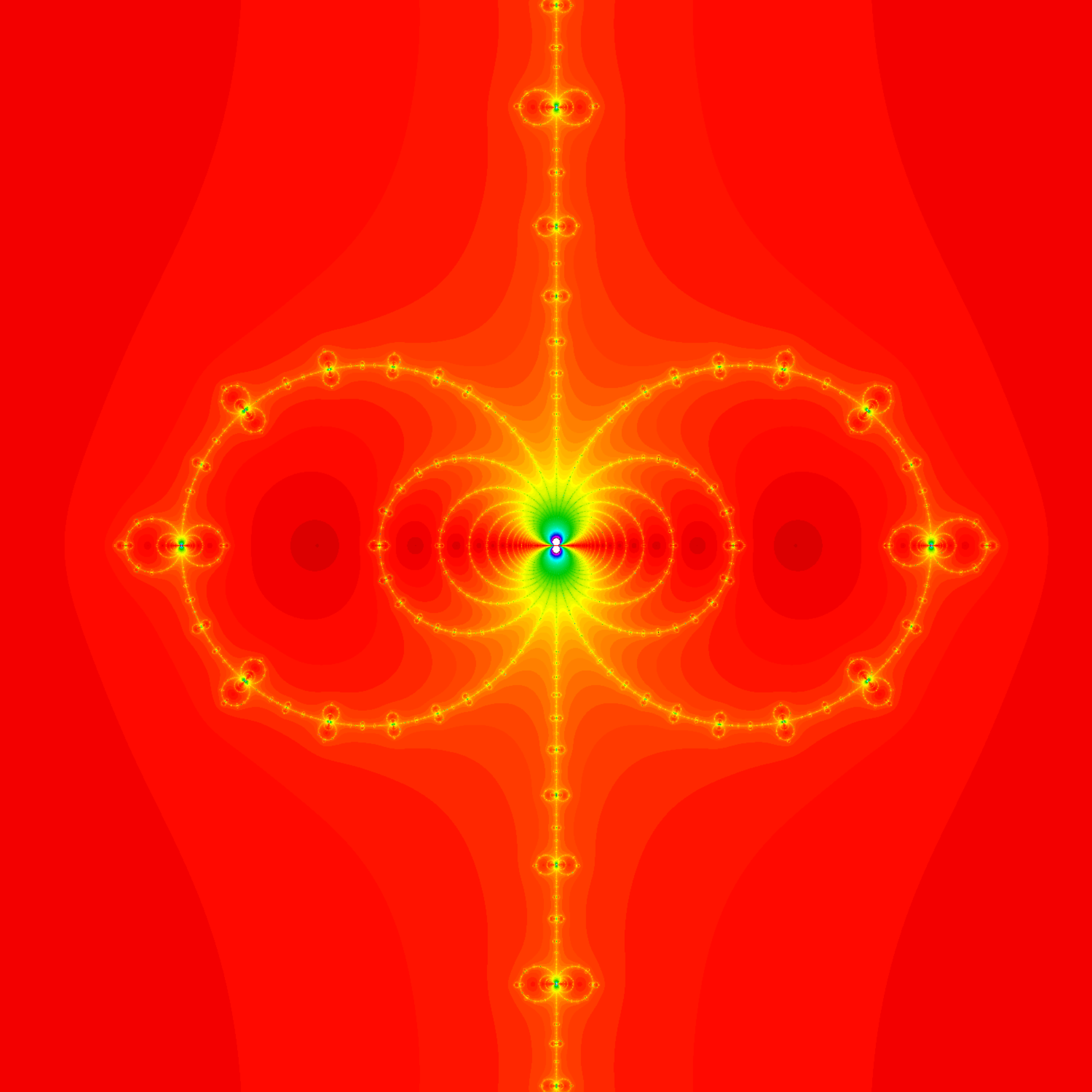}
}
\caption{\label{dynplanesine} \small Left: The dynamical plane of the map $f(z)=z-\tan z$, showing the invariant attracting basins $U_k$, $k \in \Z$. Right: A zoom of the dynamical plane near a pole $p_k$.}
\end{figure}

\begin{ThmC}
\label{thm:sin}
Let $f(z)=z-\tan z$, Newton's method for $g(z)=\sin z$. Then $J(f)$ is locally connected. 
\end{ThmC}

This provides the first non-trivial example of a locally connected Julia set of a transcendental meromorphic map $f$ outside class $\mathcal B$, with an infinite number of unbounded Fatou components.

The structure of the paper is as follows. After preliminaries in Section~\ref{sec:prelim} and a description of the dynamics of $f$ on attracting and repelling petals of parabolic periodic points, presented in Section~\ref{sec:petals_at_par}, in the subsequent Section~\ref{sec:petals_at_inf} we define attracting/repelling petals of $f$ at infinity (Definition~\ref{defn:petal_at_inf}) and prove their contracting/expanding properties (Theorem~\ref{thm:der-inf}). The metric $d\varsigma$ is constructed in Section~\ref{sec:metric} (Theorem~\ref{thm:metric}). Since the construction is quite involved, we split it into several parts, presented in Subsections~\ref{subsec:metric_petal_dyn}--\ref{subsec:metric_metric}, providing a short summary of the proof at the beginning of the section. Proposition~\ref{prop:extrabonus} and Theorem~B are proved in Section~\ref{sec:proofB}, while the proof of Theorem~A is presented in Section~\ref{sec:proofA}. Finally, in Section~\ref{sec:proofC} we prove Theorem~C.

\section{Preliminaries}\label{sec:prelim} 

\subsection{Notation} \label{subsec:notat}

By $\overline A$, $\inter A$ and $\bd A$ we denote, respectively, the closure, interior and boundary in $\C$ of a set $A \subset \C$. By $\conv A$ we denote the convex hull of a set $A$. For $A \subset \C$ and $z \in \C$ we write $A+z = \{a+z : a \in A\}$. We write $\Acc (A)$ for the set of accumulation points of $A$ in $\C$. 
By $\clC = \C \cup \{\infty\}$ we denote the Riemann sphere with the standard topology. We write $\D(z,r)$ for the Euclidean disc in $\C$ of radius $r$ and center $z$, while $\D$ is the open unit disc in $\C$. The Euclidean diameter of a set $A \subset \C$ is denoted by $\diam A$, and area of a measurable set $A \subset \C$ by $\area A$. We set $\N = \{1, 2, \ldots\}$. 

Let $F \colon V' \to V$ be a meromorphic map on a domain $V' \subset \C$ into a set $V \subset \clC$. For $z \in V'$ we set $\Orb(z) = \{F^n(z): n \ge 0\}$, neglecting the cases when $F^n$ is not defined.  By $\Crit(F)$ we denote the set of critical points of $F$ (we do not treat multiple poles as critical points). The points $F(z)$ for $z \in \Crit(F)$ are \emph{critical values} of $F$. A point $v \in V$ is an \emph{asymptotic value} of $F$ if 
there exists a curve $\gamma\colon [0, +\infty) \to V'$ such that $ \gamma(t)\xrightarrow[t \to +\infty]{} \bd V'$ and $F(\gamma(t))\xrightarrow[t \to +\infty]{} v$.

We denote by $\Sing(F)$ the \emph{singular set} of $F$, i.e.~the set of finite singularities of the inverse function $F^{-1}$ (the critical and asymptotic values of $F$ and their accumulation points in $V$). The {\em post-singular set} of $F$ is defined as
\[
\PP(F) = \bigcup_{n=0}^\infty F^n(\Sing(F)),
\]
neglecting the cases when $F^n$ is not defined. We also define the {\em post-critical} and {\em post-asymptotic set} of $F$ as 
\begin{align*}
\PP_{crit}(F) &= \{F^n(v) : v \text{ is a critical value of } F, \, n \ge 0\},\\
\PP_{asym}(F) &= \{F^n(v) : v \text{ is an asymptotic value of } F, \, n \ge 0\},
\end{align*}
again neglecting the cases when $F^n$ is not defined.

For $z_0 \in V'$ we denote by $\deg_{z_0} F$ the local degree of $F$ at $z_0$, i.e.~the positive integer $d$ such that $F(z) = F(z_0) + a(z-z_0)^d + \cdots$ for $a \neq 0$ if $F(z_0) \in \C$, and $F(z) = a(z-z_0)^{-d} + \cdots$ for $a \neq 0$ if $F(z_0) = \infty$. 

\subsection{Conformal metrics}\label{subsec:conf_metric}

By a \emph{conformal metric} on an open set $V \subset \C$ we mean a Riemannian metric of the form $d\rho(z) = \rho(z)|dz|$ for a positive continuous function $\rho$ on $V$, where $|dz|$ denotes the standard (Euclidean) metric in $\C$. On each component $\tilde V$ of $V$, the distance between points $z_1, z_2$ with respect to this metric, denoted $\dist_\rho(z_1, z_2)$, is defined as the infimum of the lengths of piecewise $C^1$-curves $\gamma$ joining $z_1$ and $z_2$ within this component, counted with respect to the metric $d\rho$ and denoted by $\length_\rho (\gamma)$, where
\[
\length_\rho \gamma = \int_\gamma \rho(z)|dz|.
\]
The diameter of a set $A \subset \tilde V$ with respect to $d\rho$ is defined as
\[
\diam_\rho A = \sup\{\dist_\rho(z_1, z_2): z_1, z_2 \in A\},
\]
while $\DD_\rho(z,r)$ denotes the disc of center $z\in \tilde V$ and radius $r > 0$ with respect to $d\rho$. The area of a measurable set $A \subset V$ with respect to $d\rho$ is denoted by
\[
\area_\rho A = \int_A (\rho(z))^2|dz|^2.
\]

If $V \subset \C$ is a \emph{hyperbolic domain} (i.e.~a domain such that $\C \setminus V$ contains at least two points), then we denote by $d\varrho_V$ the \emph{hyperbolic metric} in $V$ (see e.g.~\cite{carlesongamelin}).

The standard \emph{spherical metric} is defined as
\[
d\sigma_{sph} = \sigma_{sph}(z) |dz| = \frac{2}{1 + |z|^2}|dz|.
\]
for $z \in \C$. Note that $d\sigma_{sph}$ extends to a Riemannian metric on the Riemann sphere $\clC$ by the use of the coordinates $z \mapsto 1/z$ near infinity. For simplicity, we write $\dist_{sph}$, $\diam_{sph}$, $\area_{sph}$ and $\DD_{sph}(z,r)$ for the spherical distance, diameter, area and disc, respectively.

The derivative of a holomorphic map $F$ with respect to the metric $d\rho$ on $V$ is equal to
\[
|F'(z)|_\rho = \frac{\rho(F(z))}{\rho(z)}|F'(z)|,
\]
provided $F$ is defined in a neighbourhood of a point $z \in V \cap \rp^{-1}(V)$. For the spherical metric we use the symbol $|F'(z)|_{sph}$ instead of $|F'(z)|_{\sigma_{sph}}$.

We say that a holomorphic map $F$ is \emph{locally contracting} (resp.~\emph{locally expanding}) with respect to the metric $d\rho$ on a set $V' \subset V \cap \rp^{-1}(V)$ if $|F'(z)|_\rho < 1$ (resp.~$|F'(z)|_\rho > 1$) for every $z \in V'$. We also say that in this case the metric $d\rho$ is locally contracting/expanding with respect to $F$.

We will use the following version of the Koebe distortion theorem for the spherical metric (for the proof see \cite[p.~1170]{bowen}). 

\begin{thm}[\bf{Spherical Koebe distortion theorem}{}]\label{thm:sph-Koebe}
Let $0 < r_1, r_2 < \diam_{sph} \clC$. Then there exists $c>0$ depending only on $r_1, r_2$, such that for every spherical disc $D = \DD_{sph}(z, r)$ and every univalent holomorphic map $F\colon D \to \clC$ with $z \in \clC$, $r > 0$, $\diam_{sph} D < r_1$ and $\diam_{sph}(\clC\setminus F(D))>r_2$, if $z_1, z_2 \in \DD_{sph}(z, \lambda r)$ for some $0 < \lambda < 1$, then 
\[
\frac{|F'(z_1)|_{sph}}{|F'(z_2)|_{sph}} \leq\frac{c}{(1 - \lambda)^4}.
\]
\end{thm}

\subsection{Orbifolds}\label{subsec:orbifold}

We recall some facts about hyperbolic orbifolds (for details, see \cite[\S 19]{milnor}). Let $V \subset \C$ be a hyperbolic domain. A \emph{hyperbolic orbifold} over $V$ is a pair $(V, \nu)$, where $\nu\colon V \to \N$ is a function such that the set of points $z \in V$ with $\nu(z) > 1$ has no accumulation points in $V$. Every orbifold $(V, \nu)$ has a universal branched covering, i.e.~a holomorphic map $\pi\colon \D \to V$ onto $V$, such that 
\begin{equation}\label{eq:multiple-orbifold}
\deg_u \pi = \nu(\pi(u)) \qquad\text{for } u \in \D.
\end{equation}
(see e.g.~\cite[Theorem~E.1]{milnor}). This implies that the Riemannian metric 
\[
d\rho = \rho |dz| = \pi_*(d\varrho_{\D}),
\]
which is the push-forward under $\pi$ of the hyperbolic metric $d\varrho_{\D}$ on $\D$,
is well-defined on $V \setminus \{z : \nu(z) > 1\}$ (if $\nu(z) = 1$, then different points in $\pi^{-1}(z)$ are related via automorphisms of $\D$, so the metric is independent of the choice of a point in the fiber). The metric $d\rho$ is called the {\em orbifold metric}. 
Note that if $\nu(z) = 1$ for every $z \in V$, then the map $\pi$ is a universal covering of $V$ and $d\rho$ is equal to the hyperbolic metric $d\varrho_V$ on $V$. 

If $\nu(z_0) > 1$ for some $z_0 \in V$, then $d\rho$ has a singularity at $z_0$. More precisely, there exist constants $c_1, c_2 > 0$ such that 
\begin{equation}\label{eq:orb_metric}
\frac{c_1}{|z-z_0|^{1 - 1/\nu(z_0)}} < \rho(z) < \frac{c_2}{|z-z_0|^{1 - 1/\nu(z_0)}} 
\end{equation}
for $z$ in some punctured neighbourhood of $z_0$ (see \cite[p.183, Appendix~A]{mcmullen-book}). Note also that we have
\begin{equation}\label{eq:orb>hyp}
\rho \ge \varrho_V \qquad \text{on} \quad V \setminus \{z : \nu(z) > 1\}.
\end{equation}
This holds due to the fact that the identity map $V \to V$ defines a holomorphic orbifold map\footnote{A \emph{holomorphic orbifold map} $\phi$ between orbifolds $(V, \nu)$ and $(\tilde V, \tilde\nu)$ is a holomorphic map $\phi\colon V \to \tilde V$, for which $\nu(z)\deg_z \phi$ is divisible by $\tilde{\nu}(\phi(z))$ for $z \in V$.} from $(V, \nu)$ to the orbifold $(V, \tilde \nu)$ with $\tilde\nu \equiv 1$ and the orbifold metric equal to $d\varrho_V$, so \eqref{eq:orb>hyp} follows from the Schwarz--Pick orbifold lemma (see e.g.~\cite[Theorem~A.3]{mcmullen-book}).

\subsection{Logarithmically convex functions}\label{subsec:log_conv}

Recall that a function $g\colon (a,b) \to \R_+$, for $a, b \in \R\cup\{\pm \infty\}$, is \emph{logarithmically convex}, if $\ln g$ is convex. We will use the following facts on non-increasing logarithmically convex functions.

\begin{lem}\label{lem:log_conv}
Let $g\colon (t_0, +\infty) \to \R_+$ for some $t_0 \in \R$ be a non-increasing logarithmically convex function. Define inductively a sequence $t_n \in (t_0, +\infty)$, $n \in \N$, by choosing some $t_1 \in (t_0, +\infty)$ and setting 
\[
t_{n+1} = t_n + g(t_n)
\]
for $n \in \N$. Then:
\begin{enumerate}[$($a$)$]
\item the sequence $(t_n)_{n=1}^\infty$ is increasing and converges to $+\infty$ as $n \to \infty$, 
\item the sequence $(\frac{t_{n+1}}{t_n})_{n=1}^\infty$ is decreasing for sufficiently large $n$ and converges to $1$ as $n \to \infty$, 
\item the sequence $\big(\frac{g(t_{n+1})}{g(t_n)}\big)_{n=1}^\infty$ is non-decreasing and converges to $1$ as $n \to \infty$.
\end{enumerate}
\end{lem}
\begin{proof} Since $g$ is positive, we have $t_{n+1} > t_n$. Note also that $g$ is convex and hence continuous on $(t_0, +\infty)$. This implies $t_n \to +\infty$ as $n \to \infty$, because otherwise $t_n \to \overline t$ for some $\overline t \in \R$ and $g(\overline t) = \lim_{n\to\infty} g(t_n) = \lim_{n\to\infty} (t_{n+1} - t_n) = 0$, which contradicts the fact that $g$ is positive. This shows~(a). 

As $t_n$ increases to $+\infty$, we have $t_n >0$ for sufficiently large $n$. For such $n$, since $g$ is positive and non-increasing, the sequence $\frac{g(t_n)}{t_n}$ is decreasing, so the sequence
\[
\frac{t_{n+1}}{t_n} = 1 + \frac{g(t_n)}{t_n}
\]
is decreasing. Moreover, the sequence $g(t_n)$ is bounded, since $g$ is positive and non-increasing, while $t_n \to +\infty$ by~(a). Hence, $\frac{g(t_n)}{t_n} \to 0$, so $\frac{t_{n+1}}{t_n} \to 1$, which ends the proof of~(b). 

To show (c), note that as $g$ is non-increasing, we have
\[
t_{n+1} - t_n = g(t_n) \le g(t_{n-1}) = t_n - t_{n-1},
\]
which gives $\frac{t_{n-1} + t_{n+1}}{2} \le t_n$. Consequently, setting $h = \ln g$, we obtain
\[
\frac{h(t_{n-1}) + h(t_{n+1})}{2} \ge h\Big(\frac{t_{n-1} + t_{n+1}}{2}\Big) \ge h(t_n),
\]
since $h$ is convex and non-increasing. This implies 
\[
h(t_{n+1}) - h(t_n) \ge h(t_n) - h(t_{n-1}),
\]
so the sequence
\[
\frac{g(t_{n+1})}{g(t_n)} = e^{h(t_{n+1}) - h(t_n)}
\]
is non-decreasing. As $\frac{g(t_{n+1})}{g(t_n)} \le 1$, as remarked above, it follows that $\frac{g(t_{n+1})}{g(t_n)} \to q$ for some $0 < q \le 1$. If $q < 1$, then for large $n$ we have $\frac{g(t_{n+1})}{g(t_n)} < q'$ for some constant $q' < 1$, so $\sum_{n=1}^\infty(t_{n+1} - t_n) = \sum_{n=1}^\infty g(t_n) < \infty$ and, consequently, $t_n \to \overline t$ for some $\overline t \in \R$, which is impossible by (a). Hence, 
\[
\frac{g(t_{n+1})}{g(t_n)} \to 1.
\]
This ends the proof of (c). 
\end{proof}


\section{Attracting and repelling petals at parabolic periodic points} \label{sec:petals_at_par}

\begin{defn}[{\bf Attracting/repelling petal at a parabolic fixed point}{}]\label{defn:par_petal}
Let $p \in \C$ be a parabolic fixed point of multiplier $1$ and order $d \in \N$ of a holomorphic map $\ap$ defined near $p$. Then $\ap$ has the form 
\begin{equation*}\label{eq:parabolic_map}
\ap(z) = z + a (z-p)^{d+1} +\cdots
\end{equation*}
for $z$ near $p$, where $a \in \C \setminus \{0\}$. By an \emph{attracting petal} of $\ap$ at $p$ we mean a simply connected domain $P$ contained in a small neighbourhood of $p$, such that $p \in \bd P$, $\overline{\ap(P)} \subset P \cup \{p\}$ and $\bigcap_{n=0}^\infty \ap^n(P) = \emptyset$. 

A simply connected domain $P \subset \C$ is a \emph{repelling petal} of a holomorphic map $\rp$ at its parabolic fixed point $p$ of multiplier $1$, if $\rp(P)$ is an attracting petal at $p$ of a branch $\ap$ of $\rp^{-1}$ with $\ap(p) = p$ (which is well-defined near $p$).\footnote{ There are several variants of definitions of attracting and repelling petals at parabolic fixed points, which differ in details (see e.g.~\cite[Definition 10.6]{milnor} and the discussion afterwards).}
\end{defn}
 
For $p \in \C$, $d \in \N$, $a \in \C \setminus \{0\}$, $\varepsilon, \delta > 0$ and $j \in \{0, \ldots, 2d-1\}$ let
\[
U_j (\varepsilon, \delta) =\left\{z \in \C \setminus \{p\}: \Arg(z - p) \in \left(\theta_j - \delta, \theta_j + \delta\right),\, |z- p| < \varepsilon\right\},
\]
where
\[
\theta_j = \frac{-\Arg(a)}{d} + \frac{\pi j}{d} \qquad \text{for odd (resp.~even) }j \in \Z.
\]
The facts described in the following proposition are well-known, see e.g.~\cite[Chapter~II.5]{carlesongamelin}, \cite[\S 10]{milnor}.

\begin{prop}\label{prop:par_petal}
Let $p \in \C$ be a parabolic fixed point of multiplier $1$ and order $d \in \N$ of a holomorphic map $\f$ defined near $p$.
\begin{enumerate}[$($a$)$]
\item For every $\varepsilon, \delta > 0$ and a compact set $K$ contained in an attracting $($resp.~repelling$)$ petal of $\f$ at $p$, there exist an odd $($resp.~even$)$ integer $j \in \{0, \ldots, 2d-1\}$ and $n_0 \in \N$ such that $\f^n(K) \subset U_j (\varepsilon, \delta)$ $($resp.~$(\f^{-1})^n(K) \subset U_j (\varepsilon, \delta))$ for every $n \ge n_0$.
\item There exist $d$ attracting $($resp.~repelling$)$ petals $P_j$ with Jordan boundaries, of the map $\f$ at $p$, for odd $($resp.~even$)$ integers $j \in \{0, \ldots, 2d-1\}$, such that $P_j$ $($resp.~$\f(P_j))$ are pairwise disjoint, and for every $\delta \in (0, \frac{\pi}{d})$ one can find $\varepsilon > 0$ with $U_j (\varepsilon, \delta) \subset f(P_j) \subset P_j \subset U_j (\varepsilon, \frac{\pi}{d})$ $($resp.~$U_j (\varepsilon, \delta) \subset P_j \subset  f(P_j) \subset U_j (\varepsilon, \frac{\pi}{d}))$.
\end{enumerate}
\end{prop}

See Figure~\ref{fig:petals}. 
\begin{figure}[ht!]
\includegraphics[width=0.4\textwidth]{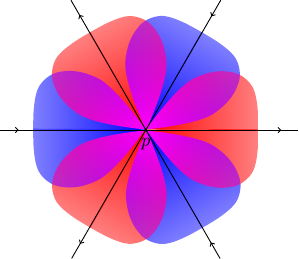}
\label{fig:petals}
\caption{\small Attracting and repelling petals of a holomorphic map at a parabolic fixed point $p$, with $d=3$ and $a=1$.}
\end{figure}

\begin{prop}\label{prop:par_petal2}
Let $P$ be an attracting petal of a holomorphic map $\ap$ at a parabolic fixed point $p$ of multiplier $1$ and order $d$. Then the following hold.
\begin{enumerate}[$($a$)$]
\item $\ap$ is univalent on $P$.
\item $\ap^n(z) \to p$ as $n \to \infty$ for $z \in P$.
\item 
For every compact set $K \subset P$ and $\delta > 0$ there exist $c_1, c_2 > 0$ and $n_0 \ge 0$, such that for every $z \in \bigcup_{n=n_0}^\infty \ap^n(K)$,
\[
c_1 |z-p|^{d+1} < |\ap(z) - z| < c_2|z-p|^{d+1}, \quad 
|\Arg(\ap(z) - z) - \Arg(p - z)| < \delta.
\]
\end{enumerate}

\end{prop}
\begin{proof} The assertions (a)--(b) follow directly from the definition of an attracting petal at a parabolic fixed point and the Schwarz--Pick lemma. To show (c), note that by Proposition~\ref{prop:par_petal}(c), for sufficiently large $n_0$ we have $\ap^n(K) \subset U_j (\varepsilon, \delta)$ for every $n \ge n_0$, where $j \in \{0, \ldots, 2d-1\}$ is an odd integer and $\varepsilon,\delta > 0$ can be chosen to be arbitrarily small. Hence, the assertion follows directly from the properties of $G$ and $U_j (\varepsilon, \delta)$. 
\end{proof}

\begin{rem}\label{rem:par}
Note that by elementary geometry, Proposition~\ref{prop:par_petal2} implies that for every compact set $K \subset P$ there exist $c > 0$ and $n_0 \ge 0$ such that 
\[
|\ap(z) - p| < |z - p|, \qquad |z - \ap(z)| \le c (|z-p| - |\ap(z)-p|)
\]
for every $z \in \bigcup_{n=n_0}^\infty \ap^n(K)$.
\end{rem}

Near a parabolic fixed (or periodic) point $p \in \C$ we consider a family of conformal metrics in $\C \setminus \{p\}$ given by
\[
d\sigma_{p,\alpha} = \sigma_{p,\alpha}(z)|dz| = \frac{|dz|}{|z-p|^\alpha}, \qquad \alpha \in [0, 1).
\]
Note that for $\alpha=0$ the metric coincides with the Euclidean one. 
Now we show that these metrics are locally contracting (resp.~expanding) in attracting (resp.~repelling) petals at a parabolic fixed point of multiplier $1$.

\begin{prop}[{\bf\boldmath Contraction properties in attracting petals at parabolic fixed points}{}] \label{prop:der-par}
Let $P$ be an attracting petal of a holomorphic map $\ap$ at a parabolic fixed point $p$ of multiplier $1$ and order $d \in \N$, let $K \subset P$ be a compact set and let $\alpha \in [0,1)$, $\varepsilon > 0$. Then there exists $n_0 \in \N$ such that for every $m \ge n_0$ one can find a sequence $(a_{m,n})_{n=0}^\infty$ of positive numbers, such that $\sum_{n=0}^\infty a_{m,n} < \infty$, $a_{m,0} = 1$, $\frac{a_{m, n + 1}}{a_{m,n}} = \frac{a_{m+1,n}}{a_{m+1, n-1}}$ for $n \in\N$, $1 > \frac{a_{m, n + 1}}{a_{m,n}} > \frac{a_{m,n}}{a_{m, n-1}} > 1 - \varepsilon$ for $n \in\N$ and
\[
|(\ap^n)'(z)|_{\sigma_{p,\alpha}} < a_{m,n} \quad \text{for every } z \in \ap^m(K), \; n \in \N.
\]

\end{prop}
\begin{proof} Suppose $\ap(z) = z + a (z-p)^{d+1} +\cdots$ for $z$ near $p$, where $a \in \C \setminus \{0\}$ and let $P$ be an attracting petal of $\ap$ at $p$. 
Take a compact set $K \subset P$. By Proposition~\ref{prop:par_petal}(c), for sufficiently large $n_0$ we have $\ap^k(K) \subset U_j (\varepsilon, \frac{\delta}{d})$ for every $k \ge n_0$, where $j \in \{0, \ldots, 2d-1\}$ is an odd integer and $\varepsilon,\delta > 0$ can be chosen to be arbitrarily small. Then for $w \in K$ and $k \ge n_0$, we have $\Arg(a(\ap^k(w)-p)^d) \in (\pi - \delta, \pi + \delta)$, so 
\begin{align*}
|\ap'(\ap^k(w))|_{\sigma_{p,\alpha}} &= \frac{|\ap^k(w)-p|^\alpha |1 + (d+1)a (\ap^k(w)-p)^d + \cdots|}{|\ap^k(w) - p + a(\ap^k(w)-p)^{d+1} + \cdots |^\alpha}\\
&= |1 + (d+1-\alpha)a(\ap^k(w)-p)^d + \cdots|\\
&< 1 - \beta |a| |\ap^k(w)-p|^d,
\end{align*}
where $\beta \in (d, d+1 - \alpha)$ is a fixed number, $n_0$ is chosen sufficiently large and $\varepsilon,\delta > 0$ are chosen sufficiently small. Fix a number $b \in (1, \frac{\beta}{d})$. By \cite[Lemma~10.1]{milnor}, the sequence $k|\ap^k-p|^d$ converges uniformly on $K$ to $\frac{1}{|a|d}$ as $k \to \infty$, so
\[
|\ap^k(w)-p|^d > \frac{b}{\beta |a|k},
\]
if $n_0$ is chosen sufficiently large. Hence,
\[
|\ap'(\ap^k(w))|_{\sigma_{p,\alpha}} < 1 - \frac{b}{k}
\]
for $w \in K$ and $k \ge n_0$. For $m \ge n_0$ define $a_{m,0} = 1$ and
\[
a_{m,n} = \prod_{k= m}^{n + m -1} \Big(1 - \frac{b}{k}\Big)
\]
for $n \in \N$. Then for $z \in \ap^m(K)$ and $n \in \N$, taking $w \in \ap^{-m}(z) \cap K$ we have
\[
|(\ap^n)'(z)|_{\sigma_{p,\alpha}} = \prod_{k= m}^{n+m-1} |(\ap'(\ap^k(w))|_{\sigma_{p,\alpha}} < a_{m,n}.
\]
By definition, $a_{m,n} >0$ and, for $n \in \N$,
\[
a_{m,n} < e^{-b \sum_{k=m}^{n+m-1}1/k} < \frac{2m^b}{(n+m)^b}
\]
if $n_0$ is chosen sufficiently large, so the series $\sum_{n=0}^\infty a_{m,n}$ is convergent, as $b > 1$. Furthermore,
\[
\frac{a_{m, n + 1}}{a_{m,n}} = 1 - \frac{b}{n+m}
\]
for $n \ge 0$, which implies
\[
\frac{a_{m, n + 1}}{a_{m,n}} = \frac{a_{m+1,n}}{a_{m+1, n-1}}, \qquad 1 > \frac{a_{m, n + 1}}{a_{m,n}}  > \frac{a_{m,n}}{a_{m, n-1}} > 1 - \varepsilon
\]
for $n \in \N$, if $n_0$ is chosen sufficiently large.
\end{proof}

Finally, we describe attracting and repelling petals at arbitrary parabolic periodic points. 

\begin{defn}[{\bf Cycle of attracting/repelling petals at a parabolic periodic point}{}]
Let $p$ be a parabolic fixed point of multiplier $1$ of $\f^\ell$, for a holomorphic map $f$ and $\ell \in \N$, and let $P$ be an attracting (resp.~repelling) petal of $\f^\ell$ at $p$.  
By a \emph{cycle} of length $\ell \in \N$ \emph{of attracting} (resp.~\emph{repelling$)$ petals} of $\f$ at $p$ \emph{generated by} $P$ we mean a pairwise disjoint union $P\cup f(P) \cup \ldots \cup\f^{\ell-1}(P)$ (resp.~$\f(P)\cup \ldots \cup\f^{\ell}(P)$).
\end{defn}

Suppose $p \in \C$ is a parabolic periodic point of $($minimal$)$ period $q \in \N$ of a $($nonlinear$)$ holomorphic map $\f$. Then $\f^q$ near $p$ has the form 
\[
\f^q(z) = p + e^{2\pi i \frac{k}{m}} (z-p) + \cdots
\]
for some $k \in \Z$ and $($minimal$)$ $m \in \N$, and $p$ is a parabolic fixed point of the map $\f^\ell$, where $\ell = qm$, of multiplier $1$ and order $d$ for some $d \in \N$. In this case Proposition~\ref{prop:par_petal}(b) implies the following. 

\begin{prop}[{\bf Attracting/repelling petals at a parabolic periodic point}{}]\label{prop:periodic_par_petal}
There exist $d$ attracting $($resp.~repelling$)$ petals $P_j$ with Jordan boundaries, of the map $\f^\ell$ at $p$, where $\ell = qm$, for odd $($resp.~even$)$ integers $j \in \{0, \ldots, 2d-1\}$, such that $P_j$ $($resp.~$\f^\ell(P_j))$ are pairwise disjoint, and for every $\delta \in (0, \frac{\pi}{d})$ one can find $\varepsilon > 0$ with $U_j (\varepsilon, \delta) \subset f^\ell(P_j) \subset P_j \subset U_j (\varepsilon, \frac{\pi}{d})$ $($resp.~$U_j (\varepsilon, \delta) \subset P_j \subset  f^\ell(P_j) \subset U_j (\varepsilon, \frac{\pi}{d}))$. Furthermore, $d$ is a multiple of $m$ and the set $\bigcup_j (P_j\cup f(P_i) \cup \ldots \cup\f^{q-1}(P_j))$ is contained in a disjoint union of $\frac{d}{m}$ cycles of length $\ell$ of attracting $($resp.~repelling$)$ petals of $\f$ at $p$, generated, respectively, by $\frac{d}{m}$ petals $P_j$.
 \end{prop}

For details, see e.g.~\cite{carlesongamelin}.


\section{Attracting and repelling petals at infinity} \label{sec:petals_at_inf}

In analogy to the properties of attracting/repelling petals at parabolic fixed points described in Proposition~\ref{prop:par_petal2}, we introduce a notion of attracting/repelling petals at infinity of holomorphic maps defined on unbounded domains. 

\begin{defn}[{\bf Attracting/repelling petal at infinity}{}]\label{defn:petal_at_inf} 
Let $P \subset \C$ be an unbounded simply connected domain and let $\ap\colon P \to \C$ be a holomorphic map extending to a continuous map from $\overline{P}$ into $\C$. We call the domain $P$ an \emph{attracting petal of $\ap$ at infinity}, if 
\begin{enumerate}[$($a$)$]
\item $\overline{\ap(P)} \subset P$,
\item $\bigcap_{n=0}^\infty \ap^n(P) = \emptyset$,
\item for every compact set $K \subset P$ there exist $c_1, c_2 > 0$, $0< \delta < \frac{\pi}{2}$, $n_0 \in \N$ and a non-increasing logarithmically convex function $g\colon (t_0, +\infty) \to \R_+$, $t_0 > 0$, with $\{|z| : z\in \bigcup_{n=n_0}^\infty \ap^n(K)\} \subset (t_0, +\infty)$, such that
\[
c_1 g(|z|) < |\ap(z) - z| < c_2 g(|z|), \qquad 
|\Arg(\ap(z) - z) -\Arg(z)| < \delta
\]
for $z \in \bigcup_{n=n_0}^\infty \ap^n(K)$.
\end{enumerate}
See Figure~\ref{fig:attr_petal_at_inf}.
\begin{figure}[ht!]
\includegraphics[width=0.4\textwidth]{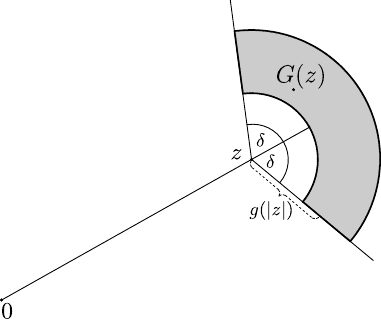}
\caption{\small Location of $\ap(z)$ with respect to $z$ in an attracting petal of $\ap$ at infinity.}\label{fig:attr_petal_at_inf}
\end{figure}

An unbounded simply connected domain $P \subset \C$ is a \emph{repelling petal at infinity} of a holomorphic map $\rp\colon P \to \C$, if $\rp$ is univalent and $\rp(P)$ is an attracting petal at infinity of the map $\ap = \rp^{-1}$.

We also say that $P$ is an attracting/repelling petal at the point $p = \infty$. 
\end{defn}

\begin{rem} Typical examples of logarithmically convex functions $g$ which can be used in Definition~\ref{defn:petal_at_inf}(c) are: 
\begin{alignat*}{2}
g(t) &= 1,\\
g(t) &= \frac{1}{t^a}, &&a > 0, \\
g(t) &= e^{a / t^b}, &\quad &a, b > 0,\\
g(t) &= e^{-a t^b}, &&a > 0, \ 0 < b \le 1.
\end{alignat*}
\end{rem}

Analogously to the properties described in Remark~\ref{rem:par}, the following hold.

\begin{prop}\label{prop:attr_petal} Let $P \subset \C$ be an attracting petal at infinity of a map $\ap$. Then the following hold.
\begin{enumerate}[$($a$)$]
\item $\ap^n(z) \to \infty$ as $n \to \infty$ for $z \in P$.
\item 
For every compact set $K \subset P$ there exist $c > 0$ and $n_0 \ge 0$ such that
\[
|z| < |\ap(z)| \le |z| + c, \qquad |\ap(z) - z| \le c (|\ap(z)| - |z|).
\]
for every $z \in \bigcup_{n=n_0}^\infty \ap^n(K)$. 
\end{enumerate}
\end{prop}
\begin{proof}The statement (a) and the second estimate in (b) follow directly from Definition~\ref{defn:petal_at_inf} and elementary geometry. Together with the fact that $g$ is non-increasing, this implies $|z| < |\ap(z)| \le |z| + c$ for a suitable $c > 0$. 
\end{proof}

Examples of attracting and repelling petals at infinity are presented in the following proposition.

\begin{prop}\label{prop:example_petal_inf} Let
\[
V_j(r, \delta,d,a) = \left\{z \in \C: \Arg(z) \in (\theta_j - \delta, \theta_j + \delta),\, |z| > r\right\},
\]
where $r, \delta > 0$, $j \in \{0, \ldots, 2d-1\}$, $d \in \N$, and
\[
\theta_j = \frac{\Arg(a)}{d} + \frac{\pi j}{d}
\]
for $a \in \C \setminus \{0\}$. 
Suppose $P \subset \C$ is an unbounded simply connected domain and $\ap\colon P \to P$ is a holomorphic map extending continuously to $\overline{P}$, such that $\overline{\ap(P)} \subset P \subset V_j(r, \delta,d,a)$ for some $0 < \delta < \frac{\pi}{d}$, a large number $r > 0$ and an even integer $j \in \{0, \ldots, 2d-1\}$, where 
\[
\ap(z) = z + \frac{a}{z^{d-1}} + o\left(\frac{1}{|z|^{d-1}}\right) \qquad \text{for }\; z \in P \quad \text{as } |z| \to \infty. 
\]
Then $P$ is an attracting petal of $\ap$ at infinity.

Analogously, if $\rp\colon P \to \C$ is a univalent map such that $F^{-1}$ extends continuously to $\overline{F(P)}$ and $\overline{P} \subset \rp(P) \subset V_j(r, \delta,d,a)$ for some $0 < \delta < \frac{\pi}{d}$, a large number $r > 0$ and an odd integer $j \in \{0, \ldots, 2d-1\}$, where 
\[
\rp(z) = z + \frac{a}{z^{d-1}} + o\left(\frac{1}{|z|^{d-1}}\right)\qquad \text{for }\; z \in P \quad \text{as } |z| \to \infty,
\]
then $P$ is a repelling petal of $\rp$ at infinity.
\end{prop}
\begin{proof} 
We proceed as in the case of parabolic fixed points (see e.g.~\cite[Chapter~II.5]{carlesongamelin}). Consider first the case $d= a=1$. Then, given a compact set $K \subset P$, we have $K \subset V_0(r, \delta,d,a)$, where $0 < \delta < \pi$ and 
\[
G(z) = z + 1 + o(1) \qquad \text{for } z \in P \quad \text{as } |z| \to \infty. 
\]
Assuming $r$ sufficiently large, we see that $K \subset V'$, where
\[
V' = \{z \in \C \setminus \{r'\}: \Arg(z-r') \in (-\delta', \delta')\}
\]
for a large $r' > 0$ and $\delta < \delta' < \pi$. Then $\Re(\ap(z)) > \Re(z) + \frac 1 2$ and $|\Arg(\ap(z) - z)| < \min(\delta', \frac{\pi}{5})$ for $z \in P \cap V'$. This implies that $\ap(P \cap V') \subset P \cap V'$ and, consequently,  there exists $n_0 \in \N$ such that $\ap^n(K) \subset V_0(r, \frac{\pi}{4})$ for every $n \ge n_0$. Consequently,
\[
|\Arg(\ap(z) - z) - \Arg(z)| < \frac{\pi}{3} \qquad \text{for } z \in \ap^n(K)
\]
for every $n \ge n_0$ by the definition of $V_0(r, \frac{\pi}{4})$. This proves the second condition from Definition~\ref{defn:petal_at_inf}(c). Taking $g(t) = \frac{1}{t^{d-1}}$, we immediately obtain the first condition.

The case $a \ne 1$, $d = 1$ can be reduced to the previous one by a linear change of coordinates $w = \frac{z}{a}$. In the case $d > 1$, a change of coordinates $w = b z^d$ on $V_j(r, \delta,d,a)$ for a suitable $b \in \C$ reduces it to the case $d = 1$. 

To deal with the case of a repelling petal, it is sufficient to note that if $F$ is univalent with
\[
\rp(z) = z + \frac{a}{z^{d-1}} + o\left(\frac{1}{|z|^{d-1}}\right)
\]
on $P \subset \overline{P} \subset \rp(P) \subset V_j(r, \delta,d,a)$ for an odd integer $j \in \{0, \ldots, 2d-1\}$, then
\[
\ap(z) = \rp^{-1}(z) = z - \frac{a}{z^{d-1}} + o\left(\frac{1}{|z|^{d-1}}\right)
\]
on $F(P) \subset V_j(r, \delta,d, a) = V_{j'}(r, \delta,d,-a)$, where $j' = j \pm 1$ is an even integer in $\{0, \ldots, 2d-1\}$.
\end{proof}

In particular, Proposition~\ref{prop:example_petal_inf} provides the following example, which will be considered in detail in Section~\ref{sec:proofC}.

\begin{ex}
Let $f(z)=z-\tan z $, Newton's method for $\sin z$. Then the half-planes $P_\pm = \{z \in \C: \pm \Im(z) > M\}$ for sufficiently large $M > 0$ are repelling petals of $f$ at infinity. 
\end{ex}

In the further considerations, we will use the following lemma.

\begin{lem}\label{lem:new} Let $P \subset \C$ be an attracting petal at infinity of a map $\ap$. Consider a function $g$ from Definition~{\rm\ref{defn:petal_at_inf}} for a compact set $K \subset P$ and let $(t_n)_{n=1}^\infty$ be the sequence defined in Lemma~{\rm\ref{lem:log_conv}}. Then there exists $M \in \N$ such that for every $z \in K$ and $n \in \N$,
\begin{enumerate}[$($a$)$]
\item the interval $[t_n, t_{n+1}]$ contains less than $M$ numbers $|\ap^k(z)|$, $k \ge 0$,
\item the interval $\conv\{|\ap^n(z)|, |\ap^{n+1}(z)|\}$ contains less than $M$ numbers $t_k$, $k \ge 1$.
\end{enumerate}
\end{lem}
\begin{proof} Take $n_0$ satisfying the conditions of Definition~\ref{defn:petal_at_inf} and Proposition~\ref{prop:attr_petal}, chosen for the set $K$. To show (a), suppose $[t_n, t_{n+1}]$ contains $N$ numbers $|\ap^k(z)|$, $k \ge 0$, for some $z \in K$ and $N > n_0+1$. Then $[t_n, t_{n+1}]$ contains at least $N - n_0$ numbers $|\ap^k(z)|$, $k \ge n_0$, so by Definition~\ref{defn:petal_at_inf} and Proposition~\ref{prop:attr_petal}, there exist $k_0 \ge n_0$ and $c > 0$ such that
\[
t_n \le |\ap^{k_0}(z)| \le \cdots \le |\ap^{k_0+N-n_0-1}(z)| \le t_{n+1}
\]
and
\begin{align*}
g(t_n) &= t_{n+1} - t_n \ge |\ap^{k_0+N-n_0-1}(z)| - |\ap^{k_0}(z)|\\
&= |\ap^{k_0+N-n_0-1}(z)| - |\ap^{k_0+N-n_0-2}(z)| + \cdots + |\ap^{k_0+1}(z)| - |\ap^{k_0}(z)|\\ 
&\ge c(g(|\ap^{k_0+N-n_0-2}(z)|)+ \cdots + g(|\ap^{k_0}(z)|)) \ge c (N-n_0-1) g(t_{n+1}),
\end{align*}
so $\frac{g(t_{n+1})}{g(t_n)} \le \frac{1}{c (N-n_0-1)}$. By Lemma~\ref{lem:log_conv}, the sequence $\frac{g(t_{n+1})}{g(t_n)}$ is bounded from below by a positive constant, which implies that $N$ is bounded from above by some $M >0$ independent of $n \in \N$. 

To show (b), note that by Definition~\ref{defn:petal_at_inf} and Proposition~\ref{prop:attr_petal}, there exists $c > 0$ such that $g(|\ap^n(z)|) \ge c(|\ap^{n+1}(z)| - |\ap^n(z)|)$ for every $z \in K$ and $n \ge n_0$. Let 
\[
q = \frac{1}{1 + c/2}.
\]
Since $q < 1$ and the sequence $\frac{g(t_{k+1})}{g(t_k)}$ converges to $1$ by Lemma~\ref{lem:log_conv}, we can find $k_0 \in \N$ such 
\[
\frac{g(t_{k+1})}{g(t_k)} \ge q \qquad \text{for } k \ge k_0.
\]

Suppose $\conv\{|\ap^n(z)|, |\ap^{n+1}(z)|\}$ contains $N$ numbers $t_k$, $k \ge 1$, for some $z \in K$ and a large $N > n_0$. Since $\ap^n(K)$ is bounded and $t_k \to \infty$ as $k \to \infty$, we can assume $n \ge n_0$ and $|\ap^n(z)| \ge t_{k_0}$. Then by Definition~\ref{defn:petal_at_inf}, Proposition~\ref{prop:attr_petal} and the fact that the sequence $\frac{g(t_{n+1})}{g(t_n)}$ is non-decreasing by Lemma~\ref{lem:log_conv},
\[
t_{k_0}\le |\ap^n(z)| \le t_{k_0+1} \le \ldots \le t_{k_0+N} \le |\ap^{n+1}(z)|
\]
and
\begin{align*}
g(t_{k_0})&\ge g(|\ap^n(z)|)\\
&\ge c(|\ap^{n+1}(z)| - |\ap^n(z)|) \ge c(t_{k_0+N} - t_{k_0+1})\\
&= c(t_{k_0+N} - t_{k_0+N-1} + \cdots + t_{k_0+2} - t_{k_0+1} )\\ 
&= c(g(t_{k_0+1}) + \cdots + g(t_{k_0+N-1}))
\end{align*}
Consequently, we have
\begin{align*}
\frac{q}{2(1-q)} = \frac{1}{c} &\ge \frac{g(t_{k_0+1})}{g(t_{k_0})} + \cdots + \frac{g(t_{k_0 + N -1})}{g(t_{k_0})}\\
&= \frac{g(t_{k_0+1})}{g(t_{k_0})} + \cdots + \frac{g(t_{k_0+1})}{g(t_{k_0})} \cdots \frac{g(t_{k_0 + N -1})}{g(t_{k_0+N -2})}\\
&\ge \frac{g(t_{k_0+1})}{g(t_{k_0})} + \cdots + \Big(\frac{g(t_{k_0+1})}{g(t_{k_0})}\Big)^{N-1}\\
&\ge q + \cdots + q^{N-1} = q \frac{1-q^{N-1}}{1-q}.
\end{align*}
Hence,
\[
q^{N-1} \ge \frac{1}{2},
\]
so
\[
N \le -\frac{\ln 2}{\ln q} + 1.
\]
\end{proof}

Analogously to the case of petals at parabolic fixed points, we consider a family of conformal metrics given by
\[
d\sigma_\alpha = \sigma_\alpha(z) |dz| = \frac{|dz|}{|z|^\alpha}, \qquad \alpha > 1,
\]
for $z \in \C$ near infinity, which have some contracting (resp.~expanding) property inside attracting (resp.~repelling) petals at infinity.

The following theorem, showing the contracting/expanding properties of the metric $\sigma_\alpha$ on attracting/repelling petals at infinity, is one of the main tools used to prove the local connectivity of the boundaries of the Fatou components considered in this paper.

\begin{thm}[{\bf Contraction properties in attracting petals at infinity}{}]\label{thm:der-inf}
Let $P$ be an attracting petal at infinity of a map $\ap$, let $K \subset P$ be a compact set and let $\alpha > 1$, $\varepsilon > 0$. Then there exist $A > 0$ and $n_0 \in \N$ such that for every $m \ge n_0$ there is a sequence $(a_{m,n})_{n=1}^\infty$ such that $0 < a_{m,n} \le A$, $\sum_{n=1}^\infty a_{m,n} < \infty$, with $1 > \frac{a_{m,n + 1}}{a_{m,n}} \ge \frac{a_{m,n}}{a_{m,n-1}}> 1 - \varepsilon$ for $n > 1$, and
\[
|(\ap^n)'(z)|_{\sigma_\alpha} < a_{m,n} \quad \text{for every } z \in \ap^m(K), \; n \in \N.
\]

\end{thm}
\begin{proof} Take a compact set $K \subset P$.
By Cowen's Theorem \cite[Theorem~3.2]{cowen} (see also \cite[Theorem~2.7]{absorb}), there exist a holomorphic map $\psi \colon P \to \Omega$ (where $\Omega \subset \C$ is an open horizontal strip, an open upper half-plane or the plane) and a domain $V \subset P$, such that:
\begin{enumerate}[(i)]
\item $\psi(\ap(z)) = \psi(z) + 1$ for $z \in P$,
\item $\psi$ is univalent on $V$,
\item for every compact set $K_1 \subset P$ there exists $n_1 \in \N$ such that $\ap^n(K_1) \subset V$ for $n \ge n_1$,
\item for every compact set $K_2 \subset \Omega$ there exists $n_2 \in \N$ such that $K_2 +n \subset \psi(V)$ for $n \ge n_2$.
\end{enumerate}
By (iii), we can choose $n_1 \in \N$ such that $\ap^m(K) \subset V$ for $m \ge n_1$. Note that by (i) and the definition of $\Omega$,
\[
L=\conv(\psi(\ap^{n_1}(K) \cup \ap^{n_1+1}(K))) = \conv(\psi(\ap^{n_1}(K)) \cup (\psi(\ap^{n_1}(K)) +1))
\]
is a compact subset of $\Omega$, and so is the set $K_2 = \overline{\bigcup_{w \in L} \D(w, \delta)}$ for a small $\delta > 0$. Therefore, by (i) and (iv), there exists $n_0 > n_1$ such that 
\begin{equation}\label{eq:conv}
\bigcup\{\D(w, \varepsilon): w \in \conv(\psi(\ap^{m + n}(K) \cup \ap^{m + n + 1}(K)))\} = K_2 + m + n - n_1 \subset \psi(V)
\end{equation}
for $m \ge n_0$ and $n \ge 0$. Enlarging $n_0$, we can assume that the properties listed in Definition~\ref{defn:petal_at_inf} and Proposition~\ref{prop:attr_petal} hold for every $z \in \bigcup_{m=n_0}^\infty \ap^m(K)$. 

Choose a point $z_0 \in K$ and let
\[
z_l = \ap^l(z_0) \in G^l(K)
\]
for $l \ge n_0$. Consider
\[
z \in G^m(K) \qquad \text{for } m \ge n_0. 
\]
This notation will be used throughout the subsequent part of the proof.

\subsection*{Convention} By $c_0, c_1, \ldots$ we denote constants independent of $z \in \ap^m(K)$, $m \ge n_0$ and $n \ge 0$. Moreover, we write $g_{m,n}(z) \asymp h_{m,n}(z)$ if 
\[
\frac{1}{c} < \frac{g_{m,n}(z)}{h_{m,n}(z)} < c
\]
for a constant $c > 0$ independent of $z \in \ap^m(K)$, $m \ge n_0$ and $n \ge 0$.

\bigskip
By \eqref{eq:conv} and the Koebe Distortion Theorem, $\psi^{-1}$ is defined on $\conv(\psi(\ap^{m+n}(K) \cup \ap^{m+n+1}(K)))$ with distortion bounded by a constant independent of $m \ge n_0, n \ge 0$. In particular,
\begin{equation} \label{eq:dist0}
|\psi'(\ap^n(z))| \asymp |\psi'(z_{m+n})|
\end{equation}
for $n \ge 0$. Moreover, by \eqref{eq:conv}, the bounded distortion of $\psi^{-1}$ and the fact $\psi(\ap^{n+1}(z)) - \psi(\ap^n(z)) = 1$, we obtain
\begin{equation} \label{eq:dist1}
|\ap^{n+1}(z) - \ap^n(z)| \asymp |z_{m+n+1} - z_{m+n}| \asymp \frac{1}{|\psi'(z_{m+n})|}.
\end{equation}
Let $w \in G^{n_0}(K)$ be such that $z = G^{m-n_0}(w)$. By \eqref{eq:dist1}, Proposition~\ref{prop:attr_petal} and the compactness of $G^{n_0}(K)$,
\begin{align*}
|z_{m+n}| &\leq |z_{m+n}- z_{m+n-1}| + \cdots + |z_{n_0+1} - z_{n_0}| +|z_{n_0}| \\
&\le c_1(|\ap^{m-n_0+n}(w) - \ap^{m-n_0+n-1}(w)| + \cdots + |\ap(w) - w| +|w|)\\
&\le c_2(|\ap^{m-n_0+n}(w)| - |\ap^{m-n_0+n-1}(w)| + \cdots + |\ap(w)| - |w| +|w|)\\ 
&= c_2 |\ap^{m-n_0+n}(w)| = c_3 |\ap^n(z)|
\end{align*}
for constants $c_1, c_2 > 0$. Analogously,
\begin{align*}
|\ap^n(z)| &= |\ap^{m-n_0+n}(w)|\\
&\leq |\ap^{m-n_0+n}(w) - \ap^{m-n_0+n-1}(w)| + \cdots + |\ap(w) - w| +|w|\\
&\le c_3(|z_{m+n}- z_{m+n-1}| + \cdots + |z_{n_0+1} - z_{n_0}| +|z_{n_0}|)\\
&\le c_4(|z_{m+n}| - |z_{m+n-1}| + \cdots + |z_{n_0+1}| - |z_{n_0}| +|z_{n_0}|) = c_4 |z_{m+n}|
\end{align*}
for constants $c_3, c_4 > 0$. We conclude
\begin{equation}\label{eq:dist3}
|\ap^n(z)| \asymp |z_{m+n}|.
\end{equation}
Furthermore, by (i),
\begin{align*}
|(\ap^n)'(z)| &= |\psi'(z)||(\psi^{-1})'(\psi(z) + n)| = \frac{|\psi'(z)|}{|\psi'(\ap^n(z))|},\\
|(\ap^n)'(z_m)| &= |\psi'(z_m)||(\psi^{-1})'(\psi(z_m) + n)| = \frac{|\psi'(z_m)|}{|\psi'(z_{m+n}))|},
\end{align*}
which together with \eqref{eq:dist0} and \eqref{eq:dist1} gives
\begin{equation} \label{eq:dist2}
|(\ap^n)'(z)| \asymp |(\ap^n)'(z_m)| \asymp \frac{|\ap^{n+1}(z) - \ap^n(z)|}{|\ap(z) - z|} \asymp \frac{|z_{m+n+1} - z_{m+n}|}{|z_{m+1} - z_m|}.
\end{equation}

Fix $\alpha > 1$. Using \eqref{eq:dist3} and \eqref{eq:dist2} we obtain
\begin{equation}\label{eq:F^k'<}
|(\ap^n)'(z)|_{\sigma_\alpha} = \frac{|z|^\alpha |(\ap^n)'(z)|}{|\ap^n(z)|^\alpha} \leq c_5 \frac{|z_m|^\alpha |z_{m+n+1} - z_{m+n}|}{|z_{m+n}|^\alpha|z_{m+1} - z_m| } 
\end{equation}
for $n \in \N$ and some constant $c_5 > 0$ (note that the metric $\sigma_\alpha$ is well-defined at $z$ and $G^n(z)$ if $n_0$ is chosen sufficiently large).
Consider now the function $g$ from Definition~\ref{defn:petal_at_inf} for the set $K$, and the sequence $(t_j)_{j=1}^\infty$ defined in Lemma~\ref{lem:log_conv}. Enlarging $n_0$ if necessary, we can assume $t_1 \le |z_{n_0}|$. Note that the sequence $|z_l|$, $l \ge n_0$, is increasing by Proposition~\ref{prop:attr_petal}. 

Let
\[
j_l = \max\{j \in \N: t_j \le |z_l|\}, \qquad l \ge n_0.
\]
By definition,
\begin{equation}\label{eq:j_l}
t_{j_l} \le |z_l| < t_{j_l + 1}
\end{equation}
and the sequence $(j_l)_{l=n_0}^\infty$ (and also $(t_{j_l})_{l=n_0}^\infty$) is non-decreasing. Choosing $n_0$ sufficiently large, we can assume $t_{j_{n_0}} > 0$ and, by Lemma~\ref{lem:log_conv},
\begin{equation}\label{eq:tj}
1 > \frac{t_j}{t_{j+1}} > \frac{t_{j-1}}{t_j} > q, \qquad 1 \ge \frac{g(t_{j+1})}{g(t_j)} \ge \frac{g(t_j)}{g(t_{j-1})} > q
\end{equation}
for $j \ge j_{n_0} -1$ and a constant $q \in (0,1)$ arbitrarily close to $1$.

Let $m \ge n_0$, $n \in \N$. Note that Lemma~\ref{lem:new} implies that there exists an integer $M > 1$, such that
\begin{equation}\label{eq:j>}
j_{m+n} \ge j_m + \frac{n+1}{M} - 1 \ge j_m + \left\lfloor\frac{n-2}{M}\right\rfloor - 1.
\end{equation}
Using consecutively \eqref{eq:F^k'<}, Definition~\ref{defn:petal_at_inf}, \eqref{eq:j_l}, \eqref{eq:tj} and \eqref{eq:j>}, we obtain
\begin{equation*}\label{eq:<g}
\begin{aligned}
|(\ap^n)'(z)|_{\sigma_\alpha} &\le c_6 \frac{|z_m|^\alpha g(|z_{m+n}|)}{|z_{m+n}|^\alpha g(|z_m|)} < c_6 \frac{t_{j_m+1}^\alpha g(t_{j_{m+n}})}{t_{j_{m+n}}^\alpha g(t_{j_m+1} )}\\
&< c_7 \frac{t_{j_m}^\alpha g(t_{j_{m+n}})}{t_{j_{m+n}}^\alpha g(t_{j_m})}\le c_7 \frac{t_{j_m}^\alpha g\big(t_{j_m + \lfloor \frac{n-2}{M} \rfloor - 1}\big)}{t_{j_m + \lfloor \frac{n-2}{M} \rfloor - 1}^\alpha g(t_{j_m})}
\end{aligned}
\end{equation*}
for some constants $c_6, c_7 > 0$, so that
\begin{equation}\label{eq:<tildea}
|(\ap^n)'(z)|_{\sigma_\alpha} < \tilde a_{m,n}
\end{equation}
for $n \in \N$, where
\[
\tilde a_{m,n} = c_7 \frac{t_{j_m}^\alpha g\big(t_{j_m + \lfloor \frac{n-2}{M} \rfloor - 1}\big)}{t_{j_m + \lfloor \frac{n-2}{M} \rfloor - 1}^\alpha g(t_{j_m})}.
\]
Note that by definition,
\begin{equation}\label{eq:tildea}
\begin{aligned}
\tilde a_{m,1} &= c_7 \frac{t_{j_m}^\alpha g(t_{j_m -2})}{t_{j_m -2}^\alpha g(t_{j_m})},\\
\tilde a_{m,kM+r} = \tilde a_{m,(k+1)M+1} &= c_7 \frac{t_{j_m}^\alpha g(t_{j_m + k-1})}{t_{j_m + k-1}^\alpha g(t_{j_m})} \qquad \text{for } k \ge 0,\; r\in\{2, \ldots, M+1\},
\end{aligned}
\end{equation}
so by \eqref{eq:tj},
\[
\tilde a_{m,1} > \tilde a_{m,2} = \cdots = \tilde a_{m,M+1} > \tilde a_{m,M+2} = \cdots = \tilde a_{m,2M+1} > \cdots
\]
and
\[
0 < \tilde a_{m,n} \le \tilde a_{m,1} \le A
\]
for
\[
A = \frac{c_7}{q^{2(\alpha + 1)}}.
\]
Moreover, \eqref{eq:tildea} and \eqref{eq:tj} imply
\begin{equation}\label{eq:a/a}
1 > \frac{\tilde a_{m, (k+1)M+1}}{\tilde a_{m,kM+1}} \ge \frac{\tilde a_{m,kM+1}}{\tilde a_{m, (k-1)M+1}} \ge q^{\alpha + 1}
\end{equation}
for $k \in \N$ and
\begin{align*}
\tilde a_{m,1} + \cdots + \tilde a_{m,(k+1)M +1} &= c_7 + c_7 M \frac{t_{j_m}^\alpha }{g(t_{j_m})} \left(\frac{g(t_{j_m})}{t_{j_m}^\alpha} + \cdots + \frac{g(t_{j_m + k-1})}{t_{j_m + k-1}^\alpha}\right)\\
&< c_7 + \frac{c_7 M}{q^\alpha} \frac{t_{j_m}^\alpha }{g(t_{j_m})} \left(\frac{g(t_{j_m})}{t_{j_m+1}^\alpha} + \cdots + \frac{g(t_{j_m+k-1})}{t_{j_m+ k}^\alpha}\right)\\
&= c_7 + \frac{c_7 M}{q^\alpha} \frac{t_{j_m}^\alpha }{g(t_{j_m})} \left(\frac{t_{j_m+1} - t_{j_m}}{t_{j_m+1}^\alpha} + \cdots + \frac{t_{j_m+k} - t_{j_m+k-1}}{t_{j_m+k}^\alpha}\right)\\
&< c_7 + \frac{c_7 M}{q^\alpha} \frac{t_{j_m}^\alpha }{g(t_{j_m})} \left(\int_{t_{j_m}}^{t_{j_m+1}} \frac{dt}{t^\alpha} + \cdots + \int_{t_{j_m + k-1}}^{t_{j_m+k}} \frac{dt}{ t^\alpha}\right)\\
&= c_7 + \frac{c_7 M}{q^\alpha} \frac{t_{j_m}^\alpha }{g(t_{j_m})} \int_{t_{j_m}}^{t_{j_m+k}} \frac{dt}{ t^\alpha} < c_7 + \frac{c_7 M}{q^\alpha} \frac{t_{j_m}^\alpha }{g(t_{j_m})}\int_{t_{j_m}}^\infty \frac{dt}{t^\alpha} < \infty
\end{align*}
for $k \in \N$. Hence, 
\[
\sum_{n=1}^\infty \tilde a_{m,n} < \infty.
\]

For $m \ge n_0$ define a sequence $(a_{m,n})_{n=1}^\infty$ setting
\[
a_{m, kM+s} = \tilde a_{m, kM+1} \left(\frac{\tilde a_{m, (k+1)M+1}}{\tilde a_{m, kM+1}}\right)^{\frac{s - 1} M}
\]
for $k \ge 0$, $s \in \{1, \ldots, M\}$. By \eqref{eq:a/a},
\[ 
\tilde a_{m, (k+1)M +1} = a_{m, (k+1)M+1} < a_{m, kM+s} \le a_{m, kM+1} = \tilde a_{m, kM+1},
\]
which implies
\[
\sum_{n=1}^\infty a_{m,n} = \sum_{k=0}^\infty\sum_{s=1}^M a_{m, kM+s} \le M \sum_{k=0}^\infty \tilde a_{m, kM+1} < \infty.
\]
and (together with \eqref{eq:tildea})
\[
0 < \tilde a_{m,n} \le a_{m,n} \le A.
\]
Note that this and \eqref{eq:<tildea} imply
\[
|(\ap^n)'(z)|_{\sigma_\alpha} < a_{m,n}
\]
for $z \in \ap^m(K)$, $m \ge n_0$, $n \in \N$. Furthermore,
\[
\frac{a_{m, kM+s+1}}{a_{m, kM+s}} = \left(\frac{\tilde a_{m, (k+1)M+1}}{\tilde a_{m, kM+1}}\right)^{\frac{1}{M}},
\]
so by \eqref{eq:a/a},
\[
1 > \frac{a_{m,n+1}}{a_{m,n}} \ge \frac{a_{m,n}}{a_{m,n-1}} \ge q^{\frac{\alpha + 1}{M}}
\]
for $n > 1$, where $q$ is arbitrarily close to $1$. This ends the proof.

\end{proof}


\section{Construction of an expanding metric} \label{sec:metric}

The goal of this section is to prove the following result.

\begin{thm}[{\bf Existence of an expanding metric -- general version}{}]\label{thm:metric}
Let $\rp \colon V' \to V$ be a holomorphic map onto $V$, where $V \subset \C$ is a hyperbolic domain and $V'$ is a domain such that $V' \subset V$, $V'\ne V$. Let $W \subset V$ be an open set such that $\bigcap_{n=0}^\infty \rp^{-n}(W) = \emptyset$. Assume that the following hold.
\begin{enumerate}[$($a$)$]

\item \label{item:asympt_values} $\rp$ has no asymptotic values.

\item \label{item:deg} For every $z \in V$, we have $\sup\{\deg_w \rp^n: w \in \rp^{-n}(z),\, n \in \N\} < \infty$.

\item \label{item:crit} $\PP_{crit}(\rp)$ has no accumulation points in $V$.

\item \label{item:extend} $\rp$ extends meromorphically to a neighbourhood $($in $\C)$ of $\overline{W \cap \rp^{-1}(W)}$.

\item \label{item:petals} 
There exist finite collections $\II_{par}$, $\II_\infty$, and sets $P_i$, $i \in \II_{par} \cup \II_\infty$, such that:
\begin{itemize}[$\circ$]
\item for $i \in \II_{par}$, the set $P_i$ is a repelling petal of $\rp^{\ell_i}$, for some $\ell_i \in \N$, at a parabolic periodic point $p_i \in \bd W$ of $\rp$, such that $P_i$ generates a cycle of length $\ell_i$ of repelling petals of $\rp$ at $p_i$, and these cycles are pairwise disjoint for $i \in \II_{par}$,
\item for $i \in \II_\infty$, the set $P_i$ is a repelling petal of $\rp$ at infinity, such that $W \cap \rp(P_i)$ is disjoint from the union of the cycles of repelling petals generated by $P_{i'}$, $i' \in\II_{par}$, and from $\bigcup_{i' \in \II_\infty, i' \ne i} P_{i'}$,
\item $\overline{W} \subset V \cup \bigcup_{i\in\II_{par}} \Orb(p_i)$,
\item $\overline{W} \setminus \Big(\bigcup_{i \in \II_{par}}\bigcup_{s=0}^{\ell_i-1} \rp^s(P_i \cup \{p_i\}) \cup \bigcup_{i \in \II_\infty} P_i\Big)$ is compact,
\item $\rp^{\ell_i}(W \cap P_i) \subset W$ for $i \in \II_{par}$ and $\rp(W \cap P_i) \subset W$ for $i \in \II_\infty$.
\end{itemize}
\end{enumerate}
Then one can find $N \in \N$ such that for every compact set $K \subset \overline{W} \setminus \bigcup_{i\in\II_{par}} \Orb(p_i)$ there exist a conformal metric $d\varsigma = \varsigma|dz|$ on $W \cap \rp^{-1}(W) \cap \ldots \cap \rp^{-N}(W)$ and a decreasing sequence $(b_n)_{n=1}^\infty$ of numbers $b_n \in (0,1)$ with $\sum_{n=1}^\infty b_n < \infty$, satisfying
\[
|(\rp^n)'(z)|_\varsigma > \frac{1}{b_n} 
\]
for every $z \in W \cap \rp^{-1}(W) \cap \ldots \cap \rp^{-(n+N)}(W) \cap \rp^{-n}(K)$, $n \in \N$. Furthermore, if $W \cap \PP_{crit}(\rp) = \emptyset$, then one can choose $N=0$.
\end{thm}

The construction of the metric $d\varsigma$ follows the ideas of \cite{orsay2,carlesongamelin,milnor,tanlei-loc_con} in the case of polynomials and rational maps and \cite{helena,leticia-orbifold,ARS} in a transcendental context. However, due to the lack of compactness on unbounded parts of the set $W$, to estimate $|(\rp^n)'|_\varsigma$ we must take a different approach than the ones used in the cited references. Since the proof is rather involved, first we present its general description. A main idea is to construct a metric $\varsigma$ which is uniformly expanding on some compact set in $V'$ and `glue' it with suitable metrics in the union of the cycles of repelling petals at parabolic points and repelling petals at infinity, such that the derivative of the iterations of $\rp$ along a block of length $n$ within this union with respect to this metric is larger than $1/\beta_n$, where $\beta_n$ is a term of a converging series and $\frac{\beta_{n + 1}}{\beta_n} \ge \frac{\beta_n}{\beta_{n-1}}$. This is described precisely in the following proposition. For simplicity, here and in the sequel we write
\[
W_n = W \cap \rp^{-1}(W) \cap \ldots \cap \rp^{-n}(W), \qquad n \ge 0.
\]

\begin{prop}\label{prop:expanding_metric} Under the assumptions of Theorem~{\rm\ref{thm:metric}}, one can find $N \in \N$ such that for every compact set $K \subset \overline{W} \setminus\bigcup_{i\in\II_{par}} \Orb(p_i)$ there exist a conformal metric $d\varsigma = \varsigma|dz|$ on $W_N$, a compact set $\widehat K \subset V \setminus \bigcup_{i\in\II_{par}} \Orb(p_i)$ containing $K$, a number $Q > 1$ and a decreasing sequence $(\beta_n)_{n=0}^\infty$ of positive numbers $\beta_n$, satisfying the following properties:
\begin{enumerate}[$($a$)$]
\item $\sum_{n=0}^\infty \beta_n < \infty$,
\item $\beta_0 = 1$ and $\frac{\beta_{n + 1}}{\beta_n} \ge \frac{\beta_n}{\beta_{n-1}}$ for every $n \in\N$,
\item[$($c$)$] $|\rp'|_\varsigma > Q$ on $W_N \cap \widehat K$,
\item[$($d$)$] $|(\rp^n)'(z)|_\varsigma > \frac{1}{\beta_n}$ for every $z \in (W_N \setminus \widehat K) \cap\rp^{-1} (W_N \setminus \widehat K) \cap \ldots \cap \rp^{-(n-1)}(W_N \setminus \widehat K) \cap \rp^{-n}(W_N \cap \widehat K)$, $n \in \N$.
\end{enumerate} Furthermore, if $W \cap \PP_{crit}(\rp) = \emptyset$, then one can choose $N=0$.
\end{prop}

First, we show how to prove Theorem~\ref{thm:metric} using this proposition.

\begin{proof}[Proof of Theorem~{\rm\ref{thm:metric}} assuming Proposition~{\rm\ref{prop:expanding_metric}}]

Set
\[
b_n = \max \Big(\frac{1}{Q^{n /2}}, \beta_{\lceil n/2\rceil}\Big).
\]
Note that $b_n \in (0,1)$, the sequence $(b_n)_{n=1}^\infty$ is decreasing, and $\sum_{n=1}^\infty\beta_{\lceil n/2\rceil} \le 2 \sum_{n=1}^\infty\beta_n < \infty$, so $\sum_{n=1}^\infty b_n < \infty$.

Take a point $z \in W_{n+N}$ for some $n \in \N$ such that $\rp^n(z) \in K$ (and hence $\rp^n(z) \in\widehat K$). We can divide the set $[0, \ldots, n-1]$ into consecutive disjoint blocks $A_1, B_1, \ldots, A_l, B_l$ of (maximal) lengths (i.e.~numbers of elements), respectively, $k_1, m_1, \ldots, k_l, m_l$ for some $l \in \N$, such that if $s \in A_j$ for some $j$, then $\rp^s(z) \in W_N \setminus \widehat K$ and if $s \in B_j$ for some $j$, then $\rp^s(z) \in W_N \cap \widehat K$. We have $k_1 + \cdots + k_l + m_1 + \cdots + m_l = n$. We have $k_j, m_j > 0$ except for $k_1$ and $m_l$, which can be equal to $0$. 

Let
\begin{align*}
\Delta(A_j) &= |(\rp^{k_j})'(\rp^{k_1 + m_1 + \cdots + k_{j-1} + m_{j-1}}(z))|_\varsigma,\\ 
\Delta(B_j) &= |(\rp^{m_j})'(\rp^{k_1 + m_1 + \cdots + k_{j-1} + m_{j-1} + k_j}(z))|_\varsigma,
\end{align*}
where we set $\Delta(\emptyset) = 1$ for an empty block. Then 
\[
|(\rp^n)'(z)|_\varsigma = \prod_{j=1}^l \Delta(A_j)\Delta(B_j).
\]
By Proposition~\ref{prop:expanding_metric},
\[
\Delta(A_j) > \frac{1}{\beta_{k_j}}, \qquad \Delta(B_j) > Q^{m_j}.
\]
Hence,
\[
|(\rp^n)'(z)|_\varsigma > \frac{Q^{m_1 + \cdots + m_l}}{\beta_{k_1} \cdots \beta_{k_l}}.
\]
If $m_1 + \cdots + m_l > \frac n 2$, then, since $\beta_{k_j} < 1$, we have
\[
|(\rp^n)'(z)|_\varsigma > Q^{m_1 + \cdots + m_l} = Q^{n /2} \ge \frac{1}{b_n}.
\]
On the other hand, if $m_1 + \cdots + m_l \le \frac n 2$, then $k_1 + \cdots + k_l > \frac n 2$ and
\[
|(\rp^n)'(z)|_\varsigma > \frac{1}{\beta_{k_1} \cdots \beta_{k_l}}.
\]
By Proposition~\ref{prop:expanding_metric}, for every $p, q \in \N$,
\[
\beta_{p+q} = \frac{\beta_{p+q}}{\beta_{p+q-1}} \frac{\beta_{p+q-1}}{\beta_{p+q-2}} \cdots \frac{\beta_{q+1}}{\beta_q} \beta_q \ge
\frac{\beta_p}{\beta_{p-1}} \frac{\beta_{p-1}}{\beta_{p-2}} \cdots \frac{\beta_1}{\beta_0}\beta_q = \beta_p \beta_q.
\]
Applying this inductively, we obtain
\[
|(\rp^n)'(z)|_\varsigma > \frac{1}{\beta_{k_1} \cdots \beta_{k_l}} \ge \frac{1}{\beta_{k_1 + \cdots + k_l}} \ge \frac{1}{\beta_{\lceil n/2\rceil}}\ge \frac{1}{b_n}.
\]
We conclude that in both cases $|(\rp^n)'(z)|_\varsigma > \frac{1}{b_n}$, which ends the proof.
\end{proof}

The plan to prove Proposition~{\rm\ref{prop:expanding_metric} is as follows.
Since dealing with petals at parabolic periodic points (instead of fixed ones with multipliers $1$) makes the construction significantly more complicated in notation, without introducing new ideas into the proof, we first assume that all the parabolic periodic points $p_i$, $i \in \II_{par}$, are fixed under $\rp$, of multipliers $1$, and present in Subsections~\ref{subsec:metric_petal_dyn}--\ref{subsec:metric_metric} all the details of the proof in this case. 

In Subsection~\ref{subsec:metric_petal_dyn} we introduce notation and describe the dynamics of the map $\rp$ within the repelling petals $P_i$. Then, in Subsection~\ref{subsec:metric_orbifold}, we show that $V$ has an orbifold structure (in the sense of Subsection~\ref{subsec:orbifold}) and the orbifold metric $d\rho$ is strictly expanding on compact sets in $V'$ outside the singularities of the metric $d\rho$  (Lemma~\ref{lem:g-expand}). This part follows a classical reasoning described e.g.~in \cite{milnor}. 

In Subsection~\ref{subsec:metric_petal} we construct a suitable metric within the repelling petals $P_i$. For a petal at a parabolic fixed point $p_i$ we use the locally expanding metric $d\sigma_{p_i, \alpha_{par}}$, for a suitable $\alpha_{par} \in (0,1)$, defined in Section~\ref{sec:petals_at_par}, and estimates provided by Proposition~\ref{prop:der-par}. In the case of a petal $P_i$ at infinity we use the metric $d\sigma_{\alpha_\infty}$, for a suitable $\alpha_\infty > 1$, introduced in Section~\ref{sec:petals_at_inf}, and estimates given by Theorem~\ref{thm:der-inf}. Note that in this case, to obtain local expansion of $d\sigma_{\alpha_\infty}$ on suitable parts of $P_i$ (Lemma~\ref{lem:sigma-expand}), we must first correct the metric by multiplying it by an appropriate real function $h_i$. In Subsection~\ref{subsec:metric_metric}, we define a metric $\varsigma$ by gluing previously constructed metrics and show that it has required properties (Lemmas~\ref{lem:g-sigma} and~\ref{lem:der_i}). 

Finally, in Subsection~\ref{subsec:periodic}, we explain how the construction is modified in the presence of parabolic periodic points $p_i$, $i \in \II_{par}$, in $\bd W$.

Now we provide a detailed proof of Proposition~\ref{prop:expanding_metric} along the lines presented above.

\subsection{Petal dynamics}\label{subsec:metric_petal_dyn}

As explained above, in Subsections~\ref{subsec:metric_petal_dyn}--\ref{subsec:metric_metric}  we assume that all the parabolic periodic points $p_i$, $i \in \II_{par}$, are fixed under the map $\rp$, of multipliers $1$. 
Note that in this case the assumption~\ref{item:petals} of Theorem~\ref{thm:metric} has the form
\emph{\begin{enumerate}[$($e\rm{'}$)$]
\item \label{item:petals'} There exist finite collections $\II_{par}$, $\II_\infty$, and sets $P_i$, $i \in \II_{par} \cup \II_\infty$, such that:
\begin{itemize}[$\circ$] 
\item for $i \in \II_{par}$, the set $P_i$ is a repelling petal of $\rp$ at a parabolic fixed point $p_i \in \bd W$ of multiplier $1$, such that $\rp(P_i)$ are pairwise disjoint for $i \in \II_{par}$,
\item for $i \in \II_\infty$, the set $P_i$ is a repelling petal of $\rp$ at infinity, such that $W \cap \rp(P_i)$ is disjoint from $\bigcup_{i' \in \II_{par}} P_{i'} \cup \bigcup_{i' \in \II_\infty, i' \ne i} P_{i'}$,
\item $\overline{W} \subset V \cup \{p_i\}_{i\in\II_{par}}$,
\item $\overline{W} \setminus \Big(\bigcup_{i \in \II_{par}}(P_i \cup \{p_i\}) \cup \bigcup_{i \in \II_\infty} P_i\Big)$ is compact,
\item $\rp(W \cap P_i) \subset W$ for $i \in \II_{par} \cup \II_\infty$.
\end{itemize}
\end{enumerate}
}
For convenience, setting 
\[
\II =  \II_{par} \cup  \II_\infty
\]
and
\[
p_i = \infty \qquad \text{for } \; i \in \II_\infty,
\]
we describe simultaneously the dynamics in repelling petals of $\rp$ at parabolic fixed points and at infinity. 

By the assumption~\ref{item:petals'}, the set
\[
Y = \overline{W} \setminus \bigcup_{i \in \II} (P_i \cup \{p_i\})
\] 
is a compact subset of $V \setminus \{p_i\}_{i \in \II}$.
Fix $i \in \II$. 
By the definition of repelling petals, denoting a possible holomorphic extension of $\rp$ to $P_i$ by the same symbol, we have $p_i \notin P_i$, $\rp$ is univalent on $P_i$ and $\rp(P_i)$ is an attracting petal at $p_i$ of the map
\[
\ap_i \colon \rp(P_i) \to P_i, \qquad \ap_i = (\rp|_{P_i})^{-1},
\]
where $\ap_i^n \to p_i$ as $n \to \infty$, $\bigcap_{n=0}^\infty \ap_i^n(P_i) = \emptyset$ and $\ap_i$ extends continuously to $\overline{\rp(P_i)}$ (see Section~\ref{sec:petals_at_par} and Proposition~\ref{prop:attr_petal}). Note that in the case $i \in \II_{par}$ we have $p_i \in \overline {P_i}$ and $\ap_i$ extends holomorphically to a neighbourhood of $p_i$. 

Let
\[
K_i = \ap_i(Y \cap \overline{\rp(P_i)}).
\]
See Figure~\ref{fig:K}.
\begin{figure}[ht!]
\includegraphics[width=0.6\textwidth]{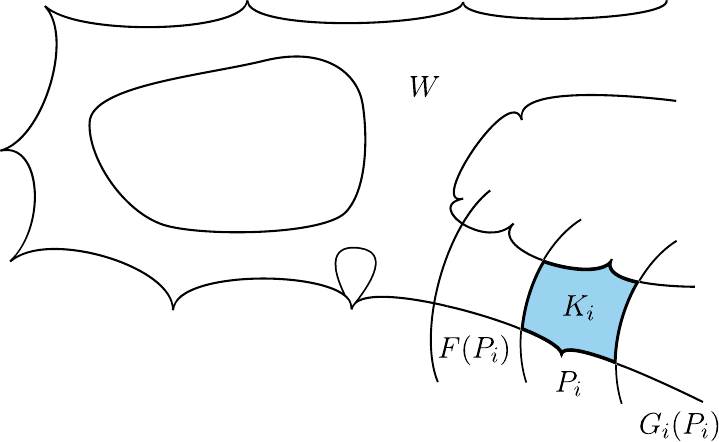}
\caption{\small The set $K_i$.}
\label{fig:K}
\end{figure}
By Definition~\ref{defn:petal_at_inf}, the set $K_i$ is a compact subset of $\rp(P_i)$. Furthermore,
\[
K_i \subset \overline{P_i} \setminus (\ap_i(P_i) \cup \{p_i\}),
\]
so $\ap_i(K_i) \subset \overline{\ap_i(P_i}) \setminus (G^2_i(P_i) \cup \{p_i\}) \subset P_i \setminus G^2_i(P_i)$, which implies 
\begin{equation}\label{eq:K_i_disjoint}
\ap_i^{n_1}(K_i) \cap \ap_i^{n_2}(K_i) = \emptyset \quad \text{for } \;n_1, n_2 \ge 0, \;|n_1 - n_2| > 1. 
\end{equation}

We will use frequently the following fact.

\begin{lem}\label{lem:WinKi} For $i \in \II$,
\[
\overline{W} \cap \overline{P_i} \setminus \{p_i\} \subset \bigcup_{n=0}^\infty \ap_i^n(K_i).
\]
\end{lem}
\begin{proof}
We show inductively
\begin{equation}\label{eq:WinKi-ind}
\overline{W} \cap \overline{P_i} \subset \bigcup_{k=0}^{n-1} \ap_i^k(K_i) \cup \overline{\ap_i^{n+1}(W \cap \rp(P_i))}
\end{equation}
for $n \ge 0$ (with the convention that a union over an empty set is empty). 
To do it, note first that the last part of the assumption~\ref{item:petals'} gives $\overline{W} \cap \overline{P_i} \subset \overline{\ap_i(W \cap \rp(P_i))}$, which shows \eqref{eq:WinKi-ind} for $n=0$. Suppose now \eqref{eq:WinKi-ind} holds for some $n \ge 0$. To prove \eqref{eq:WinKi-ind} for $n+1$, it is enough to verify
\begin{equation}\label{eq:WinKi-ind1}
\overline{\ap_i^{n+1}(W \cap \rp(P_i))} \subset \ap_i^n(K_i) \cup \overline{\ap_i^{n+2}(W \cap \rp(P_i))}.
\end{equation}
To show \eqref{eq:WinKi-ind1}, note that by the definition of $Y$, we have
\[
W \setminus Y \subset \bigcup_{i' \in \II} P_{i'},
\]
so, as $W \cap \rp(P_i) \cap P_{i'} = \emptyset$ for $i, i' \in \II$, $i\neq i'$ by the assumption~\ref{item:petals'},
\[
(W \setminus Y) \cap \rp(P_i) \subset W \cap P_i.
\]
This together with the last part of the assumption~\ref{item:petals'} implies
\[
(W \setminus Y) \cap \rp(P_i) \subset \ap_i(W \cap \rp(P_i)),
\]
which gives
\[
W \cap \rp(P_i) \subset (W \cap Y \cap \rp(P_i)) \cup \ap_i(W \cap \rp(P_i)) \subset \rp(K_i) \cup \ap_i(W \cap \rp(P_i))
\]
and, consequently,
\[
\ap_i^{n+1}(W \cap \rp(P_i)) \subset \ap_i^n(K_i) \cup \ap_i^{n+2}(W \cap \rp(P_i)).
\]
This shows \eqref{eq:WinKi-ind1}, completing the inductive proof of \eqref{eq:WinKi-ind}.

Using \eqref{eq:WinKi-ind}, we obtain
\begin{equation}\label{eq:WinKi2}
\overline{W} \cap \overline{P_i} \subset \bigcap_{n=0}^\infty \Big( \bigcup_{k=0}^{n-1} \ap_i^k(K_i) \cup \overline{\ap_i^{n+1}(W \cap \rp(P_i))}\Big) \subset \bigcup_{n=0}^\infty \ap_i^k(K_i) \cup \bigcap_{n=0}^\infty\overline{\ap_i^{n+1}(W \cap \rp(P_i))}.
\end{equation}

By the definition of attracting petals, we have $\overline{\ap_i(P_i)} \subset P_i \cup \{p_i\}$ and $\bigcap_{n=0}^\infty \ap_i^n(P_i) = \emptyset$, which implies 
\[
\bigcap_{n=0}^\infty\overline{\ap_i^{n+1}(W \cap \rp(P_i))} = \overline{\bigcap_{n=0}^\infty\ap_i^{n+1}(W \cap \rp(P_i))} \subset \overline{\bigcap_{n=0}^\infty \ap_i^n(P_i)} \subset \{p_i\}.
\]
This together with \eqref{eq:WinKi2} proves the lemma.
\end{proof}

Fix a number $n_0 \in \N$, which is larger than all the numbers $n_0$ appearing in Proposition~\ref{prop:par_petal2}, Definition~\ref{defn:petal_at_inf} and Proposition~\ref{prop:attr_petal} suited for the attracting petals $\rp(P_i)$ of the map $\ap_i$ at $p_i$ and the compact sets $K_i$, $i \in \II$ (some other requirements for $n_0$ will be specified later). 
Let
\[
\tilde K_0 = Y \cup \bigcup_{i\in \II} \bigcup_{n=0}^{n_0-1}(\overline{W} \cap \ap_i^n(K_i) ).
\]
The set $\tilde K_0$ is a compact subset of $\overline{W} \subset V \setminus \{p_i\}_{i \in \II_{par}}$, and so is $\bd \tilde K_0$. Hence, by the assumption~\ref{item:crit} of Theorem~\ref{thm:metric}, the set $\bd \tilde K_0 \cap \PP_{crit}(\rp)$ consists of a finite number of points, which are isolated in $\PP_{crit}(\rp)$. Consequently, for a sufficiently small $\varepsilon_0 > 0$, the set
\[
\tilde K = \tilde K_0 \cup \bigcup_{z \in\bd \tilde K_0 \cap \PP_{crit}(\rp)}\overline{\D(z, \varepsilon_0)} 
\]
is a compact subset of $V \setminus \{p_i\}_{i \in \II_{par}}$ satisfying
\begin{equation}\label{eq:inV}
\bd \tilde K \cap \PP_{crit}(\rp) = \emptyset.
\end{equation}
By Lemma~\ref{lem:WinKi},
\begin{equation}\label{eq:WinKi}
(\overline{W} \setminus \tilde K) \setminus \{p_i\}_{i \in \II_{par}} \subset \bigcup_{i \in \II} \bigcup_{n=n_0}^\infty \ap_i^n(K_i) \subset \bigcup_{i \in \II} P_i.
\end{equation}

Another useful property of $\tilde K$ is described in the following lemma.

\begin{lem}\label{lem:z_ij} For $i \in \II$, there exist a finite number of points $z_{i,j} \in \overline{W_1} \cap \tilde K$, $j \in \JJ_i$, such that $\rp(z_{i,j}) = p_i$ and
\[
W \cap \tilde K \cap \rp^{-1}(W \setminus \tilde K) \subset \bigcup_{i \in \II}\bigcup_{j\in \JJ_i} \D(z_{i, j}, r)
\]
for some $r > 0$, where $\D(z_{i, j}, r)$ are pairwise disjoint for distinct $(i,j)$, and $r$ can be assumed to be arbitrarily small if $n_0$ is chosen large enough.
\end{lem}
\begin{proof}
First, we show
\begin{equation}\label{eq:F-1disjoint}
W \cap \tilde K \cap \rp^{-1}(W \setminus \tilde K) \cap \bigcup_{i \in \II} P_i = \emptyset.
\end{equation}
To prove \eqref{eq:F-1disjoint}, suppose there exists $z \in W \cap \tilde K \cap \rp^{-1}(W \setminus \tilde K) \cap P_i$ for some $i \in \II$. By Lemma~\ref{lem:WinKi} and the assumption~\ref{item:petals'}, we have $z \in \bigcup_{n= n_0+1}^\infty G^n_i(K_i) \subset \overline{G^{n_0+1}_i(P_i)}$
By the definition of $\tilde K$, there is a point $z_0 \in \tilde K_0$ with $|z-z_0| < \varepsilon_0$. Note that since $z$ is in $\bigcup_{n= n_0+1}^\infty G^n_i(K_i) \cap \rp(\tilde K)$, which is a compact subset of $\ap_i^{n_0}(P_i)$, we have $z_0 \in \ap_i^{n_0}(P_i)$, provided $\varepsilon_0$ is chosen sufficiently small according to $n_0$. On the other hand, by Lemma~\ref{lem:WinKi} and the definition of $K$, $z_0 \in \bigcup_{n= 0}^{n_0-1} G^n_i(K_i) \subset \overline{P_i} \setminus G^{n_0}_i(P_i)$, which makes a contradiction. 

Now the lemma easily follows from \eqref{eq:\rp-1disjoint}, the compactness of $K$, the assumption~\ref{item:extend} of Theorem~\ref{thm:metric} and the fact that $W \setminus \tilde K \cap P_i$ is arbitrarily close (in the spherical metric) to $p_i$ if $n_0$ is chosen large enough.
\end{proof}

\subsection{Orbifold metric}\label{subsec:metric_orbifold}

For $z \in V$ let $\nu(z)$ be equal to the least common multiple of the elements of the set $\{\deg_w \rp^n: w \in \rp^{-n}(z), \, n \in \N\}$. 
By the assumption~\ref{item:deg} of Theorem~\ref{thm:metric}, the function $\nu$ is well-defined. Note that
\begin{equation}\label{eq:nu-P}
\nu(z) > 1 \iff z \in \PP_{crit}(\rp)
\end{equation}
and
\begin{equation}\label{eq:multiple}
\nu(\rp(z)) \text{ is a multiple of }\nu(z)\deg_z \rp \qquad \text{for } z \in V'.
\end{equation}
By the assumption~\ref{item:crit} of Theorem~\ref{thm:metric}, the pair $(V, \nu)$ is a hyperbolic orbifold as defined in Subsection~\ref{subsec:orbifold}. 
Let $d\rho = \rho|dz|$ be the corresponding orbifold metric on $V$. A standard estimate of the density $\varrho_V$ of the hyperbolic metric $d\varrho_V$ on $V$ (see e.g.~\cite[Lemmas~2.1 and~2.3]{bfjk}) shows that for $z \in V$ we have 
\[
\varrho_V(z) \ge \frac{c}{|z - p_i| \ln |z - p_i|}
\]
for $i \in \II_{par}$, if $|z - p_i|$ is sufficiently small, and 
\[
\varrho_V(z) \ge \frac{c}{|z| \ln |z|}
\]
if $|z|$ is sufficiently large, for a constant $c > 0$. Hence, \eqref{eq:orb>hyp} and \eqref{eq:nu-P} imply that for every $0 < \delta < 1$ and $z \in V \setminus \PP_{crit}(\rp)$,
\begin{equation}\label{eq:orb>hyp2}
\lim_{z \to p_i} \frac{\rho(z)}{\sigma_{p_i, 1 - \delta}(z)} = \infty \quad \text{for } i \in \II_{par}, \qquad\lim_{|z| \to \infty} \frac{\rho(z)}{\sigma_{1 + \delta}(z)} = \infty \quad \text{for } i \in \II_\infty,
\end{equation}
where the metrics $d\sigma_{p_i, 1-\delta} = \sigma_{p_i, 1-\delta}|dz|$, $d\sigma_{1+\delta} = \sigma_{1+\delta}|dz|$ are defined, respectively, in Sections~\ref{sec:petals_at_par} and~\ref{sec:petals_at_inf}.

Now, we describe the expanding properties of the orbifold metric $d\rho$ with respect to $\rp$. 

\begin{lem}\label{lem:g-expand} The map $\rp$ is locally expanding on $V' \setminus \rp^{-1}(\PP_{crit}(\rp))$ with respect to $d\rho$. Moreover, for every compact subset $L$ of $V'$ there exists $Q>1$ such that $|\rp'|_\rho > Q$ on $L \setminus \rp^{-1}(\PP_{crit}(\rp))$.
\end{lem}
\begin{proof} 
Let $\pi\colon \D \to V$ be a universal branch covering of the orbifold $(V, \nu)$ (see Subsection~\ref{subsec:orbifold}). By \eqref{eq:multiple-orbifold} and \eqref{eq:multiple}, 
\begin{equation}\label{eq:deg_w}
\deg_u \pi \text{ is a multiple of } \deg_v (\rp\circ \pi) \; \text{ for } z \in V, \, w \in \rp^{-1}(z), \, u \in \pi^{-1}(z), \, v \in \pi^{-1}(w).
\end{equation}
Consider a branch $H$ of $\rp^{-1}$ defined on some simply connected domain $U \subset V \setminus \PP_{crit}(\rp)$. By \eqref{eq:nu-P}, $\nu(z) = \nu(H(z)) = 1$ for every $z \in U$. Hence, \eqref{eq:deg_w} implies that $H$ lifts to a holomorphic map $\tilde H \colon \tilde U \to \D$ for some simply connected domain $\tilde U \subset \D$. In fact, $\tilde H$ extends to a holomorphic map $\tilde H \colon \D \to \D$. To see it, we extend $\tilde H$ holomorphically as a branch of $(\rp \circ \pi)^{-1} \circ \pi$ along any curve $\gamma$ in $\D$ starting from a given point of $\tilde U$. Such extension exists by \eqref{eq:deg_w} and the fact that $\rp$ has no asymptotic values. Then by the simple connectedness of $\D$ we conclude that $\tilde H \colon \D \to \D$ is well-defined as a holomorphic map. Since $V' \ne V$, we have $\tilde H(\D) \subset \pi^{-1}(V') \neq \pi^{-1}(V) = \D$, so $\tilde H$ cannot be an isometry with respect to $d\varrho_{\D}$. Hence, by the Pick--Schwarz Lemma, $\tilde H$ is locally contracting on $\D$ with respect to $d\varrho_{\D}$, so 
\begin{equation}\label{eq:tildeH}
|\tilde H'(u)|_{\varrho_\D} < 1 \qquad \text{for} \quad u \in \D.
\end{equation}

Let $z \in V$. By the assumption~\ref{item:crit} of Theorem~\ref{thm:metric}, the set $\PP_{crit}(\rp)$ is discrete in $V$, so we can take a small open disc $U_z \subset V$ around $z$, such that $(U_z \setminus \{z\}) \cap \PP_{crit}(\rp) = \emptyset$. Let $\tilde U_z$ be a component of $\pi^{-1}(U_z)$ and let $w \in \rp^{-1}(z)$. 

Suppose first $z \notin \PP_{crit}(\rp)$. Then $U_z \cap \PP_{crit}(\rp) = \emptyset$, so there exists a branch $H_w$ of $\rp^{-1}$ defined on $U_z$, such that $H_w(z) = w$. As explained above, there exists a holomorphic extension $\tilde H_w\colon \D \to \D$ of the lift of $H_w$ to $\tilde U_z$, which is locally contracting on $\D$ with respect to $d\varrho_{\D}$. Thus, \eqref{eq:tildeH} gives
\[
\sup_{\tilde U_z}|\tilde H_w'|_{\varrho_\D} < q 
\]
for some $q \in (0, 1)$. Since $U_z \cap \PP_{crit}(\rp) = \emptyset$, the metric $\rho$ has no singularities on $U_z \cup H_w(U_z) = \pi(\tilde U_z) \cup \pi(\tilde H_w(\tilde U_z))$, so
\[
\inf_{H_w(U_z)}|\rp'|_\rho = \frac{1}{\sup_{U_z}|H_w'|_\rho} = \frac{1}{\sup_{\tilde U_z} |\tilde H'_w|_{\varrho_\D}} > \frac{1}{q}. 
\]
Hence, $|\rp'|_\rho > \frac{1}{q} > 1$ in a neighbourhood of $w$. This shows that $\rp$ is locally expanding on $V' \setminus \rp^{-1}(\PP_{crit}(\rp))$ with respect to $d\rho$.

Suppose now $z \in \PP_{crit}(\rp)$. Then there are a finite number of branches $H_{w,j}$ of $\rp^{-1}$, defined on some simply connected domains $U_{z,j} \subset U_z\setminus \{z\}$, such that $\bigcup_j H_{w,j}(U_{z,j})$ contains a punctured neighbourhood of $w$. Repeating the previous arguments and applying \eqref{eq:tildeH} to extensions $\tilde H_{w,j}$ of lifts of $H_{w,j}$ to some domains $\tilde U_{z,j} \subset \tilde U_z$, we obtain
\[
\sup_{\tilde U_{z,j}}|\tilde H_{w,j}'|_{\varrho_\D} < q_j
\]
for some $q_j \in (0, 1)$. As $(U_z \setminus \{z\}) \cap \PP_{crit}(\rp) = \emptyset$, the metric $\rho$ has no singularities on $U_{z,j} \cup H_{w,j}(U_{z,j})$, so
\[
\inf_{H_{w,j}(U_{z,j})}|\rp'|_\rho = \frac{1}{\sup_{U_{z,j}}|H_{w,j}'|_\rho} > \frac{1}{q_j}. 
\]
Hence, $|\rp'|_\rho > \min_j\frac{1}{q_j} > 1$ in a punctured neighbourhood of $w$.

We conclude that for every $w \in V'$ there exists a neighbourhood $U(w)$ of $w$ and a number $Q_w > 1$ such that $|\rp'|_\rho > Q_w$ on $U(w) \setminus \rp^{-1}(\PP_{crit}(\rp))$. This provides both assertions of the lemma.
\end{proof}

\subsection{Petal metric}\label{subsec:metric_petal}

Fix numbers $\alpha_{par} \in (0, 1)$, $\alpha_\infty \in (1, \infty)$ such that
\begin{equation}\label{eq:alpha}
\begin{aligned}
\alpha_{par} &> 1 - \frac{1}{\max_{i \in \II_{par}} \max_{j \in \JJ_i}\{\nu(z_{i,j}) \deg_{z_{i,j}} \rp\}},\\
\alpha_\infty &< 1 + \frac{1}{\max_{i \in \II_\infty} \max_{j \in \JJ_i}\{\nu(z_{i,j}) \deg_{z_{i,j}} \rp\}},
\end{aligned}
\end{equation}
with the convention that the maximum over an empty set is $1$. We will use the following estimate (recall that we denote $W_1 = W \cap \rp^{-1}(W)$).

\begin{lem}\label{lem:alpha}
\begin{alignat*}{3}
|\rp'(z)| \frac{\sigma_{p_i, \alpha_{par}}(\rp(z))}{\rho(z)} &\to \infty &\quad&\text{as } z \to z_{i, j}, \; z \in W_1 &\qquad&\text{for } \; i \in \II_{par}, \; j \in \JJ_i,\\
|\rp'(z)| \frac{\sigma_{\alpha_\infty}(\rp(z))}{\rho(z)} &\to \infty &\quad&\text{as } z \to z_{i, j}, \; z \in W_1 &\qquad&\text{for } \; i \in \II_\infty, \; j \in \JJ_i.
\end{alignat*}
\end{lem}
\begin{proof}
Take $i \in \II$, $j \in \JJ_i$ and $z \in W_1$ close to $z_{i,j}$. Let $d_{i,j} = \deg_{z_{i,j}} \rp$. We write $a(z) \asymp b(z)$ if $c_1 < \frac{a(z)}{b(z)} < c_2$ for some constants $c_1, c_2 > 0$. By \eqref{eq:orb_metric},
\[
\rho(z) \asymp \frac{1}{|z - z_{i,j}|^{1 - 1/\nu(z_{i,j})}}.
\]

Suppose first $i \in \II_{par}$. Then
\[
|\rp(z) - p_i| \asymp |z - z_{i,j}|^{d_{i,j}}, \qquad |\rp'(z)| \asymp |z-z_{i,j}|^{d_{i,j}-1},
\]
so
\[
\sigma_{p_i, \alpha_{par}}(\rp(z)) \asymp \frac{1}{|z - z_{i,j}|^{d_{i,j}\alpha_{par}}},
\]
and by \eqref{eq:alpha},
\[
|\rp'(z)| \frac{\sigma_{p_i, \alpha_{par}}(\rp(z))}{\rho(z)} \ge \frac{c}{|z - z_{i,j}|^{d_{i,j}(\alpha_{par} - 1) + 1/\nu(z_{i,j})}} \ge \frac{c}{|z - z_{i,j}|^\delta}
\]
for some $c, \delta > 0$. This shows the first assertion of the lemma. 

Suppose now $i \in \II_\infty$. Then
\[
|\rp(z)| \asymp \frac{1}{|z - z_{i,j}|^{d_{i,j}}}, \qquad |\rp'(z)| \asymp \frac{1}{|z-z_{i,j}|^{d_{i,j}+1}},
\]
so
\[
\sigma_{p_i, \alpha_\infty}(\rp(z)) \asymp |z - z_{i,j}|^{d_{i,j}\alpha_\infty}, 
\]
and \eqref{eq:alpha} gives
\[
|\rp'(z)| \frac{\sigma_{p_i, \alpha_\infty}(\rp(z))}{\rho(z)} \ge \frac{c}{|z - z_{i,j}|^{d_{i,j}(1 - \alpha_\infty) + 1/\nu(z_{i,j})}} \ge \frac{c}{|z - z_{i,j}|^\delta}
\]
for some $c, \delta > 0$. This gives the second assertion of the lemma. 
\end{proof}

Assume that the number $n_0$ is larger than all the numbers $n_0$ appearing in Proposition~\ref{prop:der-par} and Theorem~\ref{thm:der-inf}, suited for the attracting petals $\rp(P_i)$ of $\ap_i$ at $p_i$ and the compact sets $K_i$, $i \in \II$, with $\alpha = \alpha_{par}$ for $i \in \II_{par}$ and $\alpha = \alpha_\infty$ for $i \in \II_\infty$. For $i \in \II_\infty$ let $g_i$ be the function $g$ from Definition~\ref{defn:petal_at_inf}, suited for the attracting petal $\rp(P_i)$ of $\ap_i$ at $p_i$ and the compact set $K_i$. Let also
\begin{equation}\label{eq:A_infty}
A_\infty = 2 \max_{i \in \II_\infty} A_i, 
\end{equation}
where $A_i$ is the constant $A$ appearing in Theorem~\ref{thm:der-inf}, suited for the attracting petal $\rp(P_i)$ of $\ap_i$ at $p_i$, the compact set $K_i$ and $\alpha = \alpha_\infty$.

By Definition~\ref{defn:petal_at_inf} and Proposition~\ref{prop:attr_petal}, there exists $c_0 > 0$ such that 
\begin{equation}\label{eq:F-z}
|z| - |\rp(z)| \ge c_0 g_i(|\rp(z)|) \quad \text{for }z \in  \bigcup_{n=n_0}^\infty \ap_i^n(K_i), \quad i \in \II_\infty.
\end{equation}

Fix a large constant $C_1 > 0$. By \eqref{eq:orb>hyp2}, we can find an arbitrarily large $R^- > 0$ such that 
\begin{equation}\label{eq:R-}
\tilde K \subset \D(0, R^-)
\end{equation}
and
\begin{equation}\label{eq:R1}
\rho (z) > A_\infty^\frac{C_1}{c_0} \sigma_{\alpha_\infty}(z) \qquad \text{for} \quad z \in V \setminus (\D(0, R^-) \cup \PP_{crit}(\rp)).
\end{equation}

For $i \in \II_\infty$ consider the function $t \mapsto t - R^- - C_1 g_i(t)$, $t \in [R^-, +\infty)$. It is negative for $t = R^-$ and tends to $+\infty$ as $t \to +\infty$ since $g_i$ are bounded by definition. Hence, it attains zero at some point $t = R^+_i > R^-$, so that
\[
R^+_i = R^- + C_1 g_i(R^+_i).
\]
Define
\[
h_i(t) = 
\begin{cases}
A_\infty^\frac{C_1}{c_0} + C_2(R^- - t)&\text{for }t \in [0,R^-)\\
A_\infty^{\frac{R^+_i - t}{c_0 g_i(R^+_i)}} &\text{for }t \in [R^-, R^+_i)\\
1 &\text{for } t \in [R^+_i, +\infty)\\
\end{cases}
\]
for a large constant $C_2 > 0$. Then $h_i \colon [0, +\infty) \to \R$ is a positive continuous non-increasing function, such that
\begin{equation}\label{eq:max_h}
\min_{[0, R^-]}h_i = \max_{[R^-, +\infty)}h_i = h_i(R^-) = A_\infty^\frac{C_1}{c_0}. 
\end{equation}
By \eqref{eq:F-z}, if $z \in\bigcup_{n=n_0}^\infty \ap_i^n(K_i)$ and $|z| \le R^+_i$, $|\rp(z)| \ge R^-$, then 
\[
|z| - |\rp(z)| \ge c_0 g_i(|\rp(z)|) \ge c_0 g_i(R^+_i),
\]
and by the definition of $h_i$,
\[
\frac{h_i(|\rp(z)|)}{h_i(|z|)} \geq A_\infty.
\]
Therefore, choosing $C_2$ sufficiently large, by compactness we can assume
\begin{equation}\label{eq:>A}
\frac{h_i(|\rp(z)|)}{h_i(|z|)} \geq A_\infty\qquad \text{for } \; z \in \overline{W \setminus \tilde K} \cap  P_i \cap \overline{\D(0, R^+_i)}, \quad i \in \II_\infty.
\end{equation}

For $i \in \II_\infty$ let
\[
R_i= R^+_i - \ln C_1 \;g_i(R^+_i).
\]
Obviously, 
\[
R^- < R_i < R^+_i.
\]
Moreover, the following holds.

\begin{lem}\label{lem:R_i}
If $z \in \bigcup_{n=n_0}^\infty \ap_i^n(K_i)$, $i \in \II_\infty$, and $|\rp(z)| \le R_i < |z|$, then
\[
R^- \le |\rp(z)| < |z| \le R^+_i.
\]
\end{lem}
\begin{proof} Take $z \in \ap_i^n(K_i)$ for some $i \in \II_\infty$, $n\ge n_0$.
By Lemma~\ref{lem:new}, considering the sequence $(t_n)_{n=1}^\infty$ for the function $g_i$, we find $j \in \N$ such that $t_j \le |\rp(z)| < |z| \le t_{j+M}$ for a constant $M \in \N$. 

Suppose $|z| > R^+_i$. Then by Definition~\ref{defn:petal_at_inf}, Proposition~\ref{prop:attr_petal} and Lemma~\ref{lem:log_conv},
\[
\ln C_1 \: g_i(R^+_i) = R^+_i - R_i < |z| - |\rp(z)| \le c_2 g_i(|\rp(z)|) \le c_2 g_i(t_j) \le c_3 g_i(t_{j+ M}) \le c_3 g_i(R^+_i)
\]
for some constants $c_2, c_3 >0$ (independent of $C_1$), which is impossible if $C_1$ was chosen sufficiently large. Therefore, $|z| \leq R^+_i$. 

Suppose now $|\rp(z)| < R^-$. Then, analogously as previously, we obtain 
\begin{equation}\label{eq:lnD}
\begin{aligned}
(C_1 - \ln C_1) g_i(R^+_i) &= R_i - R^- < |z| - |\rp(z)| \le c_2 g_i(|\rp(z)|)\\
&\le c_2 g_i(t_j) \le c_3 g_i(t_{j+ M}) \le c_3 g_i(R_i).
\end{aligned}
\end{equation}
Take the maximal $j_0 \in\N$ and the minimal $N \ge 0$ such that
\[
t_{j_0} \le R_i < R^+_i \le t_{j_0 + N}.
\]
If $N > 2$, then $t_{j_0} \le R_i < t_{j_0+1} < \cdots < t_{j_0+N-1} < R^+_i \le t_{j_0 + N}$, so by the definition of the sequence $(t_n)_{n=1}^\infty$,
\begin{align*}
\ln C_1 \: g_i(R^+_i) &= R^+_i - R_i \ge t_{j_0+N-1} - t_{j_0+1} \ge t_{j_0 + N-1} - t_{j_0 + N - 2} + \cdots + t_{j_0 + 2} - t_{j_0+1}\\
&= g_i(t_{j_0 + N-2}) + \cdots + g_i(t_{j_0+1}) \ge (N-2) g_i(R^+_i).
\end{align*}
This shows $N \le 2+\ln C_1$. By Lemma~\ref{lem:log_conv}, assuming that $n_0$ is chosen sufficiently large, we have $\frac{g_i(t_n)}{g_i(t_{n+1})} < e^{\frac 1 2}$ for $n \ge j_0$, so 
\[
\frac{g_i(R_i)}{g_i(R^+_i)} \le \frac{g_i(t_{j_0})}{g_i(t_{j_0+N})} = \frac{g_i(t_{j_0})}{g_i(t_{j_0+1})} \cdots \frac{g_i(t_{j_0+N-1})}{g_i(t_{j_0+N})} \le e^{\frac N 2} \le e^{1+\frac{\ln C_1}{2}} = e \, C_1^{\frac 1 2}.
\]
This together with \eqref{eq:lnD} gives
\[
(C_1 - \ln C_1) g_i(R^+_i) \le ec_3 C_1^{\frac 1 2} g_i(R^+_i),
\]
which is impossible if $C_1$ was chosen sufficiently large. Therefore, $|\rp(z)| \ge R^-$.
\end{proof}

Let 
\begin{alignat*}{2}
\varsigma_i(z) &= C\sigma_{p_i, \alpha_{par}}(z) &\quad&\text{for } \;z \in  \overline{W \setminus \tilde K} \cap P_i, \quad i\in \II_{par},\\
\varsigma_\infty(z) &= h_i(|z|)\sigma_{\alpha_\infty}(z) &\quad&\text{for }\; z \in \overline{W \setminus \tilde K} \cap P_i,\quad i\in \II_\infty,
\end{alignat*}
where $C >0$ is a large constant. 

\begin{lem}\label{lem:sigma-expand}
The map $\rp$ is locally expanding with respect to $d\varsigma_i$ on $\overline{W_1 \setminus \tilde K} \cap P_i$, $i\in \II_{par}$, and $|\rp'|_{\varsigma_\infty} > 2$ on $\bigcup_{i\in \II_\infty} \big( \overline{W_1 \setminus \tilde K} \cap P_i \cap \overline{\D(0, R^+_i)}\big)$.
\end{lem}
\begin{proof}
Consider $i \in \II_{par}$. By \eqref{eq:WinKi} and Proposition~\ref{prop:der-par}, suited for the attracting petal $\rp(P_i)$ of the map $\ap_i$ at $p_i$, the compact set $K_i$, $\alpha = \alpha_{par}$ and $m = n_0$, we have $|\rp'(z)|_{\sigma_{p_i, \alpha_{par}}} > 1$ for $z \in \overline{W_1 \setminus \tilde K} \cap P_i$. By the definition of $\varsigma_i$ and \eqref{eq:WinKi}, we conclude that $\rp$ is locally expanding with respect to $d\varsigma_i$ on $\overline{W_1 \setminus \tilde K} \cap P_i$.

Assume now $i \in \II_\infty$. By Theorem~\ref{thm:der-inf} suited for the attracting petal $\rp(P_i)$ of the map $\ap_i$ at infinity, the compact set $K_i$, $\alpha = \alpha_\infty$ and $m = n_0$,
we obtain $|\rp'(z)|_{\sigma_{\alpha_\infty}} > 2/A_\infty$ for $z \in \overline{W_1 \setminus \tilde K} \cap P_i$. Using \eqref{eq:>A}, the definition of $\varsigma_\infty$ and \eqref{eq:WinKi}, we see $|\rp'|_{\varsigma_\infty} > 2$ on $\overline{W_1 \setminus \tilde K} \cap P_i \cap \overline{\D(0, R^+_i)}$.
\end{proof}

\subsection{Construction of expanding metric}\label{subsec:metric_metric}

Now we construct a suitable conformal metric $d\varsigma$ on $W_N$ for a large $N$. By the assumption~\ref{item:crit} of Theorem~\ref{thm:metric}, the set $\tilde K$ contains a finite number of points from $\PP_{crit}(\rp)$. Moreover, by the assumptions of Theorem~\ref{thm:metric}, we have $W_{n+1} \cap \tilde K \subset W_n \cap \tilde K$ for $n \ge 0$ and $\bigcap_{n=0}^\infty W_n \cap \tilde K = \emptyset$, so we can find a number $N \in \N$ such that
\begin{equation}\label{eq:KcapP}
W_N \cap \tilde K \cap \PP_{crit}(\rp) = \emptyset.
\end{equation}
Note that if $W \cap \PP_{crit}(\rp)  = \emptyset$ then we can set $N = 0$, which proves the last assertion of Proposition~\ref{prop:expanding_metric}.

Let
\[
\widehat K = \tilde K \cup \bigcup_{i \in \II_{par}}\bigcup_{n=n_0}^{m_0} (\overline{W} \cap \ap_i^n(K_i)) \cup \bigcup_{i \in \II_\infty} \big(\overline{W} \cap P_i \cap \overline{\D(0, R_i)}\big)
\]
for a large $m_0 > n_0$.
Then $\widehat K$ is a compact subset of $V \setminus \{p_i\}_{i \in \II}$. Note that choosing $m_0$ and $R_i$ sufficiently large and using \eqref{eq:WinKi}, we can assume $\widehat K \supset K$ for an arbitrary given compact set $K \subset \overline{W} \setminus \{p_i\}_{i \in \II_{par}}$. 

We define a conformal metric $d\varsigma = \varsigma|dz|$ on $W_N$, setting
\[
\varsigma = 
\begin{cases}
\rho &\text{on } W_N \cap \tilde K\\
\min(\rho, \varsigma_i) &\text{on } (W_N \setminus \tilde K) \cap P_i  \setminus \PP_{crit}(\rp), \quad i \in \II_{par}\\
\varsigma_i &\text{on } (W_N \setminus \tilde K) \cap P_i  \cap \PP_{crit}(\rp), \quad i \in \II_{par}\\
\min(\rho, \varsigma_\infty) &\text{on } \bigcup_{i \in \II_\infty} \big( (W_N \setminus \tilde K) \cap P_i \big) \setminus \PP_{crit}(\rp)\\
\varsigma_\infty &\text{on } \bigcup_{i \in \II_\infty} \big( (W_N \setminus \tilde K) \cap P_i \big) \cap \PP_{crit}(\rp)
\end{cases}.
\]
To show that $d\varsigma$ is a conformal metric on $W_N$, note first that by \eqref{eq:KcapP}, \eqref{eq:nu-P} and \eqref{eq:WinKi}, the function $\rho$ has no singularities in $W_N \cap \tilde K$ and $\varsigma$ is well-defined and positive on $W_N$. It is obvious that $\varsigma$ is continuous on $W_N \setminus \big(\bigcup_{i \in \II} (\bd\tilde K \cap P_i) \cup \PP_{crit}(\rp)\big)$. Observe also that for $i \in \II$, by \eqref{eq:inV} and \eqref{eq:nu-P}, we have $\nu(z) = 1$ on the compact subset $\overline{W_N} \cap \bd\tilde K \cap \overline{P_i}$ of $V$ and hence $\rho$ is well-defined and bounded on this set. Consequently, for $i \in \II_{par}$ we have $\rho < \varsigma_i$ on $W_N \cap \bd\tilde K \cap P_i$ provided $C$ is chosen sufficiently large, which implies that
$\varsigma = \rho$ on $W_N \cap \bd\tilde K \cap P_i$ and hence
$\varsigma$ is continuous on $W_N \cap \bd\tilde K \cap P_i$. Similarly, for $i \in \II_\infty$, \eqref{eq:R-} and \eqref{eq:max_h} imply that 
$\rho < \varsigma_\infty$ on $W_N \cap \bd\tilde K \cap P_i$ provided $C_1$ is sufficiently large, so $\varsigma = \rho$ on $W_N \cap \bd\tilde K \cap P_i$ and $\varsigma$ is continuous on $W_N \cap \bd\tilde K \cap P_i$. Furthermore, if $z_0 \in W_N \cap \PP_{crit}(\rp)$, then $z_0 \in (W_N \setminus \tilde K) \cap P_i$ for some $i \in \II$, so by \eqref{eq:nu-P} and \eqref{eq:orb_metric}, we have $\rho(z) \to +\infty$ as $z \to z_0$, which implies that in a punctured neighbourhood of $z_0$ there holds $\varsigma = \varsigma_i$ if $i \in \II_{par}$ and $\varsigma = \varsigma_\infty$ if $i \in \II_\infty$. This shows that $\varsigma$ is continuous at $z_0$. We conclude that $\varsigma$ is well-defined, positive and continuous on $W_N$, so $d\varsigma$ is a conformal metric on $W_N$. 

Notice that by \eqref{eq:orb>hyp2},
\begin{equation}\label{eq:rho>sigma-par}
\rho > \varsigma_i \quad \text{on }  \bigcup_{n=m_0}^\infty G_i(K_i) \setminus \PP_{crit}(\rp) \supset  \big( \overline{W \setminus \widehat K} \cap P_i \big) \setminus \PP_{crit}(\rp), \quad i \in \II_{par},
\end{equation}
if $m_0$ was chosen sufficiently large. Similarly, by \eqref{eq:R1} and \eqref{eq:max_h}, we have
\begin{equation}\label{eq:rho>sigma-inf}
\rho > \varsigma_\infty \quad \text{on } \bigcup_{i \in \II_\infty} \big( \overline{W \setminus \tilde K} \cap P_i \big) \setminus (\D(0, R^-) \cup \PP_{crit}(\rp)) \supset \bigcup_{i \in \II_\infty} \big( \overline{W \setminus \widehat K} \cap P_i \big) \setminus \PP_{crit}(\rp).
\end{equation}

In the further considerations, we will need the following lemma.

\begin{lem}\label{lem:varsigma-limit}
If $z_n \in W_N$ and $\varsigma(z_n) \to 0$, then $z_n \to \infty$.
\end{lem}
\begin{proof}
Suppose $z_n \in W_N$, $\varsigma(z_n) \to 0$ and $z_n \not\to \infty$. Passing to a subsequence, we can assume $z_n \to z \in \overline{W_N}$ and one of the three cases appears:
\begin{enumerate}[(i)]
\item $\varsigma (z_n) = \rho(z_n)$ for all $n$,
\item $\varsigma (z_n) = \varsigma_i(z_n)$ for all $n$ and some fixed $i \in \II_{par}$,
\item $\varsigma (z_n) = \varsigma_\infty(z_n)$ for all $n$.
\end{enumerate}
In the case (i), we have $z_n \in W_N \cap \widehat K$ by \eqref{eq:rho>sigma-par} and \eqref{eq:rho>sigma-inf}, so $z \in \overline{W_N} \cap \widehat K \subset V$. If $z \notin \PP_{crit}(\rp)$, then $\varsigma(z_n) = \rho(z_n)\to \rho(z) > 0$, and if $z \in \PP_{crit}(\rp)$, then $\varsigma(z_n) = \rho(z_n)\to \infty$ by \eqref{eq:nu-P} and \eqref{eq:orb_metric}. Both possibilities lead to a contradiction.

In the case~(ii), there holds $\varsigma (z_n) = \varsigma_i(z_n)> c$ for some constant $c >0$ by the definition of $\varsigma_i$ and the fact that $z_n \in (W_N \setminus \tilde K) \cap P_i$, so $z_n$ lies in a small neighbourhood of $p_i$. Again, this is impossible. Finally, in the case~(iii), $z_n$ is in a small neighbourhood of infinity and $\varsigma (z_n) = \varsigma_\infty(z_n)\to \varsigma_\infty(z) > 0$. This makes a contradiction.
\end{proof}

Now we show expanding properties of the metric $d\varsigma$. 

\begin{lem} \label{lem:g-sigma}
The map $\rp$ is locally expanding on 
\[
W_{N+1} \setminus \bigcup_{i \in \II_\infty} \big(P_i \setminus \overline{\D(0, R_i^+)}\big)
\]
with respect to $d\varsigma$. Moreover, there exists $Q > 1$ such that $|\rp'|_\varsigma > Q$ on $W_{N+1} \cap \widehat K$.
\end{lem}
\begin{proof}
Note first that by the definition of $\varsigma$, \eqref{eq:rho>sigma-par} and \eqref{eq:rho>sigma-inf}, we have $\varsigma \le \rho$ on $W_N \cap \widehat K \setminus \PP_{crit}(\rp)$ and $\varsigma < \rho$ on $W_N \setminus (\widehat K \cup \PP_{crit}(\rp))$. In view of this and \eqref{eq:KcapP}, it is straightforward to check that for $z \in W_{N+1}$ there holds one of the three following (non-necessarily disjoint) cases.

\begin{enumerate}[$($i$)$]
\item $z \notin \rp^{-1}(\PP_{crit}(\rp))$, \; $\varsigma(z) \le \rho(z)$, \; $\rp(z) \in W_N \cap \widehat K$,\; $\varsigma(\rp(z)) = \rho(\rp(z))$, 
\item $z \in \tilde K$, \; $\rp(z) \in (W_N \setminus \tilde K) \cap P_i$, \; $\varsigma(\rp(z)) = 
\begin{cases}
\varsigma_i(\rp(z))&\text{for some } i \in \II_{par}\\
\varsigma_\infty(\rp(z))&\text{for some } i \in \II_\infty
\end{cases}$,
\item $z \notin \tilde K$,\; $\rp(z) \in (W_N \setminus \tilde K) \cap P_i$, \; $\varsigma(\rp(z)) = 
\begin{cases}
\varsigma_i(\rp(z))&\text{for some } i \in \II_{par}\\
\varsigma_\infty(\rp(z))&\text{for some } i \in \II_\infty
\end{cases}$.

\end{enumerate}

We show the first part of the lemma, considering successively the cases (i)--(iii). Take $z \in W_{N+1} \setminus \bigcup_{i \in \II_\infty} \big(P_i \setminus \overline{\D(0, R_i^+)}\big)$. 
In the case~(i) we have $|\rp'(z)|_\varsigma \ge |\rp'(z)|_\rho > 1$ by Lemma~\ref{lem:g-expand}. In the case~(ii),
\[
|\rp'(z)|_\varsigma = \left.
\begin{cases}\displaystyle
|\rp'(z)|\frac{\varsigma_i(\rp(z))}{\rho(z)} = |\rp'(z)|\frac{C\sigma_{p_i, \alpha_{par}}(\rp(z))}{\rho(z)}&\text{for some } i \in \II_{par}\\\displaystyle
|\rp'(z)|\frac{\varsigma_\infty(\rp(z))}{\rho(z)} = |\rp'(z)|\frac{h_i(|z|)\sigma_{\alpha_\infty}(\rp(z))}{\rho(z)} &\text{for some } i \in \II_\infty
\end{cases}
\right\}
> 2
\]
by Lemma~\ref{lem:z_ij} (where we can make $r$ arbitrarily small by enlarging $n_0$), Lemma~\ref{lem:alpha} and the fact $h_i \ge 1$. In the case~(iii), by \eqref{eq:WinKi} and since $W \cap \rp(P_i) \cap P_{i'} = \emptyset$ for $i, i' \in \II$, $i\neq i'$ by the assumption~\ref{item:petals'}, there holds $z, \rp(z) \in P_i$ for some $i \in \II$. If $i \in \II_{par}$, then $\varsigma(z) \le \varsigma_i(z)$ by the definition of $\varsigma$, so $|\rp'(z)|_\varsigma \ge |\rp'(z)|_{\varsigma_i} > 1$ by Lemma~\ref{lem:sigma-expand}. Similarly, if $i \in \II_\infty$, then $\varsigma(z) \le \varsigma_\infty(z)$ by the definition of $\varsigma$, so $|\rp'(z)|_\varsigma \ge |\rp'(z)|_{\varsigma_\infty} > 2$ by Lemma~\ref{lem:sigma-expand}. 

Now we prove the second part of the lemma. Again, we examine the cases (i)--(iii), using the previous considerations. In the case~(i), to use Lemma~\ref{lem:g-expand}, we show that $W_{N+1} \cap \widehat K \cap \rp^{-1}(W_N \cap \widehat K)$ is contained in a compact subset $L$ of $V'$. Suppose it is not the case. Then there exists a sequence of points $z_n \in W_{N+1}$ such that $z_n \to z \in \overline{W_{N+1}} \cap \widehat K \cap \bd V'$ and $\rp(z_n) \to w \in \overline{W_N} \cap \widehat K \subset V$. By the assumption~\ref{item:extend} of Theorem~\ref{thm:metric}, the map $\rp$ extends holomorphically to a small disc $D$ centered at $z$, such that $\rp(D) \subset V$. Then taking $z_n \in D$, we can find a curve $\gamma \colon [0,+\infty) \to V' \cap D$ with $\gamma(0) = z_n$, $\lim_{t\to +\infty} \gamma(t) \to z'$ and $\lim_{t\to +\infty} \rp(\gamma(t)) \to w'$ for some $z' \in \bd V'\cap D$ and $w' \in V$, which shows that $w'$ is an asymptotic value of $\rp$ and contradicts the assumption~\ref{item:asympt_values}. Hence, $W_{N+1} \cap \widehat K\cap \rp^{-1}(W_N \cap \widehat K)$ is contained in a compact set $L \subset V'$, so by Lemma~\ref{lem:g-expand}, there exists $Q_1>1$ such that $|\rp'|_\rho > Q_1$ on $W_{N+1} \cap \widehat K\cap \rp^{-1}(W_N \cap \widehat K)$, in particular $|\rp'(z)|_\rho > Q_1$ for $z \in W_{N+1}\cap \widehat K$ fulfilling the condition~(i). In the case~(ii), we have already showed that for $z \in W_{N+1}\cap \widehat K$ there holds $|\rp(z)'|_\rho > Q_2$ for $Q_2 = 2$. In the case~(iii) we have $|\rp'(z)|_\rho > Q_3$ for $z \in W_{N+1}\cap \widehat K$ with some $Q_3 > 1$ by Lemma~\ref{lem:sigma-expand}, since 
$\bigcup_{i \in \II_{par}} \big( \overline{W_{N+1} \cap \widehat K \setminus \tilde K} \cap P_i\big)$ is a compact subset of $\bigcup_{i \in \II_{par}} \big(\overline{W_1 \setminus \tilde K} \cap P_i\big)$ and $\bigcup_{i \in \II_\infty} \big( \overline{W_{N+1} \cap \widehat K \setminus \tilde K} \cap P_i\big)$ is a subset of $\bigcup_{i \in \II_\infty}\big(\overline{W_1 \setminus \tilde K} \cap P_i \cap \overline{\D(0, R^+_i)}\big)$.
This shows that the second assertion of the lemma holds with $Q = \min(Q_1, Q_2, Q_3)$.
\end{proof}

Now we can prove the following fact, which completes the proof of Proposition~\ref{prop:expanding_metric} in the case when all the parabolic periodic points $p_i$, $i \in \II_{par}$, are fixed under $\rp$, of multipliers $1$.

\begin{lem}\label{lem:der_i} There exists a decreasing sequence $(\beta_n)_{n=0}^\infty$ of positive numbers, such that 
\begin{enumerate}[$($a$)$]
\item $\sum_{n=0}^\infty \beta_n < \infty$,
\item $\beta_0 = 1$ and $\frac{\beta_{n + 1}}{\beta_n} \ge \frac{\beta_n}{\beta_{n-1}}$ for every $n \in \N$,
\item $|(\rp^n)'(z)|_\varsigma > \frac{1}{\beta_n}$ for every $z \in (W_N \setminus \widehat K) \cap\rp^{-1} (W_N \setminus \widehat K) \cap \ldots \cap \rp^{-(n-1)}(W_N \setminus \widehat K) \cap \rp^{-n}(W_N \cap \widehat K)$, $n \in \N$.
\end{enumerate}
\end{lem}
\begin{proof} 
For $i \in \II_{par}$ and $m \ge n_0$, let $(\beta_{m,n}^{(i)})_{n=0}^\infty$ be the sequence $(a_{m,n})_{n=0}^\infty$ from Proposition~\ref{prop:der-par}, suited for the attracting petal $\rp(P_i)$ of the map $\ap_i$ at $p_i$, the compact set $K_i$ and  $\alpha = \alpha_{par}$. Then $\beta_{m,n}^{(i)}>0$, the sequence $(\beta_{m,n}^{(i)})_{n=0}^\infty$ is decreasing, $\sum_{n=0}^\infty \beta_{m,n}^{(i)} < \infty$, $\beta_{m,0}^{(i)} = 1$, $\frac{\beta_{m,n+1}^{(i)}}{\beta_{m,n}^{(i)}} \ge \frac{\beta_{m,n}^{(i)}}{\beta_{m,n-1}^{(i)}}$ for every $n \in \N$ and
\begin{equation}\label{eq:>Q1}
|(\rp^n)'(z)|_{\sigma_{p_i, \alpha_{par}}} > \frac{1}{\beta_{m,n}^{(i)}} 
\end{equation}
for $z \in \ap_i^{m+n}(K_i)$, $n \in \N$.

Note that by \eqref{eq:WinKi} and the compactness of $\widehat K$, we have 
\begin{equation}\label{eq:widehatKminustildeK}
(W_N \cap \widehat K \setminus \tilde K) \cap \bigcup_{i\in \II_\infty} P_i \subset \bigcup_{i\in \II_\infty} \bigcup_{n=n_0}^{m_1} G^n_i(K_i)
\end{equation}
for some $m_1 \in \N$. For $i \in \II_\infty$ and $n_0 \le m \le m_1$, let $(\beta^{(i)}_{m,n})_{n=1}^\infty$ be the sequence $(a_{m,n})_{n=1}^\infty$ appearing in Theorem~\ref{thm:der-inf}, suited for the attracting petal $\rp(P_i)$ of $\ap_i$ at $p_i$, the compact set $K_i$, $\alpha = \alpha_\infty$ and $\varepsilon = \frac{1}{2}$. Then $\beta^{(i)}_{m,n} > 0$, the sequence $(\beta^{(i)}_{m,n})_{n=1}^\infty$ is decreasing, $\sum_{n=1}^\infty \beta^{(i)}_{m,n} < \infty$, $\beta^{(i)}_{m,1} \le \frac{A_\infty}{2}$ (where $A_\infty$ was defined in \eqref{eq:A_infty}), $\frac{\beta^{(i)}_{m,n+1}}{\beta^{(i)}_{m,n}} \ge \frac{\beta^{(i)}_{m,n}}{\beta^{(i)}_{m,n-1}} > \frac{1}{2}$ for every $n > 1$ and 
\begin{equation}\label{eq:>Q2}
|(\rp^n)'(z)|_{\sigma_{\alpha_\infty}} > \frac{1}{\beta^{(i)}_{m,n}}
\end{equation}
for $z \in \ap_i^{m+n}(K_i)$, $n \in \N$. Let
\[
\widehat\beta^{(i)}_{m,n} = 
\begin{cases}
1 & \text{for } n = 0 \\
\frac{\beta^{(i)}_{m,n}}{A_\infty}& \text{for } n \in \N.
\end{cases}.
\]
Then $\widehat\beta^{(i)}_{m,n} > 0$, the sequence $(\widehat\beta^{(i)}_{m,n})_{n=0}^\infty$ is decreasing, $\sum_{n=0}^\infty \widehat\beta^{(i)}_{m,n} < \infty$, $\widehat\beta^{(i)}_{m,0} =1$ and $\frac{\widehat\beta^{(i)}_{m,n+1}}{\widehat\beta^{(i)}_{m,n}} \ge \frac{\widehat\beta^{(i)}_{m,n}}{\widehat\beta^{(i)}_{m,n-1}}$ for every $n \in \N$, since
\[
\frac{\widehat\beta^{(i)}_{m,2}}{\widehat\beta^{(i)}_{m,1}} = \frac{\beta^{(i)}_{m,2}}{\beta^{(i)}_{m,1}} > \frac{1}{2} \ge \frac{\beta^{(i)}_{m,1}}{A_\infty} = \widehat\beta^{(i)}_{m,1} = \frac{\widehat\beta^{(i)}_{m,1}}{\widehat\beta^{(i)}_{m,0}}. 
\]
Finally, for $n \ge 0$ let 
\[
\beta_n = \max\Big(\max_{i \in \II_{par}} \beta_{m_0,n}^{(i)}, \max_{i \in \II_\infty}\max_{m \in\{n_0, \ldots, m_1\}} \widehat\beta^{(i)}_{m,n}\Big).
\]
for $m_0$ from the definition of $\widehat K$. Obviously, $(\beta_n)_{n=0}^\infty$ is a decreasing sequence of positive numbers and $\sum_{n=0}^\infty \beta_n < \infty$. This gives the assertion~(a). The assertion~(b) follows from the analogous properties for the sequences $(\beta_{m_0,n}^{(i)})_{n=0}^\infty$ and $(\widehat\beta^{(i)}_{m,n})_{n=0}^\infty$.

To show the assertion~(c), take $z \in (W_N \setminus \widehat K) \cap \rp^{-1}(W_N \setminus \widehat K) \cap \ldots \cap \rp^{-(n-1)}(W_N \setminus \widehat K) \cap \rp^{-n}(W_N \cap \widehat K)$, $n \in \N$. Then by \eqref{eq:WinKi}, \eqref{eq:widehatKminustildeK}, \eqref{eq:K_i_disjoint} and Lemma~\ref{lem:R_i}, we have $z \in \ap_i^{m_0+n}(K_i)$ for some $i \in \II_{par}$ or $z \in \ap_i^{m+n}(K_i)$ for some $i \in \II_\infty$, $m \in \{n_0, \ldots, m_1\}$ and $R^- \le |\rp^n(z)| < |\rp^{n-1}(z)| \le R_i^+$. In the first case, \eqref{eq:rho>sigma-par} and \eqref{eq:>Q1} imply
\[
|(\rp^n)'(z)|_\varsigma = 
|(\rp^n)'(z)|_{\varsigma_i} = |(\rp^n)'(z)|_{C\sigma_{p_i, \alpha_{par}}} = |(\rp^n)'(z)|_{\sigma_{p_i, \alpha_{par}}} > \frac{1}{\beta_{m_0,n}^{(i)}} \ge \frac{1}{\beta_n}, 
\]
which gives the assertion~(c). In the second case, by \eqref{eq:rho>sigma-inf}, \eqref{eq:>A} and \eqref{eq:>Q2},
\begin{align*}
|(\rp^n)'(z)|_\varsigma &= |(\rp^n)'(z)|_{\varsigma_\infty} = 
\frac{h_i(|\rp^n(z)|)}{h_i(|z|)}|(\rp^n)'(z)|_{\sigma_{\alpha_\infty}}\\
&\ge \frac{h_i(|\rp^n(z)|)}{h_i(|\rp^{n-1}(z)|)}|(\rp^n)'(z)|_{\sigma_{\alpha_\infty}} \ge A_\infty \: |(\rp^n)'(z)|_{\sigma_{\alpha_\infty}}\\
&> \frac{A_\infty}{\beta^{(i)}_{m,n}} = \frac{1}{\widehat\beta^{(i)}_{m,n}} \ge \frac{1}{\beta_n}, 
\end{align*}
providing the assertion~(c).

\subsection{The case of parabolic periodic points} \label{subsec:periodic}

In this subsection we describe how to modify the construction described in Subsections~\ref{subsec:metric_petal_dyn}--\ref{subsec:metric_metric}, when the points $p_i$, $i \in \II_{par}$ are arbitrarily parabolic periodic points of $\rp$. Recall that in this case, apart from repelling petals $P_i$, $i \in \II_\infty$, of $\rp$ at infinity, there is a finite number of repelling petals $P_i$, $i \in \II_{par}$, of $\rp^{\ell_i}$, for some $\ell_i \in \N$, at parabolic periodic points $p_i \in \bd W$ of $\rp$, generating disjoint cycles of length $\ell_i$ of repelling petals of $\rp$ at $p_i$, such that
\[
Y = \overline{W} \setminus \Big(\bigcup_{i \in \II_{par}}\bigcup_{s=0}^{\ell_i-1} \rp^s(P_i \cup \{p_i\}) \cup \bigcup_{i \in \II_\infty} P_i\Big)
\]
is compact. Now, we shortly explain the modifications of the proof of Proposition~\ref{prop:expanding_metric} in this case. Note that the changes concern only the petals $P_i$, $i \in \II_{par}$, while the parts dealing with the compact part of $\overline{W} \setminus \bigcup_{i\in\II_{par}} \Orb(p_i)$ and the petals $P_i$, $i \in \II_\infty$, stay untouched.

We repeat the analysis of the petal dynamics contained in  Subsection~\ref{subsec:metric_petal_dyn} (including Lem\-ma~\ref{lem:WinKi}), replacing $\rp|_{P_i}$ by $\rp^{\ell_i}|_{P_i}$ for $i \in \II_{par}$ and setting 
\[
\ap_i = (\rp^{\ell_i}|_{P_i})^{-1}.
\]
The set $\tilde K_0$ is now defined as 
\[
\tilde K_0 = Y \cup \bigcup_{i\in \II_{par}} \bigcup_{n=0}^{n_0-1}\bigcup_{s=0}^{\ell_i-1} \rp^s(\overline{W} \cap \ap_i^n(K_i)) \cup \bigcup_{i\in \II_\infty} \bigcup_{n=0}^{n_0-1}(\overline{W} \cap \ap_i^n(K_i) ).
\]
for a large $n_0$ (which changes appropriately the definition of $\tilde K$). In Lemma~\ref{lem:z_ij} and \eqref{eq:alpha}, for $i \in \II_{par}$, instead of the points $z_{i,j}$ we consider points $z_{i,s,j}$, $s \in \{0, \ldots, \ell_i-1\}$, where $\rp(z_{i,s,j}) = \rp^s(p_i)$. With these modifications, the estimates in Lemma~\ref{lem:alpha} hold for $z_{i,s,j}$ and $\rp^s(p_i)$ instead of $z_{i,j}$ and $p_i$. Lemma~\ref{lem:sigma-expand} shows now that $\rp^{\ell_i}$ is locally expanding with respect to $d\varsigma_i$ on 
$\overline{W_1 \setminus \tilde K} \cap P_i$ for $i\in \II_{par}$.

Now we explain how to modify the definition of the metric $\varsigma$ in Subsection~\ref{subsec:metric_metric}. We set
\[
\widehat K = \tilde K \cup \bigcup_{i \in \II_{par}}\bigcup_{n=n_0}^{m_0}
\bigcup_{s=0}^{\ell_i-1} \rp^s(\overline{W} \cap \ap_i^n(K_i))
\cup \bigcup_{i \in \II_\infty} \big(\overline{W} \cap P_i \cap \overline{\D(0, R_i)}\big)
\]
for a large $m_0 > n_0$. For $i \in \II_{par}$ and $z \in \overline{W \setminus \tilde K} \cap \rp^s(P_i)$, $s \in \{0, \ldots, \ell_i - 1\}$, we define
\[
\tilde\varsigma_i(z) = \frac{|(\rp^{\ell_i})'(w)|^{\frac{s}{\ell_i}}}{|(\rp^s)'(w)|} \frac{(\varsigma_i (\rp^{\ell_i}(w)))^{\frac{s}{\ell_i}}}{(\varsigma_i(w))^{\frac{s}{\ell_i} - 1}} = |(\rp^{\ell_i})'(w)|_{\varsigma_i}^{\frac{s}{\ell_i}} \frac{\varsigma_i(w)}{|(\rp^s)'(w)|},
\]
where
\[
w = \ap_i(\rp^{\ell_i-s}(z)) \in \rp^{-s}(z)\cap P_i.
\]
Note that since the cycles of repelling petals generated by $P_i$ are pairwise disjoint for $i \in \II_{par}$ and within each cycle, the sets $P_i, \rp(P_i), \ldots, \rp^{\ell_i - 1}(P_i)$ are pairwise disjoint, the metric $\tilde\varsigma_i$ is well defined. A direct checking gives
\begin{align}
&\label{eq:first} \tilde\varsigma_i = \varsigma_i \text{ on } \overline{W \setminus \tilde K} \cap P_i, \text{ in particular } |(\rp^{\ell_i})'|_{\tilde\varsigma_i} = |(\rp^{\ell_i})'|_{\varsigma_i} \text{ on } \overline{W \setminus \tilde K} \cap P_i,\\
&\label{eq:second} |\rp'(z)|_{\tilde\varsigma_i} = |(\rp^{\ell_i})'(w)|_{\varsigma_i}^{\frac{1}{\ell_i}} \quad  \text{for } z \in \overline{W \setminus \tilde K} \cap \rp^s(P_i), \; w = \ap_i(\rp^{\ell_i-s}(z)).
\end{align}
Moreover, since $|(\rp^s)'| \asymp 1$ near $p_i$, we have
\begin{equation}
\label{eq:third} \tilde\varsigma_i(z) \asymp \frac{1}{|z - \rp^s(p_i)|^{\alpha_{par}}} = \sigma_{\rp^s(p_i), \alpha_{par}}(z) \quad \text{for } z \in \overline{W \setminus \tilde K} \cap \rp^s(P_i).
\end{equation}
By \eqref{eq:first}, \eqref{eq:second} and (modified) Lemma~\ref{lem:sigma-expand}, the map $\rp$ is locally expanding with respect to $\tilde\varsigma_i$ on $\overline{W_1 \setminus \tilde K} \cap (P_i \cup \rp(P_i) \cup \ldots, \rp^{\ell_i-1}(P_i))$.

Now, in the second and third item of the definition of $\varsigma$, we replace $\varsigma_i$ on $(W_N \setminus \tilde K) \cap P_i$ by $\tilde\varsigma_i$ on $(W_N \setminus \tilde K) \cap (P_i \cup \rp(P_i) \cup \ldots \cup \rp^{\ell_i - 1}(P_i))$. In the proof of Lemma~\ref{lem:g-sigma}, showing that $\rp$ is locally expanding with respect to $d\varsigma$, in the case~(ii) we use Lemma~\ref{lem:alpha} and \eqref{eq:third}, while the case~(iii) follows by the fact that $\rp$ is locally expanding with respect to $\tilde\varsigma_i$.

Finally, in the proof of Lemma~\ref{lem:der_i}, for $i \in \II_{par}$ and $m \ge n_0$ we take $(\beta_{m,k}^{(i)})_{k=0}^\infty$ to be the sequence $(a_{m,k})_{k=0}^\infty$ from Proposition~\ref{prop:der-par}, suited for the attracting petal $\rp^{\ell_i}(P_i)$ of the map $\ap_i = (\rp^{\ell_i}|_{P_i})^{-1}$ at $p_i$, the compact set $K_i$ and $\alpha = \alpha_{par}$. By Proposition~\ref{prop:der-par}, we have $\beta_{m,k}^{(i)}>0$, the sequence $(\beta_{m,k}^{(i)})_{k=0}^\infty$ is decreasing, $\sum_{k=0}^\infty \beta_{m,k}^{(i)} < \infty$, $\beta_{m,0}^{(i)} = 1$, $\frac{\beta_{m,k+1}^{(i)}}{\beta_{m,k}^{(i)}} = \frac{\beta_{m+k,1}^{(i)}}{\beta_{m+k,0}^{(i)}} = \beta_{m+k,1}^{(i)}$ for every $k \in \N$, $\frac{\beta_{m,k+1}^{(i)}}{\beta_{m,k}^{(i)}} \ge \frac{\beta_{m,k}^{(i)}}{\beta_{m,k-1}^{(i)}}$ for every $k \in \N$ and
\begin{equation}\label{eq:>Q11}
|(\rp^{k\ell_i})'(z)|_{\sigma_{p_i, \alpha_{par}}} > \frac{1}{\beta_{m,k}^{(i)}} 
\end{equation}
for $z \in \ap_i^{m+k}(K_i)$, $k \in \N$.
For $n \ge 0$ we define
\[
\widetilde\beta_{m,n}^{(i)} =
\beta_{m,k}^{(i)} \bigg(\frac{\beta_{m,k+1}^{(i)}}{\beta_{m,k}^{(i)}}\bigg)^{\frac{r}{\ell_i}} \quad\text{for } \; n = k\ell_i + r, \; k \ge 0, \; r \in \{0, \ldots, \ell_i - 1\}
\]
and
\[
\beta_n = \max\Big(\max_{i \in \II_{par}} \widetilde\beta_{m_0,n}^{(i)}, \max_{i \in \II_\infty}\max_{m \in\{n_0, \ldots, m_1\}} \widehat\beta^{(i)}_{m,n}\Big)
\]
for $m_0$ from the definition of $\widehat K$. Then $(\beta_n)_{n=0}^\infty$ is a decreasing sequence of positive numbers, satisfying the assertions (a)--(b) of Lemma~\ref{lem:der_i}. To show the assertion~(c), we note that if $z \in (W_N \setminus \widehat K) \cap \rp^{-1}(W_N \setminus \widehat K) \cap \ldots \cap \rp^{-(n-1)}(W_N \setminus \widehat K) \cap \rp^{-n}(W_N \cap \widehat K)$ and $\rp^{-n}(z) \notin \bigcup_{i \in \II_\infty} P_i$, then $z \in \rp^r(\ap_i^{m_0+k+1}(K_i))$, $\rp^r(z) \in \ap_i^{m_0+k}(K_i)$, $\rp^n(z) =  \rp^{k\ell_i}(\rp^r(z)) \in \ap_i^{m_0}(K_i)$ for some $i \in \II_{par}$, where $n = k\ell_i + r$ for $k \ge 0$ and $r \in \{0, \ldots, \ell_i - 1\}$. Consequently, \eqref{eq:first}, \eqref{eq:second} and \eqref{eq:>Q11} imply
\begin{align*}
|(\rp^n)'(z)|_\varsigma &= 
|(\rp^n)'(z)|_{\tilde\varsigma_i}= |(\rp^{k\ell_i})'(\rp^r(z))|_{\tilde\varsigma_i} |(\rp^r)'(z)|_{\tilde\varsigma_i}\\
&= |(\rp^{k\ell_i})'(\rp^r(z))|_{\varsigma_i} |(\rp^{\ell_i})'(\ap_i(\rp^r(z)))|_{\varsigma_i}^{\frac{r}{\ell_i}}\\
&= |(\rp^{k\ell_i})'(\rp^r(z))|_{\sigma_{p_i, \alpha_{par}}} |(\rp^{\ell_i})'(\ap_i(\rp^r(z)))|_{\sigma_{p_i, \alpha_{par}}}^{\frac{r}{\ell_i}}\\
&>\frac{1}{\beta_{m_0,k}^{(i)}\big(\beta_{m_0+k, 1}^{(i)}\big)^{\frac{r}{\ell_i}}} = 
\frac{1}{\beta_{m_0,k}^{(i)}\Big(\frac{\beta_{m_0,k+1}^{(i)}}{\beta_{m_0,k}^{(i)}}\Big)^{\frac{r}{\ell_i}}} = \frac{1}{\widetilde\beta_{m_0,n}^{(i)}}\ge \frac{1}{\beta_n},
\end{align*}
which gives the assertion~(c) and ends the proof of Proposition~\ref{prop:expanding_metric}.
\end{proof}


\section{Proof of Theorem~B} \label{sec:proofB}

Before proceeding with the proof of Theorem~B, we show Proposition~{\rm\ref{prop:extrabonus}}, which describes some properties of the maps and attracting basins we are dealing with. We start by proving a useful lemma.

\begin{lem}\label{lem:prop1} Let $U$ be a simply connected attracting invariant basin of a meromorphic map $\f$ satisfying the hypotheses of Theorem~A. Then there exist finite collections $\II_{par}$, $\II_\infty$ and sets $P_i$, $i \in \II_{par} \cup \II_\infty$, such that the following hold.
\begin{enumerate}[$($a$)$]

\item For $i \in \II_{par}$, the set $P_i$ is a repelling petal of $\f^{\ell_i}$, for some $\ell_i \in \N$, at a parabolic periodic point $p_i \in \bd U$ of $\f$, such that $P_i$ generates a cycle of length $\ell_i$ of repelling petals of $\f$ at $p_i$, and these cycles are pairwise disjoint for $i \in \II_{par}$. Moreover, the set $\bigcup_{i \in \II_{par}}\{p_i, f(p_i), \ldots,  \f^{\ell_i-1}(p_i)\}  =  \bigcup_{i \in \II_{par}} \Orb(p_i)$ is equal to the set of all parabolic periodic points of $\f$ in $\bd U$.

\item For $i \in \II_\infty$, the set $P_i$ is a repelling petal of $\rp$ at infinity, such that $\bigcup_{i \in \II_\infty} \f(P_i)$ is contained in an arbitrarily small neighbourhood of infinity, and $U \cap \f(P_i)$ are pairwise disjoint for $i \in \II_\infty$. In particular, $\bigcup_{i \in \II_\infty} \f(P_i)$ is disjoint with the cycles of the repelling petals generated by $P_{i}$, $i \in\II_{par}$.

\item The set $\overline{U} \setminus  \bigcup_{i \in \II_{par}}\bigcup_{s=0}^{\ell_i-1} \f^s(P_i \cup \{p_i\}) \cup \bigcup_{i \in \II_\infty} P_i$ is compact.

\end{enumerate}
\end{lem}

\begin{proof}[Proof of Lemma~{\rm\ref{lem:prop1}}] By the assumption~(c) of Theorem~A, there exists a finite collection of repelling petals $P_i$, $i \in \II_\infty$, of $\f$ at infinity, such that $\overline{U} \setminus \bigcup_{i \in \II_\infty} P_i$ is compact and $U \cap \f(P_i) \cap P_{i'} = \emptyset$ for $i \neq i'$. We will show that
\begin{equation}\label{eq:PtoG(P)}
\overline{U} \setminus \bigcup_{i \in \II_\infty} \ap_i(P_i) \quad \text{is compact},
\end{equation}
where $\ap_i = (f|_{P_i})^{-1}$. Suppose \eqref{eq:PtoG(P)} does not hold. Then there exists a sequence $z_n \in \overline{U} \setminus \bigcup_{i \in \II_\infty} \ap_i(P_i)$ such that $z_n \to \infty$. Since $\overline{U} \setminus \bigcup_{i \in \II_\infty} P_i$ is compact, we have $z_n \in \overline{U} \cap \big(\bigcup_{i \in \II_\infty} P_i\big) \setminus \bigcup_{i \in \II_\infty} \ap_i(P_i)$ for sufficiently large $n$. Consequently, there exists $i \in \II_\infty$ such that $z_n \in \overline{U} \cap P_i \setminus \ap_i(P_i)$ for infinitely many $n$. By Definition~\ref{defn:petal_at_inf} and the invariance of $U$, for such $n$ we have $f(z_n) \in \overline{U} \cap f(P_i) \setminus P_i$ and hence, as $U \cap \f(P_i) \cap P_{i'} = \emptyset$ for $i \neq i'$, there holds $f(z_n) \in \overline{U} \setminus \bigcup_{i' \in \II_\infty}P_{i'}$ for infinitely many $n$. Since $\overline{U} \setminus \bigcup_{i' \in \II_\infty} P_{i'}$ is compact, we have $f(z_{n_k}) \to w$ for some subsequence $n_k$ and $w\in \overline{U} \cap \overline{f(P_i)}$. Then $z_{n_k} = \ap_i(f(z_{n_k})) \to \ap_i(w) \in \overline{P_i}$, which makes a contraction. This ends the proof of \eqref{eq:PtoG(P)}. 

By \eqref{eq:PtoG(P)}, without loss of generality, for $i \in \II_\infty$ we can replace $P_i$ by $G_i(P_i)$ and, inductively, by $G_i^n(P_i)$ for any given $n \in \N$. As $\bigcap_n G_i^n(P_i) = \emptyset$ by Definition~\ref{defn:petal_at_inf}, we can thus assume that $\f(P_i)$, $i \in \II_\infty$ are contained in an arbitrarily small neighbourhood of infinity and $U \cap \f(P_i)$ are pairwise disjoint, which proves the first two parts of the assertion~(b). 

Note that by the assumption~(a) of Theorem~A, there are at most finitely many parabolic periodic points in $\bd U$. By Proposition~\ref{prop:periodic_par_petal}, a union of small punctured neighbourhoods of these points is covered by a finite union of cycles of attracting and repelling petals of $f$, where the attracting petals are contained in the basins of attraction to the orbits of these points, and the cycles of the repelling petals are pairwise disjoint. It follows that the intersection of $\overline{U}$ with these punctured neighbourhoods is contained in a disjoint union of cycles of repelling petals of $f$, which by Proposition~\ref{prop:periodic_par_petal} are generated, respectively, by $P_i$, $i \in \II_{par}$, where $\II_{par}$ is a finite set, $P_i$ is a repelling petal of $f^{\ell_i}$, for some $\ell_i \in \N$, at a parabolic periodic point $p_i \in \bd U$, the number $\ell_i$ is a multiple of the period of $p_i$ and $\bigcup_{i \in \II_{par}} \Orb(p_i)$ is equal to the set of all parabolic periodic points in $\bd U$. (Note that it is possible to have $p_i = p_{i'}$ for $i \ne i'$, as a parabolic periodic point of $f$ in $\bd U$ may correspond to several cycles of repelling petals). This shows the assertions~(a) and~(c). Since $f(P_i)$, $i \in \II_\infty$, are contained in a small neighbourhood of infinity, they are disjoint with the cycles of repelling petals generated by $P_i$, $i \in \II_{par}$. This shows the last part of the assertion~(b) and ends the proof.
\end{proof}

\begin{proof}[Proof of Proposition~{\rm\ref{prop:extrabonus}}]
Assume the hypotheses of Theorem~A and consider the repelling petals $P_i$, $i \in \II_{par} \cup \II_\infty$, which exist according to Lemma~\ref{lem:prop1}. Similarly as in Section~\ref{sec:metric}, we write
\[
\II = \II_{par} \cup \II_\infty
\]
and
\[
p_i = \infty \quad \text{for} \quad i \in \II_\infty.
\]
Under the notation of Lemma~\ref{lem:prop1}, by the definition of repelling petals, for $i \in \II$ we have $\bigcap_{n=0}^\infty \ap_i^n(P_i) = \emptyset$, where
\[
\ap_i = (f^{\ell_i}|_{P_i})^{-1} \quad  \text{for } \; i \in \II_{par}, \qquad  \ap_i = (f|_{P_i})^{-1} \quad\text{ for } \; i \in \II_\infty.
\]
This implies
\begin{equation}\label{eq:Gn-Gn+1}
P_i = \bigcup_{n=0}^\infty (\ap_i^n(P_i) \setminus \ap_i^{n+1}(P_i)) \quad\text{ for } \; i \in \II.
\end{equation}
Moreover,
\begin{equation}\label{eq:zn}
\text{if} \in \II_\infty \quad \text{and} \quad z_n \in P_i, \; z_n \xrightarrow[n\to\infty]{} \infty, \quad \text{then} \quad f(z_n) \xrightarrow[n\to\infty]{} \infty.
\end{equation}
To check \eqref{eq:zn}, suppose $z_n \in P_i$, $z_n \to \infty$, $f(z_n) \not\to \infty$. Passing to a subsequence, we can assume $f(z_n) \to w \in \overline{f(P_i)}$. Then $z_n = \ap_i(f(z_n)) \to \ap_i(w) \in \overline{P_i} \subset \C$, which is a contradiction.

Furthermore, recall by Lemma~\ref{lem:prop1} there exists a compact set $K \subset \overline{U} \setminus \bigcup_{i \in \II_{par}} \Orb(p_i)$, such that 
\[
\overline{U} \setminus K \subset \bigcup_{i \in \II_{par}}\bigcup_{s=0}^{\ell_i-1} \f^s(P_i \cup \{p_i\}) \cup \bigcup_{i \in \II_\infty} P_i.
\]

Now we proceed with the proof of the proposition. First, note that since the map $f$ is univalent in each repelling petal $P_i$, $i \in \II_\infty$ and $\overline{U} \setminus \bigcup_{i \in \II_\infty} P_i$ is compact, we obtain that $U$ is bounded or every point in $U \cap \bigcup_{i \in \II_\infty} P_i$ has a finite number of preimages in $U$. This proves the assertion~(a).

Consider the assertion~(b). Since the degree of $f$ on $U$ is finite, there are only a finite number of critical points of $f$ in $U$. Moreover, as $f$
is univalent in each cycle of repelling petals generated by $P_i$, $i \in \II_{par}$, and in each repelling petal $P_i$, $i \in \II_\infty$, all critical points of $f$ in $\bd U$ are contained in $K$. Since the set of critical points cannot have an accumulation point in $K$, there are only a finite number of critical points of $f$ in $\bd U$.

To prove the assertion~(c), suppose that $z \in \bd U$ is a post-critical point of $f$ with infinite orbit. Obviously, this implies $z \notin \bigcup_{n=0}^\infty f^{-n}\big(\bigcup_{i \in \II_{par}} \Orb(p_i) \cup \{p_i\}_{i\in \II_\infty}\big)$. Then, by the assumption~(a) of Theorem~A and Lemma~\ref{lem:prop1}(a), the set $K$ cannot contain an infinite number of elements of this orbit. Consequently, the orbit of $z$ contains a point $w$ such that 
\begin{equation}\label{eq:Orb(w)}
\Orb(w) \subset \bigcup_{i \in \II_{par}}\bigcup_{s=0}^{\ell_i-1} \f^s(P_i) \cup \bigcup_{i \in \II_\infty} P_i.
\end{equation}
Suppose first $w \in \f^s(P_i)$ for some $i \in \II_{par}$, $s \in \{0, \ldots, \ell_i-1\}$. Then $u = \f^{\ell_i - s}(w) \in \f^{\ell_i}(P_i)$, so by \eqref{eq:Orb(w)} and Lemma~\ref{lem:prop1}, in fact we have $u \in P_i$. Repeating this argument inductively, we obtain $\f^{n\ell_i}(u) \in P_i$ for every $n \in \N$. But by 
\eqref{eq:Gn-Gn+1}, $u \in \ap_i^n(P_i) \setminus \ap_i^{n+1}(P_i)$ for some $n \ge 0$, so $f^{(n+1)\ell_i}(u) \notin P_i$, which is a contradiction. If $w \in P_i$ for some $i \in \II_\infty$, then again by \eqref{eq:Orb(w)} and Lemma~\ref{lem:prop1}, we have $\f^{n\ell_i}(w) \in P_i$ for every $n \in \N$, which contradicts \eqref{eq:Gn-Gn+1}.

To prove the assertion~(d), consider an asymptotic curve $\gamma\colon [0,+\infty) \to \C$ of an asymptotic value $v \in \overline{U}$. If $\gamma$ is not eventually contained in $\C \setminus\overline{U}$, then there is a sequence of points $z_n \in \gamma \cap \overline{U}$, such that $z_n \to \infty$, $f(z_n) \to v$. Passing to a subsequence, we can assume $z_n \in P_i$ for some $i \in \II_\infty$. However, this contradicts \eqref{eq:zn}.

Finally, we show the assertion~(e). Let us consider $w\in \overline{U}\cup\{\infty\}$ and analyze several different cases. If $w\in U$, then $w$ has $d > 0$ preimages in $U$ counting multiplicities and there is nothing to prove. Hence, we can assume $w\in \partial U \cup \{\infty\}$. 

Suppose $w\in \partial U$ and consider a sequence $w_n\to w$ with $w_n\in U$. Let $z_n$ be any sequence of preimages in $U$ such that $f(z_n)=w_n$. Passing to a subsequence, we may assume $z_n \to z \in \overline{U}\cup\{\infty\}$. If $z = \infty$, then (again after passing to a subsequence), we have $z_n \in P_i$ for some $i \in \II_\infty$, which contradicts \eqref{eq:zn}. Hence, $z\in \overline{U}$. Observe that by continuity, $f(z)=w$ and $z \in\partial U$ by the maximum principle. 
 
It remains to consider the case $w=\infty$. Take a sequence $w_n\to w$ with $w_n\in U$. Since $U$ is a simply connected invariant attracting basin of finite degree, it contains a critical point and $d = \deg f|_U \ge 2$. Hence, there exist two sequences $z_n^{(1)}, z_n^{(2)} \in U$, such that $f(z_n^{(1)}) = f(z_n^{(2)}) = w_n$, $z_n^{(1)} \ne z_n^{(2)}$. Passing to subsequences, we can assume $z_n^{(1)} \to z^{(1)}$, $z_n^{(2)} \to z^{(1)}$ for some $z^{(1)}, z^{(2)} \in \overline{U}\cup\{\infty\}$. If $z^{(1)} = z^{(2)} = \infty$, then (again after passing to a subsequence), we have $w_n \in P_i$, $z_n^{(1)} \in P_{i_1}$, $z_n^{(2)} \in P_{i_2}$ for some $i, i_1, i_2 \in \II_\infty$.
By Lemma~\ref{lem:prop1}, we have $i = i_1 = i_2$, which contradicts the injectivity of $f$ in repelling petals. Hence, one of the points $z^{(1)},z^{(2)}$ is in $\bd U$. By continuity, this point is mapped by $f$ to $w = \infty$.
\end{proof}

\begin{proof}[Proof of Theorem~B]
We will show that one may define appropriate sets in the dynamical plane of $f$ satisfying the hypotheses of Theorem~{\rm\ref{thm:metric}}, which will provide a metric with suitable expanding properties defined near the boundary of $U$.

Let $U$ be an invariant simply connected attracting basin of a meromorphic map $\f\colon \C \to \clC$ satisfying the assumptions of Theorem~A. 
Let $\varphi \colon \D \to U$ be a Riemann map, such that $0 \in \D$ is the fixed point of the map $g = \varphi^{-1} \circ \f \circ \varphi \colon \D \to \D$. Since the degree of $f$ in $U$ is finite by Proposition~\ref{prop:extrabonus}, the map $g$ is a finite Blaschke product. Let 
\[
E = \varphi(\D(0, r))
\]
for $r \in (0,1)$ close to $1$. Obviously, $E$ is a Jordan domain and $\overline{E} \subset U$. By the Schwarz lemma, we have
\begin{equation}\label{eq:fAinA}
\overline{\f(E)} \subset E.
\end{equation}
Notice that $E$ is an absorbing domain for $\f$ in $U$ (i.e.~every compact set in $U$ is mapped by an iterate of $f$ into $E$), since $\D(0,r)$ is absorbing for $g$. Moreover, by the assumption~(a) of Theorem~A and since $\f$ has finite degree on $U$, we can choose $r$ so that
\begin{equation}\label{eq:SinginA}
\big(\big(\overline{\PP_{asym}(\f)} \cup \Acc(\PP_{crit}(\f)) \big) \cap U\big) \cup \overline{\Crit(\f|_U) \cup \PP_{crit}(\f|_U)} \subset E
\end{equation}
and 
\begin{equation}\label{eq:LinfE}
L \subset \f(E)
\end{equation}
for the compact set $L \subset U$ from the assumption~(b) of Theorem~A.

Consider the repelling petals $P_i$, $i \in \II_{par} \cup \II_\infty$, which exist according to Lemma~\ref{lem:prop1}. In particular, the assertion~(b) of this lemma shows that we can assume
\begin{equation}\label{eq:PcapA}
\overline{E} \cap \bigcup_{i \in \II_\infty} f(P_i) = \emptyset.
\end{equation}

Now, we want to define hyperbolic domains $V'\subsetneq V\subset \C$, an open set $W\subset V$ and a holomorphic map $\rp\colon V'\to V$ satisfying the hypotheses of Theorem~{\rm\ref{thm:metric}}. In particular, $\rp$ must map $V'$ onto $V$, and no point $z\in W$ can have its entire orbit contained in $W$. To do so, set
\[ 
W=U\setminus \overline{E}.
\]
See Figure~\ref{fig:W}.
\begin{figure}[ht!]
\includegraphics[width=0.6\textwidth]{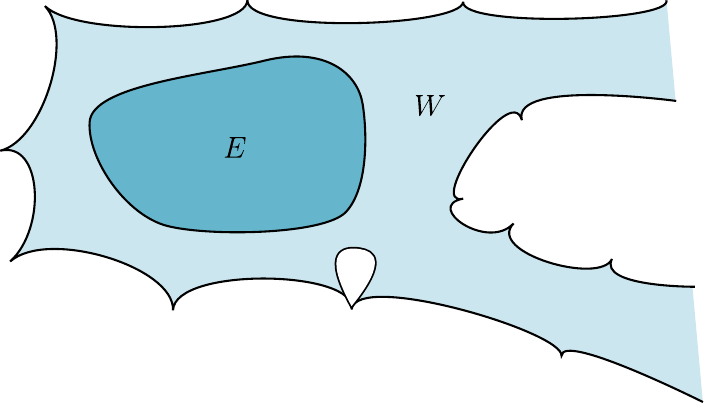}
\caption{\small The sets $E, W$.}
\label{fig:W}
\end{figure}
Since $E$ is an absorbing domain, every orbit in $W\subset U$ eventually enters $E$ and remains there, leaving $W$. 
Now, define $V$ to be the connected component of 
\[
\C \setminus \Big(\overline{\f(E)} \cup \overline{\PP_{asym}(\f)} \cup \Acc(\PP_{crit}(\f)) \cup \Big(\PP_{crit}(\f) \setminus \bigcup_{n=0}^\infty \f^{-n}(\overline{U})\Big)\Big)
\]
containing $\f(W)=U\setminus \overline{\f(E)}$, and let $V'$ be the connected component of $\f^{-1}(V)$
containing $W$. Note that neither $V$ nor $V'$ contains parabolic periodic points of $\f$ in $\partial U$, since they belong to $\overline{\PP_{asym}(\f)} \cup \Acc(\PP_{crit}(\f))$. On the other hand,
\begin{equation}\label{eq:WinV}
\overline{W} \subset V \cup \bigcup_{i \in \II_{par}} \Orb(p_i)
\end{equation}
by Lemma~\ref{lem:prop1}, \eqref{eq:fAinA}, \eqref{eq:SinginA} and the assumption~(a) of Theorem~A.

Let
\[
\rp = \f|_{V'}.
\]
Note that by definition, $V', V, W$ are domains in $\C$. Note also that $V$ is hyperbolic since it is disjoint from the domain $\f(E)$, which in turn implies that $V'$ is disjoint from $E$, so $V'$ is hyperbolic. It is straightforward to check
\begin{equation}\label{eq:WinV'}
W \subset V \cap V'.
\end{equation}
Hence, as $V' \subset \f^{-1}(V)$ and $W$ is connected, $\rp$ maps $V'$ into $V$. We have $V' \neq V$ since $V \setminus V' \supset \overline{E} \setminus \f(E)$.
Now we check $V'\subset V$. By \eqref{eq:WinV'} and the connectedness of $W$, it is sufficient to show 
\[
V'\subset \C \setminus \Big(\overline{\f(E)} \cup \overline{\PP_{asym}(\f)} \cup \Acc(\PP_{crit}(\f)) \cup \Big(\PP_{crit}(\f) \setminus \bigcup_{n=0}^\infty \f^{-n}(\overline{U})\Big)\Big).
\]
This follows since the sets $\overline{\f(E)}$, $\overline{\PP_{asym}(\f)}$, $\Acc(\PP_{crit}(\f))$ and $\PP_{crit}(\f) \setminus \bigcup_{n=0}^\infty \f^{-n}(\overline{U})$ are forward-invariant.

To see that $\rp\colon V'\to V$ is onto (i.e.~$\f(V') = V$), observe that every point in $V \setminus \f(V')$ is a locally omitted value and hence an asymptotic value of $\f$ (the argument for this fact, using Gross' theorem, is described e.g.~in \cite[proof of Theorem~2]{herring}). This shows $V \setminus f(V') = \emptyset$ since $V$ contains no asymptotic values of $f$ by definition. The condition $\bigcap_{n=0}^\infty \rp^{-n}(W) = \emptyset$ is implied by the fact that $E$ is an absorbing domain for $\f|_U$.

It remains to check the conditions~(a)--(e) in Theorem~\ref{thm:metric}, which follow in a fairly direct way from the definitions of $V, V'$ and $W$. Indeed, $V$ is defined not to contain asymptotic values of $f$ and $V'$ is a component of $f^{-1}(V)$, which gives~(a). To check~(b), take a point $z\in \PP_{crit}(\rp)$ and suppose there exist points $w_k \in \rp^{-n_k}(z)$ for some $n_k > 0$, $k \in \N$, such that $\deg_{w_k} \rp^{n_k} \to +\infty$. Note that by the definition of $V$, we have $\f^{n_0}(z) \in \overline{U}$ for some $n_0 \ge 0$. Moreover, as $V \cap L = \emptyset$ by \eqref{eq:LinfE}, we have $\f^{n_0}(z) \notin L$, where $L \subset U$ is the compact set from the assumption~(b) of Theorem~A. Then $w_k \in \f^{-(n_k+n_0)}(\f^{n_0}(z))$ and $\deg_{w_k} \f^{n_k+n_0} \ge \deg_{w_k} \rp^{n_k} \to +\infty$, which contradicts the assumption~(b) of Theorem~A, as $\f^{n_0}(z) \in \PP_{crit} \cap \overline{U} \setminus L$. This shows (b). The condition~(c) follows from the definition of $V$, while~(d) is trivial since $f$ is meromorphic on $\C$. The condition~(e) follows from Lemma~\ref{lem:prop1}, \eqref{eq:PcapA}, \eqref{eq:WinV} and the invariance of $U$. 

Having checked that $\rp$, $V$, $V'$ and $W$ satisfy the hypotheses of Theorem~{\rm\ref{thm:metric}}, we can use this theorem and \eqref{eq:fAinA} to find $N \in \N$ such that for every compact set $K_0 \subset \overline{U} \setminus \big(E \cup \bigcup_{i \in \II_{par}} \Orb(p_i)\big)$ there exist a conformal metric $d\varsigma = \varsigma|dz|$ on $U \setminus \overline{f^{-N}(E)}$ and a decreasing sequence $(b_n)_{n=1}^\infty$ of numbers $b_n \in (0,1)$ with $\sum_{n=1}^\infty b_n < \infty$, satisfying
\begin{equation}\label{eq:final-exp}
|(f^n)'(z)|_\varsigma > \frac{1}{b_n} 
\end{equation}
for every $z \in (U \setminus f^{-n}(\overline{f^{-N}(E)})) \cap f^{-n}(K_0)$, $n \in \N$. 

Now we conclude the proof of Theorem~B. Let $K$ be a compact set in $U$. Set
\[
K_0 = K \setminus E
\]
and
\[
A = f^{-N}(E) \cap U.
\]
Note that by \eqref{eq:SinginA}, the set $A$ is a Jordan domain, such that $\overline{A} \subset U$ and $U \setminus \overline{A} = U \setminus \overline{f^{-N}(E)}$. Moreover, $K_0$ is a compact subset of $\overline{U} \setminus \big(E \cup \bigcup_{i \in \II_{par}} \Orb(p_i)\big)$ and 
\[
(U \setminus f^{-n}(\overline{f^{-N}(E)})) \cap f^{-n}(K_0) = 
(U \setminus f^{-n}(\overline{A})) \cap f^{-n}(K) = f^{-n}(K \setminus \overline{A}),
\]
so by \eqref{eq:final-exp}, there exist a conformal metric $d\varsigma$ on $U \setminus \overline{A}$ and a decreasing sequence $(b_n)_{n=1}^\infty$ of numbers $b_n \in (0,1)$ with $\sum_{n=1}^\infty b_n < \infty$, such that
\[
|(f^n)'(z)|_\varsigma > \frac{1}{b_n} 
\]
for $z \in U \cap f^{-n}(K\setminus \overline{A})$, $n \in \N$. This ends the proof of Theorem~B.

\end{proof}


\section{Local connectivity of boundaries of Fatou components: proof of Theorem~A}\label{sec:proofA}

To prove the local connectivity of the boundary of a simply connected invariant attracting basin $U$ satisfying the conditions of Theorem~A, we will construct a sequence of Jordan curves $\{\gamma_n\}_{n=0}^\infty$ which approximate $\partial U \cup \{\infty\}$ and satisfying the uniform Cauchy condition with respect to the spherical metric, thus showing that its limit is also a curve in $\clC$ and hence is locally connected. For simplicity, we use the symbol $\gamma$ indistinctly for a curve $\gamma\colon [a,b]\to \chat$ and for its image in $\chat$.


\subsection{Equipotential curves and ray germs}\label{subsec:curves}

Consider a simply connected domain $A \subset U$ satisfying the properties  listed in Theorem~B. The goal of this subsection is to construct an increasing family of Jordan domains $A_n \subset U \setminus \overline{A}$, $n =0, 1, \ldots$, exhausting $U$ and such that $f$ maps $\overline{A_{n+1}} \setminus A_n$ onto $\overline{A_n}\setminus A_{n-1}$ as a degree $d$ covering, where $d=\deg f|_U$. The Jordan curves $\gamma_n$ mentioned above will be defined as the boundaries of these domains. 

\begin{prop} \label{prop:curvesattracting} Let $f\colon\C \to \chat$ be a transcendental meromorphic function with a simply connected immediate basin of attraction $U$ of an attracting fixed point $\zeta$, such that $d =\deg f|_U < \infty$, and let $A \subset U$ be a domain such that $\overline{A} \subset U$. 
Then there exists a family of Jordan domains $\{A_n\}_{n=0}^\infty$ with smooth boundaries $\gamma_n \colon [0,1]\to U$, such that
\begin{itemize}
\item[$\circ$] $A_0$ contains $\zeta$ and all the critical points of $f|_U$, 
\item[$\circ$] $\overline{A} \subset A_0$, $\overline{A_n} \subset A_{n+1}$ for $n \ge 0$ and $\bigcup_{n=0}^\infty A_n = U$,
\item[$\circ$] $f$ maps $\overline{A_{n+1}}$ onto $\overline{A_n}$ as a proper map of degree $d$ such that $\overline{A_{n+1}} \setminus A_n$ is mapped onto $\overline{A_n}\setminus A_{n-1}$ for $n \ge 1$ as an $($unbranched$)$ degree $d$ covering,
\item[$\circ$] $f(\gamma_{n+1}(\theta))=\gamma_n(d \theta \bmod 1)$ for $n \ge 0$, $\theta \in [0,1]$. 
\end{itemize}
Moreover, for every $\theta \in [0,1]$ there exists a smooth arc $\alpha_\theta \colon[0,1]\to \overline{A_1}\setminus A_0$, joining $\gamma_0(\theta)$ and $\gamma_1(\theta)$, such that 
\[
\alpha_\theta((0,1)) \subset A_1\setminus \overline{A_0} \quad \text{ and} \quad \length \alpha_\theta \leq M,
\]
where $M$ is a constant independent of $\theta$ and $\length$ denotes the Euclidean length.
\end{prop}

\begin{proof}
Similarly as in the proof of Theorem~B, let $\varphi \colon \D \to U$ be a Riemann map, such that $\varphi^{-1}(\zeta) = 0$ is the fixed point of the degree $d$ Blaschke product $g = \varphi^{-1} \circ \f \circ \varphi \colon \D \to \D$. Note that $d \ge 2$ since $U$ must contain a critical point of $f$. We define the domain $A_0$ as
\[
A_0 = \varphi(\D(0, r_0))
\]
for $r_0 \in (0,1)$ so close to $1$, that
\[
\overline{\Crit(\f|_U) \cup \PP_{crit}(\f|_U)} \cup \overline{A} \subset A_0.
\]
Then $A_0$ is a Jordan domain with a smooth boundary $\gamma_0$, such that $\gamma_0 \subset U$. Let
\[
A_n=f^{-1}(A_{n-1}) \cap U, \qquad n =1, 2, \ldots.
\]
By the Schwarz lemma and the fact that $A_0$ contains all the critical points and values of $f$ in $U$, the sets $A_n$ are Jordan domains with smooth boundaries $\gamma_n \subset U$, such that
\[ 
A \Subset A_0 \Subset A_1 \Subset \cdots \Subset A_n \Subset A_{n+1} \Subset \cdots
\]
and $U= \bigcup_{n = 0}^\infty A_n$. Moreover, $f$ maps $\overline{A_{n+1}}$ onto $\overline{A_n}$ for $n \ge 0$ as a proper map of degree $d$, and $\overline{A_{n+1}} \setminus A_n$ is mapped onto $\overline{A_n}\setminus A_{n-1}$ for $n \ge 1$ as a covering of degree $d$. See Figure~\ref{fig:curvesattracting}.

\begin{figure}[ht!]
\setlength{\unitlength}{0.9\textwidth}
\includegraphics[width=0.9\textwidth]{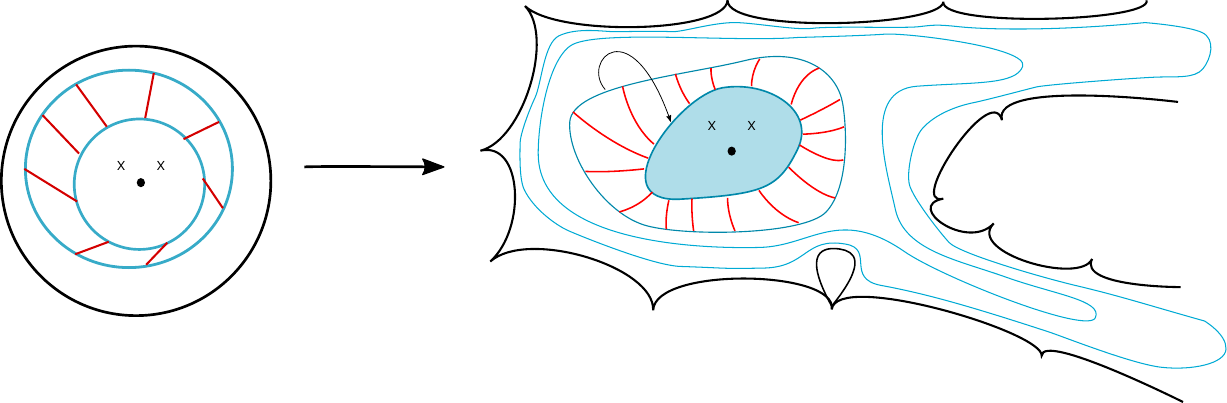}
\put(-0.89,0.16){\scriptsize $0$}
\put(-0.41,0.18){\scriptsize $\zeta$}
\put(-0.7,0.22){\small $\varphi$}
\put(-0.46,0.18){\small $A_0$}
\put(-0.51,0.17){\small $A_1$}
\put(-0.37,0.17){\scriptsize $\gamma_0$}
\put(-0.31,0.18){\scriptsize $\gamma_1$}
\put(-0.51,0.22){\scriptsize ${\r \alpha_\theta}$}
\put(-0.48,0.28){\tiny $d:1$}
\put(-0.87,0.23){\scriptsize $\tilde{\gamma}_0$}
\put(-0.84,0.25){\scriptsize $\tilde{\gamma}_1$}
\put(-0.16,0.28){\scriptsize $\gamma_2$}
\put(-0.01,0.28){\scriptsize $\gamma_3$}
\caption{ \small Sketch of the construction of the sets $A_n$, the family of Jordan curves $\{\gamma_n\}_n$ and the transversal ray germs $\alpha_\theta$.} \label{fig:curvesattracting}
\end{figure}
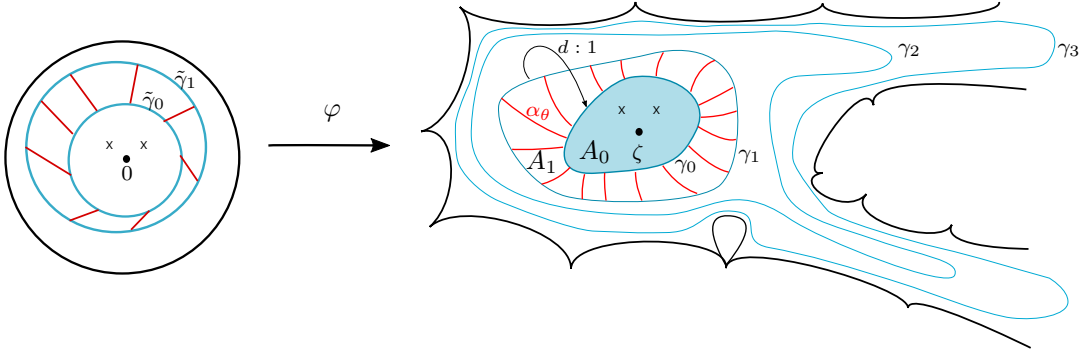

Let us choose a smooth parametrization of $\gamma_0 = \partial A_0$ and denote it by $\gamma_0\colon [0,1] \to \partial A_0$. For every $n>0$, parametrize the boundary of $A_n$ by $\gamma_n\colon [0,1] \to \partial A_n$ in such a way that 
\[
f(\gamma_{n}(\theta))= \gamma_{n-1}(d \theta \bmod 1).
\]
This parametrization is not unique but once we choose $\gamma_{n}(0)$ to be one of the $d$ preimages of $\gamma_{n-1}(0)$, then there is only one continuous choice for $\gamma_n(\theta)$, $\theta\in [0,1]$ satisfying the relation above. 

Now we show the existence of the transversal curves $\alpha_\theta$.
Set $\widetilde{\gamma}_0= \varphi^{-1}(\gamma_0) = \{z: |z|=r\}$. By choosing $r_0$ sufficiently close to 1, the Jordan curve $\widetilde{\gamma}_1=g^{-1} (\widetilde{\gamma}_0)$ can be written in polar coordinates as $(r(\theta), \theta)$ for a smooth function $r(\theta)$, $\theta\in [0,1]$ (see e.g. \cite[p.~208]{milnor}). In particular, this implies that any two points in the closed annulus $R \subset \D$ bounded by $\widetilde{\gamma}_1$ and $\widetilde{\gamma}_0$ can be joined by a curve of length smaller than $4\pi$. By applying the Riemann map $\varphi$, whose derivative is bounded in modulus on the compact set $R$, we deduce that any two points in the annulus $\overline{A_1}\setminus A_0 =\varphi(R) \subset U$ can be joined by a curve of length $M$ for some constant $M > 0$. In particular, for every $\theta \in [0,1]$ there exists an arc $\alpha_\theta$ of length bounded by $M$, joining $\gamma_0(\theta)$ and $\gamma_1(\theta)$. 
\end{proof}

\begin{rem*}
For given angle $\theta$, the sequence $\{\gamma_n(\theta)\}_n$ accumulates at the boundary of $U$ in $\clC$. In general, the curves $\gamma_n$ may have no limit and accumulate at a non-locally connected boundary. We will show that under the conditions of Theorem~A this cannot occur. 
\end{rem*}

\subsection{Local connectivity of the boundary of $U$}\label{subsec:loc_conn}
Let $A \subset U$ be a simply connected domain satisfying the properties  listed in Theorem~B. 
Our goal is to show that under the assumptions of Theorem~A, the curves $\gamma_n$ constructed in Proposition~\ref{prop:curvesattracting} for the domain $A$, converge uniformly on $[0,1]$ with respect to the spherical metric, which will imply that the boundary of $U$ in $\clC$ is locally connected. 

By Theorem~B for $K = \overline{A_1} \setminus A_0$, there exists a conformal metric $d\varsigma = \varsigma|dz|$ on $U\setminus \overline{A}$ and numbers $b_n \in (0,1)$, $n \in \N$, such that $\sum_{n=1}^\infty b_n < \infty$ and 
\[
|(f^{n})'(z)|_\varsigma > \frac{1}{b_n}
\]
for every $z \in f^{-n}(\overline{A_1} \setminus A_0) \cap U$, i.e. for $z \in \overline{A_{n+1}} \setminus A_n$. Proposition~\ref{prop:curvesattracting} shows that for every $\theta\in [0,1]$, the point $\gamma_0(\theta)$ can be joined to $\gamma_1(\theta)$ by a $C^1$-arc $\alpha_\theta \subset \overline{A_1} \setminus A_0$ of Euclidean length bounded by a constant $M$ independent of $\theta$. Since $\overline{A_1} \setminus A_0$ is a compact subset of $U\setminus\overline{A}$, the density $\varsigma$ of the metric $d\varsigma$ is bounded on this set, so the lengths of $\alpha_\theta$ with respect to the metric $d\varsigma$ are also uniformly bounded, that is
\[
\length_\varsigma \alpha_\theta \leq M'< \infty,
\]
where $M' > 0$ is a constant independent of $\theta$. 

We will show that the sequence $\gamma_n$ satisfies the uniform Cauchy condition with respect to the metric $d\varsigma$. By Proposition~\ref{prop:curvesattracting}, we can define inductively families of arcs $\alpha^{(n)}_\theta$, $n \ge 0$, $\theta \in [0,1]$ by 
\begin{align*}
\alpha^{(0)}_\theta &= \alpha_\theta,\\
\alpha^{(n+1)}_\theta &= f^{-1}_{n,\theta}\big(\alpha^{(n)}_{d\theta \bmod 1}\big),
\end{align*}
where $f^{-1}_{n,\theta}$ is the branch of $f^{-1}$ mapping $\gamma_n(d\theta \bmod 1)$ to $\gamma_{n+1}(\theta)$. Then $\alpha^{(n)}_\theta$ is a $C^1$-arc joining $\gamma_n(\theta)$ to $\gamma_{n+1}(\theta)$ within $\overline{A_{n+1}} \setminus A_n$ and $f^n$ maps $\alpha^{(n)}_\theta$ injectively onto $\alpha_{d^n\theta \bmod 1}$. Consequently,
\[
\dist_\varsigma (\gamma_n(\theta), \gamma_{n+1}(\theta)) \leq \int_{\alpha^{(n)}_\theta}\varsigma(z)|dz| \leq \frac{\length_\varsigma \alpha_{d^n\theta \bmod 1}}{\inf_{\alpha^{(n)}_\theta}|(f^n)'|_\varsigma}
< M'b_n.
\]
It follows that for every $\theta\in[0,1]$ and $m> n> 0$ we have
\[
\dist_\varsigma(\gamma_n(\theta),\gamma_m(\theta)) \leq 
\dist_\varsigma(\gamma_{n+1}(\theta),\gamma_n(\theta)) 
+ \cdots + 
\dist_\varsigma(\gamma_m(\theta),\gamma_{m-1}(\theta)) 
\leq M' \sum_{k=n}^m b_k.
\]
Since $\sum_{k=1}^\infty b_k$ is a convergent series, the sequence $\gamma_n$ satisfies the uniform Cauchy condition with respect to the metric $d\varsigma$.
We show that this implies that $\gamma_n$ satisfies also the uniform Cauchy condition with respect to the spherical metric. Suppose this does not hold. Then there exists $\varepsilon > 0$ and sequences $n_k, m_k \to \infty$, $\theta_k \in [0,1]$, such that 
\[
\dist_{sph} (\gamma_{n_k}(\theta_k), \gamma_{m_k}(\theta_k)) \geq \varepsilon.
\]
Passing to a subsequence, we can assume $\gamma_{n_k}(\theta_k) \to z$, $\gamma_{m_k}(\theta_k) \to w$ for some distinct points $z,w \in \clC$, where at least one of them (say $z$) is not equal to $\infty$. By Lemma~\ref{lem:varsigma-limit}, there exist $\delta, c > 0$, such that $w \notin \D(z, \delta)$ and $\varsigma > c$ on $W_N \cap \D(z, \delta)$. Then for sufficiently large $k$ we have $\gamma_{n_k}(\theta_k) \in \D(z, \frac{\delta}3)$ and $\gamma_{m_k}(\theta_k) \notin \D(z, \frac{2\delta}3)$, so
\[
\dist_\varsigma(\gamma_{n_k}(\theta_k),\gamma_{m_k}(\theta_k)) \ge \frac{c\delta}{3},
\]
which contradicts the uniform Cauchy condition with respect to the metric $d\varsigma$. Therefore, the sequence $\gamma_n$ satisfies the uniform Cauchy condition with respect to the spherical metric, which implies that it has a limit $\gamma$, which is equal to the boundary of $U$ in $\clC$, providing its local connectivity.


\section{Local connectivity of Julia sets: Proof of Theorem C}\label{sec:proofC}

Throughout this section, we consider the transcendental meromorphic map 
\[
f(z)=z-\tan z,
\]
which is the Newton map of the entire transcendental function $g(z)=\sin z$. We will prove that the Julia set $J(f)$ is locally connected, establishing Theorem~C. See Figure~\ref{dynplanesine}. Throughout the section, we will understand the word `boundary' as the boundary in $\clC$.

A useful tool to prove local connectivity of a compact connected set in $\clC$ (like the Julia set of a rational or transcendental map) is Whyburn's Theorem, which reads as follows.

\begin{thm}[{\cite[Theorem~4.4, p.~113]{whyburn}}]\label{thm:whyburn}
A compact connected set $J\subset \chat$ is locally connected if
and only if it satisfies the following properties.
\begin{enumerate}[$($a$)$]
\item The boundary of every connected component of $\chat \setminus J$ is locally connected.
\item For every $\varepsilon>0$, only a finite number of connected components of $\chat \setminus J$
have spherical diameter greater than $\varepsilon$.
\end{enumerate}
\end{thm}

To check the conditions of Whyburn's Theorem for $J = J(f)$, we first recall some properties of the map $f$. By \cite{bfjk}, $J$ is connected as the Julia set of  Newton's method for a transcendental entire map. Note that
\begin{equation}\label{eq:f+pi}
f(z+\pi) = f(z) + \pi, \qquad z \in \C,
\end{equation}
which implies a `translation invariance' of the dynamical plane of $f$ (see Figure~\ref{dynplanesine}). The properties listed below are proved in \cite[Example~7.2]{bfjkaccesses} and \cite[Proposition~4.1]{barfagjarkar2020}.

\begin{lem} \label{lem:Theorem B}
The following statements hold.
\begin{enumerate}[$($a$)$]
\item\label{item:U_k} $f$ has infinitely many simply connected immediate basins of attraction $U_k$, $k\in\Z$, of superattracting fixed points 
$$c_k=k\pi, \quad k \in\Z,$$
such that $U_k = U_0 + k\pi$ and $\deg_{c_k} f = \deg f|_{U_k} = 3$. The points $c_k$, $k \in \Z$, are the only critical points of~$f$. They are located in the vertical lines
\[
\ell_k(t)= k \pi + it, \qquad t \in \R, \; k \in\Z, 
\]
which are invariant and contained in $U_k$, respectively. Moreover, $f$ has no finite asymptotic values. 

\item The poles of $f$, 
\[
p_k = \frac{\pi}{2} + k \pi,\ k\in \mathbb Z,
\]
are simple. They are located in the vertical lines 
$$
r_k(t)=\frac \pi 2 + k \pi + it, \qquad t \in \R, \; k \in\Z, 
$$
which are contained in $J(f)$.

\item\label{item:asymptotic} The asymptotics of $f$ for $\Im(z) \to \pm\infty$ is given by 
\[
f(z) = \begin{cases}
z-i + \mathcal O(e^{-2\Im(z)})& \text{as} \;\; \Im (z) \to +\infty \\
z+i + \mathcal O(e^{2\Im(z)})& \text{as} \;\; \Im (z) \to -\infty
\end{cases}.
\]

\item We have
\[
J(f) \cap \C = \bigcup_{k\in\Z} J_k,
\]
where $J_k$ is the connected component of $J(f) \cap \C$ containing the line $r_k$. For every $k\in\Z$, the basin $U_k$ has exactly two accesses to infinity,
and $\bd U_k$ contains exactly two poles of $f$, i.e.~$p_{k-1}, p_k$, which are accessible from $U_k$.
\item All Fatou components of $f$ are preperiodic and eventually mapped by iterates of $f$ into $U_k$ for some $k\in\Z$. In particular, $f$ has no wandering domains. 
\end{enumerate}
\end{lem}

See Figure~\ref{fig:R}.
\begin{figure}[ht!]
\includegraphics[width=0.7\textwidth]{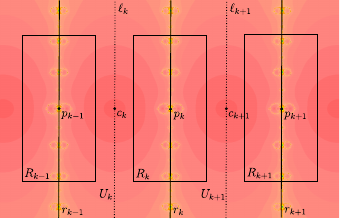}
\label{fig:R}
\caption{\small The dynamical plane of the map $f(z) = z - \tan z$.}
\end{figure}

The following proposition establishes the first condition in Whyburn's criterium (Theorem~\ref{thm:whyburn}).
\begin{prop} \label{prop:first_crit}
Every Fatou component of $f$ has locally connected boundary. 
\end{prop}

\begin{proof} 
First, we show that the boundaries of the attracting basins $U_k$ are locally connected. To this end, we check that they fulfil the assumptions of Theorem~A.
Recall that the basins $U_k$ are simply connected, since $J(f)$ is connected. By Lemma~\ref{lem:Theorem B}, we have $\deg f|_{U_k} = 3$ and 
\begin{equation}\label{eq:P(f)}
\Crit(f) = \PP(f) = \{c_k\}_{k\in\Z} = \{k\pi\}_{k\in\Z}.
\end{equation}
This immediately implies the assumptions (a)--(b) of Theorem~A. To check the assumption~(c), note that by Lemma~\ref{lem:Theorem B}\ref{item:asymptotic}, for $M>0$ large enough, the map $f$ on the half planes 
\begin{equation}\label{eq:P_pm}
P_+=\{z \in \C: \Im (z)>M\}, \qquad P_-=\{z \in \C: \Im(z)<-M\}
\end{equation}
is arbitrarily close to $z \mapsto z\mp i$. In particular, this implies $\overline{P_\pm} \subset f(P_\pm)$ and $f(P_+) \cap f(P_-) = \emptyset$. We show that $P_\pm$ are repelling petals of $f$ at infinity. In view of Lemma~\ref{lem:Theorem B}\ref{item:asymptotic}, the univalency of $f$ on $P_\pm$ follows easily from Rouch\'e's Theorem, while the remaining conditions of Definition~\ref{defn:petal_at_inf} with $g(t)\equiv 1$ are satisfied by Proposition~\ref{prop:example_petal_inf}, where we take $r = M$, $\delta = \frac{\pi}{2}$, $d = 1$, $a= \mp i$ and $j = 1$. Hence, $P_\pm$ are repelling petals of $f$ at infinity. Note that Lemma~\ref{lem:Theorem B} asserts that the basins $U_k$ are located between the vertical lines $r_{k-1}$ and $r_k$, both contained in the Julia set of $f$. Hence, outside a compact set, $U_k$ is contained in the two repelling petals $P_\pm$ of $f$ at infinity, which shows that $f$ and $U_k$ satisfy the assumptions of Theorem~A. Hence, the boundaries of $U_k$ are locally connected.

By Lemma~\ref{lem:Theorem B}(e), all Fatou components of $f$ are successive  preimages of the basins $U_k$ which have locally connected boundaries, so their boundaries are also locally connected. 
\end{proof}

It remains to show the second condition in Whyburn's criterium. Note that in contrast to the rational or even entire (hyperbolic) cases, the meromorphic setting presents some additional difficulties, since for any given $n$, components of preperiod $n$ can accumulate at poles and prepoles. 

The idea of the proof is as follows. We show that the sum of the squares of the spherical diameters of all Fatou components is finite, which immediately implies that only a finite number of them can have diameter larger than any given $\varepsilon$. To this end, we prove that the spherical distortion of branches of $f^{-n}$ on Fatou components $U \subset f^{-1}(U_k) \setminus U_k$, $k\in\Z$, is uniformly bounded. From this, it follows that the ratio between the square of the spherical diameter  and the spherical area is approximately the same for the component $U$ and its all inverse images by branches of $f^{-n}$, $n > 0$. Since the sum of spherical areas over all components is finite (smaller than the spherical area of the whole sphere), the same holds for the sum of the squares of spherical diameters. 

Now we proceed to present the proof in detail. The following lemma is straightforward to check.

\begin{lem}\label{lem:sph-eucl}
For every $r_0 > 0$ there exist $c_1, c_2 > 0$ such that
\[
\DD_{sph}\Big(z, \frac{c_1 r}{|z|^2+1}\Big) \subset \D(z,r) \subset \DD_{sph}\Big(z, \frac{c_2 r}{|z|^2+1}\Big)
\]
for every $z \in \C$ and every $0 < r < r_0$.
\end{lem}

For $n \ge 0$ define
\begin{align*}
\FF_n=\{U : \ &U \text{ is a Fatou component of $f$,}\\
&\text{and $n$ is minimal such that } f^n(U) \subset U_k \text{ for some $k\in\Z$}\}.
\end{align*}
In particular, $\FF_0 = \{U_k\}_{k\in \Z}$. In the following lemma we describe the Fatou components from $\FF_1$. Let
\[
R_k = R_k(\delta, R) = \{z\in \C: \Re(z) \in [k\pi + \delta, (k+1)\pi - \delta], \; \Im(z) \in [-R, R]\}, \quad k \in \Z
\]
for $\delta, R > 0$. See Figure~\ref{fig:R}.

\begin{lem}\label{lem:F1}
We have
\[
\FF_1 = \{ U_{k,l} : k\in \Z, \: l \in \Z \setminus \{k, k+1\}\},
\]
where 
\[
U_{k,l} \subset R_k, \qquad 
U_{k,l} = U_{0,l-k} + k \pi, \qquad
f(U_{k,l}) = U_l
\]
and $R_k = R_k(\delta, R)$ for some $\delta, R > 0$. Moreover, $f$ is univalent on $U_{k,l}$.
\end{lem}
\begin{proof} Fix a large $\tilde R > 0$ such that $\{z \in \C: |\Im(z)| > \tilde R\} \subset P_+ \cup P_-$ for the repelling petals $P_\pm$ from \eqref{eq:P_pm}. Take an arbitrary point $z_0 \in U_0$ with $|\Im(z_0)| > \tilde R$ (note that such points exist since $\ell_0 \subset U_0$) and let $z_l = z_0 + l\pi$ for $l \in \Z$. By \eqref{eq:f+pi} and Lemma~\ref{lem:Theorem B}(a), we have $z_l \in U_l$ and $f^{-1}(z_l) \subset P_+ \cup P_- \cup \bigcup_{k\in \Z}\D(p_k, \varepsilon)$ for a small fixed $\varepsilon > 0$, provided $\tilde R$ was chosen sufficiently large. Moreover, since $f$ is univalent on $P_\pm$ with $f(z) \sim z \mp i$, the set $f^{-1}(z_l) \cap (P_+ \cup P_-)$ consists of exactly one point $z'_l \in P_\pm$ close to $z_l \pm i$. Similarly, as the pole $p_0$ is simple, the map $f$ is univalent on $\D(p_0, \varepsilon)$ and $f^{-1}(z_l) \cap \D(p_0, \varepsilon)$ consists of exactly one point $z_{0,l}$. Consequently, by 
\eqref{eq:f+pi}, $f$ is univalent on $\D(p_k, \varepsilon)$ and $f^{-1}(z_l) \cap \D(p_k, \varepsilon)$ consists of exactly one point $z_{k,l}$, where
\[
z_{k,l} \in  \D(p_k, \varepsilon), \qquad z_{k,l} = z_{0,l-k} + k\pi, \qquad k, l \in \Z.
\]
Hence, 
\begin{equation}\label{eq:f^-1zl}
f^{-1}(z_l) = \{z'_l\} \cup \{z_{k,l} : k \in \Z\},
\end{equation}
where all the points $z_{k,l}$, $k \in \Z \setminus \{l-1,l\}$ are outside $P_+ \cup P_- \cup S_l$ for
\[
S_l =  \Big\{z\in \C: \Re(z) \in \Big[\frac{\pi}{2} + (l-1)\pi, \frac{\pi}{2} + l\pi\Big]\Big\}.
\]
This implies $f^{-1}(z_l) \cap (P_+ \cup P_- \cup S_l) \subset \{z'_l, z_{l-1,l}, z_{l,l}\}$. On the other hand, since $\deg f|_{U_l} = 3$, the point $z_l$ has exactly three preimages under $f$ in $U_l$. As $U_l \subset S_l$, we conclude 
\begin{equation}\label{eq:f^-1Ul}
f^{-1}(z_l) \cap (P_+ \cup P_- \cup S_l) = f^{-1}(z_l) \cap U_l = \{z'_l, z_{l-1,l}, z_{l,l}\}.
\end{equation}

Take $U \in \FF_1$. Then $U \cap \bigcup_{k\in \Z} U_k = \emptyset$ and $f(U) \subset U_l$ for some $l \in \Z$. Suppose there exists a point $w \in U$ with $|\Im(w)| > \tilde R + 2$. Then $w \in P_+ \cup P_- \setminus U_l$ and $f(w) \in U_l$ with $|\Im(f(w))| > \tilde R$, contradicting \eqref{eq:f^-1Ul} for $z_0 = f(w) -l\pi$. Hence, $U \subset \{z \in \C: |\Im(z)| \le R\}$ for $R = \tilde R + 2$. Moreover, since $\ell_k \subset U_k$ and $U_k = U_0 + k\pi$ for $k \in \Z$ by Lemma~\ref{lem:Theorem B}(a), the set $\{z \in \C: \Re(z) \in (k\pi - \delta, k\pi + \delta), \; |\Im(z)| \le R\}$ is contained in $U_k$ for a sufficiently small $\delta > 0$ independent of $k \in \Z$. Consequently, $U \subset \bigcup_{k\in \Z} R_k$ for $R_k = R_k(\delta,R)$. Since the sets $R_k$ are disjoint and compact, and $U$ is connected, in fact $U \subset R_k$ for some $k \in \Z$. As $U_l$ is simply connected and $R_k \cap \Crit(f) = \emptyset$ by \eqref{eq:P(f)}, the component $U$ is the image of $U_l$ under a well-defined branch $g$ of $f^{-1}$. 
Note that by \eqref{eq:f^-1zl} and \eqref{eq:f^-1Ul}, we have $g(z_l) = z_{k',l}\in U$ for some $k' \in \Z \setminus \{l-1, l\}$. As $z_{k',l} \in R_{k'}$ and the sets $R_{k'}$ are disjoint, in fact $k'=k$, so $z_{k,l} \in U$ and $k \neq l-1, l$. 

We have proved that for every component $U \in \FF_1$ there exist $k, l\in \Z$ with $k \neq l-1, l$, such that $z_{k,l} \in U \subset R_k$ and $f$ maps $U$ univalently onto $U_l$. On the other hand, \eqref{eq:f^-1zl} and \eqref{eq:f^-1Ul} imply that for every $k, l\in \Z$ with $k \neq l-1, l$, there exists a component from $\FF_1$ containing $z_{k,l}$. This together with \eqref{eq:f+pi} and Lemma~\ref{lem:Theorem B}(a) ends the proof.
\end{proof}

Now we can show the second condition in Whyburn's criterium. 

\begin{prop}\label{prop:second_crit}
For every $\varepsilon > 0$ there are only a finite number of Fatou components $U$ of $f$ with $\diam_{sph} U \ge \varepsilon$.
\end{prop}
\begin{proof}
Recall that by Lemma~\ref{lem:Theorem B}(e), all Fatou components of $f$ are elements of $\bigcup_{n=0}^\infty \FF_n$. Since $U_k \subset S_k \subset \C \setminus \D(0,\frac{\pi}{2} + (|k| - 1)\pi)$ for $k \in \Z\setminus\{0\}$, the spherical diameter of $U_k$ tends to zero as $|k|\to\infty$, so we only need to consider components in $\bigcup_{n=1}^\infty \FF_n$. 

By Lemma~\ref{lem:F1}, for every $U \in \FF_1$ we have $U = U_{k,l}$ for some $k\in \Z$, $l \in \Z \setminus \{k, k+1\}$, such that 
\begin{equation}\label{eq:F1}
U_{k,l} = U_{0,l-k} + k\pi \subset R_k
\end{equation}
and $f$ maps $U_{k,l}$ univalently onto $U_l$. Since $U_l = U_0 + l\pi \subset S_l$, $\D(l\pi, r_0) \subset U_l$ for some $r_0 > 0$ independent of $l$, and $p_0$ is a simple pole of $f$, there exists $c_1 > 0$ such that 
\[
\diam U_{0,l} \le \frac{c_1}{|l| + 1}, \qquad \area U_{0,l} \ge \frac{c_1}{l^4 + 1} 
\]
for $l \in \Z \setminus \{0, 1\}$. Consequently, by \eqref{eq:F1} and since $\sigma_{sph}|_{R_k} \asymp \frac{1}{k^2+1}$, we have
\begin{equation}\label{eq:diamUk,l}
\diam U_{k, l} \le \frac{c_1}{|l-k| + 1}, \quad \area_{sph} U_{k, l} \ge \frac{c_2}{((l-k)^4 + 1)(k^4 + 1)}
\end{equation}
for $k \in \Z$, $l \in \Z \setminus \{k, k+1\}$ and some constant $c_2 > 0$.

By \eqref{eq:P(f)}, all branches of $f^{-n}$, $n \ge 1$, are well-defined on $\{z \in \C: \Re(z) \in (k\pi, (k+1)\pi)\} \supset R_k$ for $k \in \Z$. Hence, for $n \ge 1$,
\[
\FF_n = \{g_{k,n}(U_{k, l}): k \in \Z, \: l \in \Z \setminus \{k, k+1\}, \: g_{k,n} \in \GG_{k,n}\},
\]
where $\GG_{k,n}$ is the family of all branches of $f^{-(n-1)}$ on $R_k$ (note that we include the case $n = 1$, for which $g_{k,1} = \Id$). We claim that all the branches $g_{k,n} \in \GG_{k,n}$ have spherical distortion bounded uniformly with respect to $k,n$. Indeed, for given $r > 0$, any two points $z, w \in R_k$ can be joined by $N = \lceil 2(2R+\pi)/r\rceil$ Euclidean disks $\D(z_1, r), \ldots, \D(z_N, r)$, such that $z_1, \ldots, z_N \in R_k$, $z_1 =z$, $z_N = w$ and 
\begin{equation}\label{eq:Djdisjoint}
\D(z_j, r) \cap \D(z_{j+1}, r) \neq \emptyset \quad\text{for } \;j = 1, \dots, N-1.
\end{equation}
Define
\[
r = \frac{c_1 \delta}{2c_2}, \qquad r_j = \frac{c_1\delta}{|z_j|^2+1}
\]
for the constants $c_1, c_2$ from Lemma~\ref{lem:sph-eucl}, and the constant $\delta$ from Lemma \ref{lem:F1}. Then, 
\[
\D(z_j,r)  \subset\DD_{sph}\Big(z_j, \frac{r_j}{2}\Big), \qquad \DD_{sph}(z_j, r_j) \subset \D(z_j,\delta)
\]
by Lemma~\ref{lem:sph-eucl}, so by \eqref{eq:Djdisjoint} and \eqref{eq:P(f)},
\[
\DD_{sph}\Big(z_j,\frac{r_j}{2}\Big) \cap \DD_{sph}\Big(z_{j+1}, \frac{r_j}{2}\Big) \neq \emptyset, \qquad \DD_{sph}(z_j, r_j) \cap \PP(f) = \emptyset,
\]
and hence using repeatedly Theorem~\ref{thm:sph-Koebe} for the chain of spherical disks with $\lambda=1/2$, we conclude that 
\begin{equation}\label{eq:distortion-g}
\frac{|g'_{k,n}(z)|_{sph}}{|g'_{k,n}(w)|_{sph}} \le c_3
\end{equation}
for some $c_3 > 0$ independent of $k,n,g_{k,n}, z,w$. 
Let
\[
M_{g_{k,n}} = \max_{R_k} |g'_{k,n}|_{sph}.
\]
By \eqref{eq:diamUk,l}, \eqref{eq:distortion-g} and since $\sigma_{sph}|_{R_k} \asymp \frac{1}{k^2+1}$,
\[
\diam_{sph}g_{k,n}(U_{k,l}) \le \frac{c_4 M_{g_{k,n}} \diam U_{k,l}}{k^2+1} \le \frac{c_1 c_4 M_{g_{k,n}}}{(|l-k|+1)(k^2+1)}
\]
and
\begin{align*}
\area_{sph}g_{k,n}(U_{k,l}) 
&\ge \big(\min_{R_k} |g'_{k,n}|_{sph}^2\big) \area_{sph}U_{k,l}\\ 
&\geq \frac{M_{g_{k,n}}^2}{c_3^2}\area_{sph}U_{k,l} \ge \frac{c_2}{c_3^2}\frac{M_{g_{k,n}}^2}{((l-k)^4 + 1)(k^4 + 1)},
\end{align*}
for a constant $c_4 > 0$. Hence, for some constants $c_5, c_6, c_7 > 0$,
\begin{align*}
\sum_{U \in \bigcup_{n= 1}^\infty \FF_n} (\diam_{sph} U)^2&=
\sum_{k\in\Z}\sum_{l \in \Z \setminus \{k,k+1\}} \sum_{\substack{n\ge 1,\\g_{k,n} \in \GG_{k,n}}}(\diam_{sph} g_{k,n}(U_{k,l}))^2\\
&\leq c_5\sum_{k\in\Z}\sum_{l \in \Z \setminus \{k,k+1\}}\sum_{n,g_{k,n}} \frac{M_{g_{k,n}}^2}{((l-k)^2+1)(k^4+1)}\\
&= c_5 \sum_{l \in \Z \setminus \{0,1\}} \frac{1}{l^2+1} \sum_{k \in \Z} \sum_{n,g_{k,n}} \frac{M_{g_{k,n}}^2}{k^4+1}\leq c_6\sum_{l \in \Z \setminus \{0,1\}} \frac{1}{l^4+1} \sum_{k \in \Z}\sum_{n,g_{k,n}}\frac{M_{g_{k,n}}^2}{k^4+1}\\ 
&= c_6\sum_{k\in\Z}\sum_{l \in \Z \setminus \{k,k+1\}} \sum_{n,g_{k,n}}\frac{M_{g_{k,n}}^2}{((l-k)^4+1)(k^4+1)}\\
&\leq c_7 \sum_{k\in\Z}\sum_{l \in \Z \setminus \{k,k+1\}} \sum_{n,g_{k,n}}\area_{sph}g_{k,n}(U_{k,l})
\le c_6 \sum_{U \in \bigcup_{n= 1}^\infty \FF_n} \area_{sph} U\\
&\le c_7 \area_{sph} \clC < \infty,
\end{align*}
as $U \in \bigcup_{n= 1}^\infty \FF_n$ are pairwise disjoint. Concluding, we have showed that the series $$\sum_{U \in \bigcup_{n= 1}^\infty \FF_n} (\diam_{sph} U)^2$$ is convergent, which immediately implies that for every $\varepsilon > 0$ there can be only a finite number of $U\in \bigcup_{n= 1}^\infty \FF_n$ with $\diam_{sph} U \ge \varepsilon$, proving the proposition.
\end{proof}

Propositions~\ref{prop:first_crit} and~\ref{prop:second_crit} complete the proof of Theorem~C.

\subsection*{Acknowledgements} The authors thank the Centro di Ricerca Matematica Ennio De Giorgi (Pisa) for its hospitality.

\subsection*{Conflict of interest} On behalf of all authors, the corresponding author states that there is no conflict of interest.

\subsection*{Data availability} Data sharing not applicable to this article as no datasets were generated or analysed during the current study.

\bibliographystyle{acm}
\bibliography{loc_con}

\end{document}